\documentclass[11pt]{article}

\usepackage{amsmath, amssymb, amscd, amsthm, amsfonts}
\usepackage{graphicx}
\usepackage{hyperref}
\usepackage{indentfirst}
\usepackage{mathtools}
\usepackage{enumerate}
\usepackage{geometry}
\usepackage{gensymb}

\title{\textsc{\Large{\textbf{Existence for Stable Rotating Star-Planet Systems}}}}

\author{Hangsheng Chen\thanks{Department of Mathematics, Statistics, and Computer Science, M/C 249, University of Illinois at Chicago, 851 S. Morgan Street, Chicago, IL 60607, USA. Email: hchen261@uic.edu}}

\date{ }

\newtheorem{theorem}{Theorem}
\newtheorem{conjecture}[theorem]{Conjecture}

\newtheorem  {corollary} [theorem]{Corollary}

\newtheorem  {lemma} [theorem]{Lemma}
\newtheorem  {proposition}[theorem]{Proposition}

\theoremstyle{definition}
\newtheorem{definition}[theorem]{Definition}

\newtheorem  {remark} [theorem]{Remark}

\numberwithin{theorem}{section}

\makeatletter
\newcommand{\mylabel}[2]{#2\def\@currentlabel{#2}\label{#1}}
\makeatother

\begin{document}

\maketitle

\begin{abstract}\label{abstract}	
	
	This paper investigates the existence and properties of stable, uniformly rotating star-planet systems, i.e., mass ratio is sufficiently small. It is modeled by the Euler-Poisson equations. Following the framework established by McCann for binary stars \cite{McC06}, we adopt a variational approach, and prove the existence of local energy minimizers with respect to the Wasserstein $L^\infty$ metric, under the assumed equation of state $P(\rho)=K\rho^\gamma$ and under the condition that the mass ratio $m$ is sufficiently small, corresponding to a star–planet system. Such minimizers correspond to solutions of the Euler–Poisson system. We consider two cases. For $\gamma > 2$, we not only prove existence but also show, via scaling arguments, that the radii (to be precise, the bounds of the supports of the minimizers) tend to zero. For $\frac{3}{2} < \gamma \leq 2$, we estimate an upper bound for the (potential) expansion rates of the radii, and it turns out that the existence result remains valid in this case as well. Finally, we provide estimates for the distances between different connected components of supports of minimizers and propose a conjecture regarding the number of connected components.

	{\footnotesize\textbf{Key words:} Gaseous Star-Planet Systems, Euler-Poisson Equations, Calculus of Variations, Wasserstein $L^\infty$ Topology, Scaling Method, Bootstrap Method}
\end{abstract}

 \tableofcontents
 
\section{Introduction}\label{section-introduction}

\setlength{\parindent}{1.5em}

In astrophysical fluid dynamics, gas stars and gas planets can be modeled as isolated self-gravitating fluid masses. A fundamental hydrodynamic description is provided by the Euler–Poisson system:
\begin{equation} \label{EP}
	\begin{split}
		\partial_{t}\rho + \nabla \cdot \left( {\rho v} \right) &= 0,    \\
		\rho\partial_{t}v + \rho\left( {v \cdot \nabla} \right)v + \nabla P(\rho) &= \rho\nabla V,	\\
		\Delta V &= - 4\pi\rho.
	\end{split}	\tag{EP}
\end{equation}
Here, $\rho(x,t)\geq0$ denotes the fluid density at position $x\in \mathbb{R}^3$ and time $t\geq 0$, $v(x,t)\in \mathbb{R}^3$ is the velocity, and $V(x,t)\in \mathbb{R}$ represents the gravitational potential. The pressure $P(\rho)$ depends solely on density. The existence of solutions under different settings has been established via variational methods (e.g., \cite{AB71M, AB71, Auc91, CF80, CL94, Li91, McC06}) and perturbative approaches (e.g., \cite{JM17, JM19, Hei94, SW17, SW19, AKV23, Lic33}). 

In this work, we adopt the polytropic equation of state $P(\rho)=K\rho^\gamma$, unless otherwise stated. We mainly consider two stellar objects, under the simplifying assumption that individual rotations are neglected and only orbital revolution is considered. We still refer to such configurations as rotating systems. The construction of rotating binary stars dates back to Lichtenstein \cite{Lic33}. As discussed by Jang and Seok \cite{JS22}, in the $N$-body uniformly rotating case, the Euler–Poisson equations reduce to a single equation:
\begin{equation} \label{EP'}
	-\omega ^2\tilde{\rho}\left( x \right) P\left( x \right) +\nabla P\left( \tilde{\rho}\left( x \right) \right) -\tilde{\rho}\left( x \right) \left( \nabla V_{\tilde{\rho}}\left( x \right) \right) =0
	\tag{EP'}
\end{equation}
where $\omega>0$ is the angular velocity, $\tilde{\rho}$ is a compactly supported density function, and $V_{\tilde{\rho}}(x) = {\int_{\mathbb{R}^3}{\frac{\tilde{\rho}(y)}{\left| {y - x} \right|}\,dy}}$. As shown in \cite[Section 2]{Che26G1}, if $\tilde{\rho}$ satisfies the reduced Euler–Poisson equation (\ref{EP'}), then $\rho\left( {t,y} \right) = \tilde{\rho}\left(R_{- \omega t}y \right)$ and $v\left({t,y} \right) = \omega\left( {- y_{2},y_{1},0} \right)^{T}$ satisfy the original system (\ref{EP}), where $R_{- \omega t}y$ is the rotation map given in Definition \ref{notations}. In this case, it is easy to see that $\Delta v = 0$, so $(\rho,v)$ also satisfies the Navier-Stokes-Poisson system. Moreover, since $ \nabla \cdot v = 0 $, no compression occurs in this uniform rotation scenario. Nevertheless, based on physical intuition, we still regard this as a compressible model, corresponding to gaseous stars and planets. 

Examples of two-body systems include binary star systems and star-planet systems. In \cite{McC06}, McCann constructed binary star solutions to (\ref{EP'}) with separated supports, (almost) determined via a Kepler problem, by formulating a minimization problem with a prescribed mass ratio. In that work, the angular momentum $J$ is assumed to be large. It leads to a wide separation between the two objects and a correspondingly small angular velocity $\omega=\frac{J}{I(\tilde{\rho})}$, because the moment of inertia $I(\tilde{\rho})$ grows faster than $J$ when $J$ grows. However, the existence of stable rotating solutions for arbitrary angular momentum under small mass ratio (the physical regime for star–planet systems) has not been addressed. This paper hence extends McCann's framework for binary-star systems to this setting and establishes existence for star–planet systems. In addition, it studies some asymptotic properties as the mass ratio tends to zero. Note that in both McCann’s work \cite{McC06} and this paper, the angular velocity $\omega$ turns out to be small, since solutions or minimizers may not exist if $\omega$ is too large, as indicated by Li \cite{Li91}.

A notable feature of McCann’s framework is that each star admits a variational characterization as a Hamiltonian or energy minimizer under certain constraints, while also being interpretable as a perturbation of simpler objects—such as non-rotating Lane–Emden stars and point-mass relative equilibria (see subsection \ref{subsection5.3-Estimate for the supports of densities} and Section \ref{section6-Existence for Star-Planet Systems} of this paper or \cite[Section 3.2]{JS22}). This characterization is crucial for showing that the support of an energy-constrained minimizer lies in the interior of balls of certain radii. Consequently, the minimizer $\rho$ is in fact a local minimizer in the Wasserstein $L^\infty$ distance sense. This ensures that $\nabla P(\rho)$ is well-defined throughout $\mathbb{R}^3$, and that (\ref{EP'}) holds in the whole space $\mathbb{R}^3$ (see Theorem \ref{Properties of LEM}, \cite[Theorem 2.1]{McC06}, \cite[Section 2]{Che26G1}).

Beyond existence, many questions remain for two-body systems. For instance, conjectures concerning the maximal possible number of connected components of the support of solution are discussed in Section \ref{section7-The Maximum Number of Connected Components of Minimizers} of this paper. Jang and Seok \cite{JS22} study the asymptotic profiles, uniqueness, and orbital stability of McCann’s uniformly rotating binary stars. They also establish existence and stability for rotating binary galaxies modeled by the Vlasov–Poisson system \cite{JS22}, a work that extends the discussion of binary star solutions by adapting Rein's reduction method \cite{Rei03, Rei07}.

This paper presents the core results extracted from the author's Master's thesis \cite{CVD24}. It investigates the existence and qualitative properties of stable, uniformly rotating star–planet systems governed by the Euler–Poisson equations, as established in Theorems \ref{themA} and \ref{themA'}.

Although this paper builds upon the preparatory work in two companion papers \cite{Che26G1, Che26R2} originating from the same thesis \cite{CVD24}, it is largely self-contained. Relevant auxiliary results can also be found in other existing literature cited at the corresponding locations herein, with the author's papers \cite{Che26G1, Che26R2} providing refined and more complete proofs. 

The paper is structured as follows. Section \ref{section2-construction and results} introduces the variational formulation and states the main results. Section \ref{section3-Preliminary Results} collects preliminary results about two-body systems and non-rotating single star cases from \cite{McC06, AB71, LY87, Che26G1, Che26R2}. Sections \ref{section4-Existence of Constrained Minimizers}, \ref{section5-Bound of the Supports of Density Functions} and \ref{section6-Existence for Star-Planet Systems} are the core of this paper, where we will prove existence of star-planet systems. Sections \ref{section7-The Maximum Number of Connected Components of Minimizers} studies distances between (possible) different connected components of planets or stars, leading to the conjecture that minimizer may consist of at most two connected components, corresponding to exactly one star and exactly one planet respectively.

\section{Variational formulation and Statement of Main Results}\label{section2-construction and results}

\subsection{Variational formulation}\label{subsection2.1-variational formulation}

We introduce some notations and variational formulation here, which are based on McCann's construction of binary star solutions, see \cite{McC06} and \cite{Che26G1}. 

\begin{definition}[Notations] \label{notations}
	\rm
	We give the following definitions:
	\begin{enumerate}[(i)]
		\item The projection operator of $x$ to the $x_1 x_2$ plane: $P_{12}(x)=P_{12}\left(x_1, x_2, x_3\right):=\left(x_1, x_2, 0\right)$.
		\item A bilinear form $\langle \cdot , \cdot \rangle_2: \mathbb{R}^3 \times \mathbb{R}^3 \rightarrow \mathbb{R}: \forall x, y \in \mathbb{R}^3,\langle x, y\rangle_2=P_{12}(x) \cdot P_{12}(y)= x_1 y_1+x_2 y_2$.
		\item $r(x):=\left(\langle x, x\rangle_2\right)^{\frac{1}{2}}=\sqrt{x_1^2+x_2^2}$.
		\item Rotation map:
		$$
		R_\theta:=\left(\begin{array}{ccc}
			\cos \theta & -\sin \theta & 0 \\
			\sin \theta & \cos \theta & 0 \\
			0 & 0 & 1
		\end{array}\right)
		$$
		\item $W^{\infty}$ metric: Wasserstein $L^{\infty}$ metric. (See Definition \ref{Wasser.})
	\end{enumerate}
\end{definition}

The state of a fluid may be represented by its mass density $\rho(x) \geq 0$ and velocity vector field $v(x)$. The fluid interacts with itself through Newtonian gravity hence we need to consider gravitational interaction energy, which will be given later. Moreover, to define internal energy, we need to first consider the equation of state for pressure $P(\rho)$. In this paper, we take the polytropic law, i.e., 

\begin{equation}\label{polytropic law}
	P(\rho)=K\rho^\gamma
\end{equation}

We also define $A(s)$ as the following:
\begin{equation}\label{A}
	A(s):=s \int_0^s P(\tau) \tau^{-2} d \tau=\frac{K}{\gamma-1} s^\gamma
\end{equation}

We first give ``admissible classes'' for $\rho$ and $v$ as the following:
\begin{align}
	R\left(\mathbb{R}^3\right) & :=\left\{\left.\rho \in L^{\frac{4}{3}}\left(\mathbb{R}^3\right) \right\rvert\, \rho \geq 0, \int_{\mathbb{R}^3} \rho \,dx=1\right\} \\
	V\left(\mathbb{R}^3\right) & :=\left\{v: \mathbb{R}^3 \rightarrow \mathbb{R}^3 \mid v \text { is measurable.}\right\}
\end{align}

Then given $\rho$ and $v$ in such sets, the energy $E(\rho, v)$ consists of three terms:

\begin{equation}\label{energy}
	E(\rho, v):=U(\rho)-\frac{G(\rho, \rho)}{2}+T(\rho, v) 
\end{equation}

\begin{equation} \label{U}
	U(\rho):=\int_{\mathbb{R}^3} A(\rho(x)) \,dx
\end{equation}

\begin{equation} \label{G}
	G(\sigma, \rho):=\int_{\mathbb{R}^3} V_\sigma \rho \,dx=\iint_{\mathbb{R}^3\times \mathbb{R}^3}  \frac{\rho(x) \sigma(y)}{|x-y|} \,dy\,dx
\end{equation}

\begin{equation} \label{T}
	T(\rho, v):=\frac{1}{2} \int_{\mathbb{R}^3}|v|^2 \rho\,dx
\end{equation}

Here $A(\rho)$ is a convex function given in (\ref{A}), and $U(\rho)$ is the \textit{internal energy}.

We can choose units so that the total mass of fluid is one and the gravitational constant $G=1$, then $V_\rho$ represents the \textit{gravitational potential} of the mass density $\rho(x)$
\begin{equation}\label{V}
	V_\rho(x):=\int_{\mathbb{R}^3} \frac{\rho(y)}{|y-x|} \,dy
\end{equation}

Hence $G(\rho, \rho)$ is the \textit{gravitational potential energy} (also called \textit{gravitational interaction energy}), and $T(\rho, v)$ is the \textit{kinetic energy}.

\begin{remark}\label{finite gravitational interaction energy}
	Since $\rho \in L^{\frac{4}{3}}\left( \mathbb{R}^{3} \right) \cap L^{1}\left( \mathbb{R}^{3} \right)$, we have $G(\rho,\rho)<\infty$, see for example \cite[Section 2]{Che26G1}. Hence $E(\rho,v)$ is well-defined. 
\end{remark}


We can choose a frame of reference in which the \textit{center of mass}
\begin{equation}\label{centerofmass}
	\bar{x}(\rho):=\frac{\int_{\mathbb{R}^3} x \rho(x) \,dx}{\int_{\mathbb{R}^3} \rho(x) \,dx}
\end{equation}
is at rest. We are interested in finding minimum energy configurations subject to constraints of fixed mass ratio and fixed angular momentum $\boldsymbol{J}$ with respect to the center of mass $\bar{x}(\rho)$. The fluid \textit{angular momentum} $\boldsymbol{J}(\rho, v)$ is given by:
\begin{equation} \label{angular momentum}
	\boldsymbol{J}(\rho, v):=\int_{\mathbb{R}^3}(x-\bar{x}(\rho)) \times v \rho(x) \,dx
\end{equation}

We denote by $J_z$ the z-component of $\boldsymbol{J}$, that is, $J_z(\rho, v):=\hat{e}_z \cdot \boldsymbol{J}(\rho, v)$, where $\hat{e}_z=(0,0,1)^T$. For simplicity of notation, we will sometimes use  $J$  to represent $J_z$ and call $J$ the angular momentum of the system when no confusion arises.

Since the z-component of the angular momentum is specified, the moment of inertia $I(\rho)$ of $\rho$ in the direction of $\hat{e}_z$ will be relevant. That is, we define the \textit{moment of inertia} of $\rho$ in the direction of $\hat{e}_z$, denoted by $I(\rho)$, as follows:
\begin{equation} \label{Moment of Inertia}
	I(\rho):=\int_{\mathbb{R}^3} \rho r^2(x-\bar{x}(\rho))\,dx=\int_{\mathbb{R}^3} \rho(x)\left( (x_1-\bar{x}(\rho)_1)^2+(x_2-\bar{x}(\rho)_2)^2\right)\,dx
\end{equation}
where $r$ is given in Definition \ref{notations}.

%
%

\begin{remark} \label{positive MoI}
	When $\rho$ has positive mass, we have $I(\rho)>0$. See \cite[Section 2]{Che26G1}.
\end{remark}

Since the star and planet are separated, it is convenient to describe the relations between the total moment of inertia $I\left(\rho_m+\rho_{1-m}\right)$ and $I\left(\rho_m\right), I\left(\rho_{1-m}\right)$.

\begin{lemma}[Expansion of Moment of Inertia {\cite[Section 2]{Che26G1}}]\label{expansion of MoI}
	Let $\rho\geq 0$, $\sigma \geq 0$ be the density functions in $\mathbb{R}^3$ with mass $\int_{\mathbb{R}^3} \rho\,dx=m_1<\infty$, $\int_{\mathbb{R}^3} \sigma\,dx=m_2<\infty$. $\bar{x}(\rho)$ and $\bar{x}(\sigma)$ denote the centers of mass, $I(\rho)$ and $I(\sigma)$ denote the moments of inertia, function $r$ is given in (\ref{notations}).
	
	\begin{enumerate}[(1)]
		\item If $m_1+m_2=0$, then $I(\rho+\sigma)=0$.
		\item If $m_1+m_2>0$, then we have the moment of inertia of $\rho+\sigma$ satisfies
		\begin{equation} \label{Expansion of MoI}
			I(\rho+\sigma)=I(\rho)+I(\sigma)+\frac{m_1 m_2}{m_1+m_2} r^2(\bar{x}(\rho)-\bar{x}(\sigma))
		\end{equation}
	\end{enumerate}
\end{lemma}

Let the support of $\rho$ be the smallest closed set carrying the full mass of $\rho$ (denoted by spt $\rho$). Intuitively we hope spt $\rho$  to be compact as this aligns with the case of a star-planet system and ensures a finite moment of inertia. Therefore, we also introduce a subset $R_a\left(\mathbb{R}^3\right)$ of $R\left(\mathbb{R}^3\right)$:

\begin{equation}\label{R_a}
	{R}_a\left(\mathbb{R}^3\right):=\left\{\rho \in {R}\left(\mathbb{R}^3\right) \mid \bar{x}(\rho)=a ; \text { spt } \rho \text { is bounded.}\right\}
\end{equation}

Since the energy is translation-invariant, we may just consider looking for an energy minimizer $\rho$ such that the center of mass $\bar{x}(\rho)$ is $0$, which means $\rho \in {R}_0\left(\mathbb{R}^3\right)$.

Thanks to Theorem \ref{Properties of LEM} below, it turns out the problem of minimizing $E(\rho, v)$ locally in $R_0(\mathbb{R}^3)\times V(\mathbb{R}^3)$ is equivalent to a minimization problem of the energy $E_J(\rho)$ locally in $R_0(\mathbb{R}^3)$. Here $E_J(\rho)$ corresponds to uniform rotation with the angular momentum $\boldsymbol{J}=J\hat{e}_z=(0,0,J)^T$ specified a priori, and $E_J(\rho)$ is defined as the following: 

\begin{equation} \label{energy of UR}
	E_J(\rho)=U(\rho)-\frac{G(\rho, \rho)}{2}+T_J(\rho)
\end{equation}

Here, $T_J$ is given by
\begin{equation}\label{T_J}
	T_J(\rho):=\frac{J^2}{2 I(\rho)}
\end{equation}
\begin{remark}\label{finite energy under L infty}
	Given $\rho \in R(\mathbb{R}^3)$, if we further know $\rho \in L^\infty(\mathbb{R}^3)$, then we know $U(\rho)$ is finite by interpolation inequality \cite[Section 4.2]{Bre11}: 
	\begin{equation}
		\|\rho\|_{L^r} \leq \|\rho\|_{L^p}^\theta \|\rho\|_{L^q}^{1-\theta}, \quad \text{for } p \leq r \leq q,
	\end{equation}
	where $\frac{1}{r} = \frac{\theta}{p} + \frac{1-\theta}{q}$ for some $\theta \in [0,1]$. This observation, together with Remark \ref{finite gravitational interaction energy} and Remark \ref{positive MoI}, tells us $E_J(\rho)$ is well-defined and finite. Similar results hold for more general pressures; see \cite[Section 5]{Che26G1}.
\end{remark}

We intend to solve (\ref{EP'}) by minimizing $E_J(\rho)$ and using results from calculus of variations, which involves the variational derivative of energy. We therefore now proceed to define it. We first define the perturbation set $P_{\infty}(\rho)$, which depends on energy minimizer $\rho$, as the following:
\begin{equation}\label{perturbation sets}
	P_{\infty}(\rho):=\bigcup_{R<\infty} P_R(\rho)
\end{equation}

Here, $P_R(\rho)$ is defined as:
$$P_R(\rho)=\left\{ \sigma \in L^{\infty}(\mathbb{R}^3)\mid \begin{array}{ll}
	\sigma(x)=0, & \text{where } x \text{ statisfies } \rho(x) >R \text{ or } |X|>R\\
	\sigma(x)\geq 0, & \text{where } x \text{ statisfies } \rho(x) <R^{-1}
\end{array} \right\}$$

One can see $P_{\infty}(\rho)$ is a convex cone. And then we have the following result:

\begin{lemma}[Differentiability of Energy $E_J(\rho)$ {\cite[Section 5]{Che26G1}}]\label{diff. of energy}
	Given $\rho \in mR(\mathbb{R}^3)$ with $U(\rho)<\infty$, $E_J(\rho)$ is $P_{\infty}(\rho)$-differentiable at $\rho$, i.e. it is differentiable at $\rho$ in the direction of $P_\infty (\rho)$. Moreover, the derivative at $\rho$ is $E_J^{\prime}(\rho)$ in the sense that $\forall \sigma\in P_{\infty}(\rho)$, $E_J^{\prime}(\rho)(\sigma)=\int_{\mathbb{R}^3} E_J^{\prime}(\rho) \sigma\,dx$ \footnote{To remain consistent with the notation in \cite[Section 4]{AB71}, we use $E_J^{\prime}$ as the symbol for both linear functional and function,  provided it does not cause confusion.}. The function $E_J^{\prime}(\rho)$ on the right‑hand side is given by
	\begin{equation}\label{variational derivative}
		E_J^{\prime}(\rho)(x):=A^{\prime}(\rho(x))-V_\rho(x)-\frac{J^2}{2 I^2(\rho)} r^2(x-\bar{x}(\rho))
	\end{equation}
\end{lemma}

A simple case is the non-rotating problem, i.e. $\boldsymbol{J}=0$, with energy 
\begin{equation}\label{non-rotating energy}
	E_0(\rho)=U(\rho)-\frac{G(\rho, \rho)}{2}
\end{equation}

Consider the non-rotating minimizer $\sigma_m$ of $E_0(\rho)$ among configurations of mass $m \geq 0$, the corresponding minimum energy turns out to be finite due to Theorem \ref{non-rotating}. For the sake of convenience, we denote them by
\begin{equation}\label{e_0}
	e_0(m):=E_0\left(\sigma_m\right)=\inf _{\rho \in {R}\left(\mathbb{R}^3\right)} E_0(m \rho)<\infty.
\end{equation}
We also denote $e_0(1)$ by $e_0$.
Properties of non-rotating minimizers are introduced in \cite{Che26R2, McC06, AB71, LY87}. We will also discuss them in Theorem \ref{non-rotating}.

In general, the angular momentum $\mathbf{J}$ is not 0. But under the constraint $J_z(\rho,v)=J$, where $J_z$ denotes the $z$-component of the angular momentum $\mathbf{J}$, if we know $\widetilde{\rho}$ is a local (w.r.t the topology induced by Wasserstein $L^\infty$ distance) minimizer over ${R}_0\left(\mathbb{R}^3\right)$ of $E_J(\rho)$, then $\widetilde{\rho}$ is a solution to (\ref{EP'}). Note the internal energy and gravitational potential energy are rotation‑invariant, and physically we expect celestial bodies to be rotating rather than stationary. Therefore, we can further define
$$
({\rho}(t, x), v(t, x)):=\left(\widetilde{\rho}\left(R_{-w t} x\right), \omega\left(-x_2, x_1, 0\right)^T\right) 
$$
where $\omega=\frac{J}{I\left(\rho\right)}$ and $R_\theta$ is a rotation map about $x_3$ axis given in Definition \ref{notations}, then $({\rho}(t, x), v(t, x))$ gives a uniform rotating star-planet system. Here ``rotating'' means the orbital revolution around each other, while the rotations of objects around their own axes are not considered. Moreover, $E(\rho,v)$ is local minimum, and $({\rho}(t, x), v(t, x))$ is solution to (\ref{EP}). We will discuss these results in detail in Theorem \ref{Properties of LEM}. 

\vspace{0.8em}
For prescribed angular momentum, we also know the energy $E_J(\rho)$ is bounded from below on ${R}_0\left(\mathbb{R}^3\right)$ by the non-rotating energy $e_0$. However, as in Morgan \cite{Mor02}, McCann \cite[Section 3]{McC06} demonstrates that this bound --- although approached --- will not be attained, see also \cite{McC94, Che26G1}. Therefore, instead of searching for a global energy minimizer in star-planet system, we adopt McCann's strategy \cite{McC06}: first establishing the existence of a constrained minimizer of $E_J(\rho)$, and then proving that it is also a local minimizer under the appropriate topology induced by Wasserstein $L^\infty$ distance.

Due to this uniform rotation observation, in the presence of 2-body systems, although the systems rotate in universe, after fixing a time, we can still assume that they fall within two disjoint regions $\Omega_m$ and $\Omega_{1-m}$, widely separated relative to $\frac{J^2}{\mu_r^{2}}$, where $\mu_r=m(1-m)$ is their reduced mass. For the planet's mass $m \in(0,1)$ and the star's mass $(1-m)$, we consider $E_J(\rho)$ is minimized subject to the constraint
\begin{equation} \label{admissible class}
	W_{m}:=\left\{\rho(m)=\rho_m+\rho_{1-m} \in R\left(\mathbb{R}^3\right) \mid \int_{\mathbb{R}^3} \rho_m\,dx=m, \text{spt } \rho_m \subset \Omega_m, \text{spt } \rho_{1-m} \subset \Omega_{1-m}\right\}
\end{equation}
where $\Omega_m$ and $\Omega_{1-m}$ are subsets of $\mathbb{R}^3$, which are given in the following. Note $W_{m}$, $\Omega_m$ and $\Omega_{1-m}$ are actually related to both $m$ and $J$.

Fix two points $y_m$ and $y_{1-m}$ in $\mathbb{R}^3$ from the plane $z=0$, which are separated by 
\begin{equation}\label{eta}
	\eta=\frac{J^2}{\mu_r^2}=\frac{J^2}{m^2(1-m)^2}
\end{equation}
i.e. $\eta=\left|y_m-y_{1-m}\right|$. The $\Omega_m$ and $\Omega_{1-m}$ are defined as closed balls in $\mathbb{R}^3$ centered at $y_m$ and $y_{1-m}$, whose size and separation scale with $\eta$ as the following:

\begin{equation}\label{domains}
	\begin{aligned}
		\Omega_m&:=\left\{x \in \mathbb{R}^3 \mid | x-y_m |\, \leq \frac{\eta}{4}\right\} \\
		\Omega_{1-m}&:=\left\{x \in \mathbb{R}^3\mid | x-y_{1-m} |\, \leq \frac{\eta}{4}\right\}
	\end{aligned}
\end{equation}

The distance separating $\Omega_m$ and $\Omega_{1-m}$, and the diameter of their union is given by:

\begin{align}
	{dist}\left(\Omega_m, \Omega_{1-m}\right) &=\frac{\eta}{2} \label{dist} \\
	{diam}\left(\Omega_m, \Omega_{1-m}\right) & =\frac{3 \eta}{2} \label{diam}
\end{align}

\begin{remark}
	The reason we set the separation $\eta=\frac{J^2}{\mu_r^2}$ in the definitions above is inspired by the Kepler problem. Given two point masses $m$ and $1-m$, rotating with angular momentum $J>0$ about their fixed center of mass, if we assume their separation is $d$, then the gravitational energy plus kinetic energy is $-\frac{\mu_r}{d}+\frac{J^2}{2 \mu_r d^2}$, which reaches its minimum at separation $d=\eta$. 
	
	In the star-planet model, when $\eta$ is large, the distance between two objects is large, and the gravitational interaction becomes weak. The system thus approximates the ideal case of two isolated, non-rotating, and non-interacting bodies. Therefore, it can be proven that the supports of the objects remain bounded, whose sizes are hence negligible compared to the distance between them. It implies the distance between their centers of mass asymptotically approaches that of the point mass model. We will discuss these facts in more detail in Section \ref{section5-Bound of the Supports of Density Functions} and Section \ref{section6-Existence for Star-Planet Systems}, assuming $m$ is sufficiently small. We will then establish their connection to the existence of local energy minimizers under the topology induced by $W^\infty$ (Wasserstein $L^\infty$) distance. For the case $J$ is sufficiently large, one can also check McCann's paper \cite[Section 6]{McC06}.
\end{remark}

	We note the choice of topology for ${R}\left(\mathbb{R}^3\right)$ is quite delicate: for $J>0$, it turns out local energy minimizer for $E_J(\rho)$ will not exist if the topology of $R(\mathbb{R}^3)$ is inherited from a topological vector space (\cite[Remark 3.7]{McC06} and \cite[Section 5]{Che26G1}). Roughly speaking, the reason is that for any reasonable candidate local minimizer $\rho$, one can always split off a portion of it with very small mass and shift the portion arbitrarily far away, thereby decreasing the energy. However, the resulting configuration still lies in a small neighborhood of $\rho$ with respect to such topology, which implies that $\rho$ cannot be a local minimizer. 
	
	Hence we need another topology, i.e., the topology induced by the $W^\infty$ distance.

\begin{definition}[Wasserstein $L^{\infty}$ distance]\label{Wasser.}
	Let $(X, d)$ be a metric space and $(S, \Sigma, v)$ be a finite positive measure space, given $\rho, \kappa \in \mathcal{P}(X)$, the \textit{Wasserstein $L^{\infty}$ distance} between $\rho$ and $\kappa$ is defined as
	
	\begin{equation}
		W^{\infty}(\rho, \kappa):=\inf \left\{\begin{array}{l}
			\|d(f(x), g(x))\|_{L^{\infty},v} \mid f: S \rightarrow X  \text { Borel }, \\
			f_{\#} v=\rho \text { and } g: S \rightarrow X  \text { Borel, } g_{\#} v=\kappa
		\end{array}\right\}
	\end{equation}
	
	Here $\|h\|_{L^{\infty},v}$ denotes the supremum of $|h|$ over $S$, discarding sets of $v$-measure zero. $M(S)$ and $M(X)$ are measure spaces. $f_{\#}: M(S) \rightarrow M(X)$ ($g_{\#}$ is similar) is the corresponding push forward operator defined by
	$$
	f_{\#} \tilde{\mu}(B):=\tilde{\mu}\left(f^{-1}(B)\right) \quad  \text{ for all } \tilde{\mu}\in M(S) \text{ and all Borel sets } B \subseteq X
	$$ 
\end{definition}

We call $f, g$ the \textit{transport maps}. Thanks to Strassen's Theorem, one can check $W^{\infty}$ is truly a metric, as explained in Givens and Shortt \cite{GS84}. We give two properties of the $W^\infty$ distance. For other properties, one can see \cite[Lemma 5.1]{McC06} and \cite[Section 4]{Che26G1}.

\begin{lemma}[Simple Properties of the Wasserstein $L^{\infty}$ Metric {\cite[Lemma 5.1]{McC06}}]\label{properties of Wasser.}
	Let $\rho, \kappa$ in $t {R}\left(\mathbb{R}^3\right)$, then
	\begin{enumerate}[(i)]
		\item [\mylabel{i}{($i$)}] $W^{\infty}(\rho, \kappa)$ does not exceed the diameter of $\operatorname{spt}(\rho-\kappa)$;
		\setcounter{enumi}{1}
		\item if $W^{\infty}(\rho, \kappa)<\delta$, each connected component of the $\delta$-neighbourhood of spt $\rho$ has the same mass for $\kappa$ as for $\rho$. Here the $\delta$-neighbourhood of some $\Omega \subset \mathbb{R}^3$ is defined as $\bigcup_{y \in \Omega}\left\{x \in \mathbb{R}^3\mid | x-y |<\delta\right\}$.
	\end{enumerate}
\end{lemma}

\begin{remark}
	The advantage of the topology induced by Wasserstein $L^\infty$ metric is that a neighborhood of $\rho$ does not include such “portion-shifted-to-far-away” configurations mention above due to Lemma \ref{properties of Wasser.} (ii), while thanks to Lemma \ref{properties of Wasser.} (i), the neighborhood still contains all perturbations of the form $\rho+t\sigma$ for all $t$, where $\sigma$ is supported on a set whose diameter is small enough, with $\int_{\mathbb{R}^3} \sigma\,dx =0$. Hence this setting still allows one to legitimately consider variational derivatives (\ref{variational derivative}) and consequently Euler-Lagrange equations (\ref{EL}) and Euler-Poisson equations \eqref{EP'} (see also Remark \ref{difference between McCann and Chen binary star system}).
\end{remark}

\begin{remark}\label{motivation of interior arguements}
	According to Lemma \ref{properties of Wasser.}, to prove $\rho(m)$ is a $W^\infty$ local minimizer (and hence a solution to (\ref{EP'}) due to Theorem \ref{Properties of LEM}), it suffices to show $\rho(m)$ is a constraint minimizer on $W_m$ such that its support lies in the interior of $\Omega_m \cup \Omega_{1-m}$. In fact, once we know $\rho(m)$'s support lies in the interior of $\Omega_m \cup \Omega_{1-m}$, one can also show $\rho(m)$ is a solution to (\ref{EP'}), as well as most results in Theorem \ref{Properties of LEM} (excluding the equivalence result that $(\rho, v)$ minimizes $E(\rho, v)$ locally if and only if $\rho$ minimizes $E_J(\rho)$ locally) without specifying $\rho(m)$ is a $W^\infty$ local minimizer, see \cite[Section 3]{Che26G1}. 
	
	It motivates our choice to define $\Omega_m$ and $\Omega_{1-m}$ with their radii that increase as $\eta$ increases, and then show the size of stars or planets will not expand too much as $\eta$ increases. We will show those results for small $m$ later, inspired by McCann's arguments for large $J$ \cite[Section 6]{McC06}.
\end{remark}

\subsection{Statement of Main Results}\label{subsection2.2-main results}

We are ready to describe our main results:
\begin{theorem}[Existence of Star-Planet Systems for $\gamma>2$]\label{themA}
	Given polytropic law $P(\rho)=K\rho^\gamma$ indexed by $\gamma>2$, fix $J>0$, there is a $\delta>0$, such that for all $m \in (0,\delta)$, a constrained energy minimizer $\rho(m)=\rho_{m}+\rho_{1 - m}$ on $W_m$ exists, which is actually also a Wasserstein $L^\infty$ local energy minimizer of $E_J (\rho)$ on $R(\mathbb{R}^3)$. Moreover, $\rho(m)$ satisfies the following properties:
	\begin{itemize}
		\item [$(i)$] $\rho(m)\in R_0(\mathbb{R}^3)$, and $(\rho(m), v)$ minimizes $E(\rho,v)$ locally on $R(\mathbb{R}^3)\times V(\mathbb{R}^3)$ subject to the constraint $J_z (\rho,v)\coloneq \hat{e}_z \cdot \boldsymbol{J}(\rho,v)=J$ or $\mathbf{J}(\rho,~v) = J{\hat{e}}_{z}$. Here $v(x):=\omega \hat{e}_z \times x =\omega\left(-x_2, x_1, 0\right)^T$, where $\omega = \frac{J}{I(\rho(m))}$.
		\item [$(ii)$] 	$\rho(m)$ is symmetric about the plane $z=0$ and a decreasing function of $|z|$.
		\item [$(iii)$] $\rho(m)$ is continuous and satisfies (\ref{EP'}) with $V_{\rho(m)}(x) = {\int_{\mathbb{R}^3}{\frac{\rho(m)(y)}{\left| {y - x} \right|}\,dy}}$. Moreover, the uniformly rotating fluid $(\widetilde{\rho(m)}, v)$ satisfies (\ref{EP}) with $V(t, x)=V_{{\rho(m)}}\left(R_{-\omega t} x\right)$, here $({\widetilde{\rho(m)}}(t, x), v(t, x))=\left({\rho(m)}\left(R_{-w t} x\right), \omega\left(-x_2, x_1, 0\right)^T\right)$.
		\item [$(iv)$] For the planet density, we have ${\lim\limits_{m\rightarrow 0}\left\| \rho_{m} \right\|_{L^{\infty}(\mathbb{R}^{3})}} = 0$. For the star density, $\left\| \rho_{1 - m} \right\|_{L^{\infty}(\mathbb{R}^{3})}$ is bounded uniformly for all small $m$, that is $m$ is located in a small interval $(0,\delta)$.
		\item [$(v)$] The support for the planet density spt $\rho_{m}$ is contained in a ball of a radius which goes to $0$ when $m$ goes to $0$. The support for the star density spt $\rho_{1-m}$ is contained in a ball of radius $R(J)$ for all small $m$.
	\end{itemize}
\end{theorem}

Actually, the results can be true in a larger range of index $\gamma$ if we do not require the support of the planet to shrink as the mass goes to 0.

\begin{theorem}[Existence of Star-Planet Systems for $\gamma>\frac{3}{2}$] \label{themA'}
	Given polytropic law $P(\rho)=K\rho^\gamma$ indexed by $\gamma>\frac{3}{2}$, fix $J>0$, there is a $\delta>0$, such that for all $m \in (0,\delta)$, a constrained energy minimizer $\rho(m)=\rho_{m}+\rho_{1 - m}$ on $W_m$ exists, which is actually also a Wasserstein $L^\infty$ local energy minimizer of $E_J (\rho)$ on $R(\mathbb{R}^3)$. Moreover, $\rho(m)$ satisfies the following properties:
	\begin{itemize}
		\item [$(i)$] $\rho(m)\in R_0(\mathbb{R}^3)$, and $(\rho(m), v)$ minimizes $E(\rho,v)$ locally on $R(\mathbb{R}^3)\times V(\mathbb{R}^3)$ subject to the constraint $J_z (\rho,v)\coloneq \hat{e}_z \cdot \boldsymbol{J}(\rho,v)=J$ or $\mathbf{J}(\rho,~v) = J{\hat{e}}_{z}$. Here $v(x):=\omega \hat{e}_z \times x =\omega\left(-x_2, x_1, 0\right)^T$, where $\omega = \frac{J}{I(\rho(m))}$.
		\item [$(ii)$] 	$\rho(m)$ is symmetric about the plane $z=0$ and a decreasing function of $|z|$.
		\item [$(iii)$] $\rho(m)$ is continuous and satisfies (\ref{EP'}) with $V_{\rho(m)}(x) = {\int_{\mathbb{R}^3}{\frac{\rho(m)(y)}{\left| {y - x} \right|}\,dy}}$. Moreover, the uniformly rotating fluid $(\widetilde{\rho(m)}, v)$ satisfies (\ref{EP}) with $V(t, x)=V_{{\rho(m)}}\left(R_{-\omega t} x\right)$, here $({\widetilde{\rho(m)}}(t, x), v(t, x))=\left({\rho(m)}\left(R_{-w t} x\right), \omega\left(-x_2, x_1, 0\right)^T\right)$.
		\item [$(iv)$] For the planet density, we have ${\lim\limits_{m\rightarrow 0}\left\| \rho_{m} \right\|_{L^{\infty}(\mathbb{R}^{3})}} = 0$. For the star density, $\left\| \rho_{1 - m} \right\|_{L^{\infty}(\mathbb{R}^{3})}$ is bounded uniformly for all small $m$ , that is $m$ is located in a small interval $(0,\delta)$.
		\item [$(v)$] The support for the star density spt $\rho_{1-m}$ is contained in a ball of radius $R(J)$ for all small $m$.
	\end{itemize}
\end{theorem}

We provide the proof of Theorem \ref{themA} and Theorem \ref{themA'} from Section \ref{section4-Existence of Constrained Minimizers} to Section \ref{section6-Existence for Star-Planet Systems}. The general outline is as follows: Given Theorem \ref{Properties of LEM}, the remaining task is to demonstrate the existence of a Wasserstein $L^\infty$ ($W^\infty$) local energy minimizer. To achieve this, we first prove the existence of constrained minimizers $\rho(m)$ of $E_J(\rho)$ on $W_m$ when the mass ratio is sufficiently small in Section \ref{section4-Existence of Constrained Minimizers}. In Section \ref{section5-Bound of the Supports of Density Functions}, we prove convergence of scaling densities in $L^p$, uniform $L^\infty$ bound, and then derive quantitative bounds on the supports of minimizers. In Section \ref{section6-Existence for Star-Planet Systems} we estimate the center of mass separation and show that $\rho(m)$ can be situated within the interior of the considered domains we consider, that is, $\text{dist } \left( \text{spt}~  \rho(m),~\mathbb{R}^{3}\backslash\left( \Omega_{m} \cup \Omega_{1 - m} \right) \right) > 0$, and then we verify $\rho(m)$ satisfies results in Theorem \ref{themA} and Theorem \ref{themA'}.

\section{Preliminary Results}\label{section3-Preliminary Results}

Recall that in the calculus of variations, if an energy functional admits a minimizer, then this minimizer satisfies Euler–Lagrange equation. When the minimization is carried out over a restricted class, the Euler–Lagrange equation may take the form of an inequality or involve taking the positive part; see, for example, Theorem \ref{Properties of LEM} or \cite[Section 2]{AB71}. In this section, we review some variational‑based results and their extensions. In subsection \ref{subsection3.1-results for 2-body systems}, we discuss McCann's work on Wasserstein $L^\infty$ local minimizers and existence of binary star system results. In subsection \ref{subsection3.2-Results for Non-rotating Bodies}, we review some results for single-star systems.

\subsection{Results for 2-body systems}\label{subsection3.1-results for 2-body systems}

In this subsection, we review McCann's results for Wasserstein $L^\infty$ ($W^{\infty}$)  local energy minimizers and existence theorem of binary star solution with slight modifications. We also make some remarks on them.

The following are McCann's results for $W^\infty$ local minimizers.

\begin{theorem}[Properties of $W^{\infty}$-Local Energy Minimizers {\cite[Section 2]{McC06}{\cite[Section 2]{Che26G1}}}] \label{Properties of LEM}
	Let $J>0$. If $(\rho, v)$ minimizes $E(\rho, v)$ locally on ${R}_0\left(\mathbb{R}^3\right) \times {V}\left(\mathbb{R}^3\right)$ subject to the constraint $J_z(\rho, v)=J$, where $J_z$ denotes the 
	$z$-component of the angular momentum $\mathbf{J}$, then:
	\begin{enumerate}[(i)]
		\item the z-axis is a principal axis of inertia for $\rho$, with a moment of inertia $I(\rho)$ which is maximal and non-degenerate;
		\item $(\rho, v)$ minimizes $E(\rho, v)$ locally on ${R}_0\left(\mathbb{R}^3\right) \times {V}\left(\mathbb{R}^3\right)$ if and only if $\rho$ minimizes $E_J(\rho)$ locally on ${R}_0\left(\mathbb{R}^3\right)$ and $v(x)=\omega \hat{e}_z \times x=\omega\left(-x_2, x_1, 0\right)^T$, where $\omega = \frac{J}{I(\rho(m))}$, which implies $E(\rho,v)=E_J(\rho)$. It describes uniformly rotating fluid, in the sense that after defining $({\widetilde{\rho}}(t, x), v(t, x))=\left({\rho}\left(R_{-w t} x\right), \omega\left(-x_2, x_1, 0\right)^T\right)$. In particular, $E(\rho, v)=E(\widetilde{\rho},v)=E_J(\widetilde{\rho})=E_J(\rho)$.
		\item $\rho$ is continuous on $\mathbb{R}^3$;
		\item on each connected component $\Omega_i$ of $\{\rho>0\}, \rho$ satisfies Euler-Lagrange equation:
		\begin{equation} \label{EL}
			A^{\prime}(\rho(x))=\left[\frac{J^2}{2 I^2(\rho)} r^2(x-\bar{x}(\rho))+V_\rho(x)+\lambda_i\right]_{+} \tag{EL}
		\end{equation}
		for some Lagrange multiplier $\lambda_i<0$ depending on the component. Here $[\cdot]_{+}$ is the nonnegative (positive) part function defined as $[\lambda]_{+}:=\max \{\lambda, 0\}$;
		\item the equations (\ref{EL}) continue to hold on a $\delta$-neighbourhood of the $\Omega_i$;
		\item $\rho\in C^1(\{\rho>0\})$;
		\item $\rho$ satisfies the reduced Euler-Poisson equations (\ref{EP'}) on $\mathbb{R}^3$, where the pressure $P(\rho)$ is given in (\ref{polytropic law}). Moreover, $({\widetilde{\rho}}(t, x), v(t, x))$ solves Euler-Poisson equations (\ref{EP}) with $V(t, x)=V_{{\rho}}\left(R_{-\omega t} x\right)$.
		\item this solution is stable with respect to $L^{\infty}$-small perturbations of the Lagrangian fluid variables.
	\end{enumerate}
\end{theorem}

\begin{remark}
	In Theorem \ref{Properties of LEM} we only mention the case where $P(\rho)$ satisfies polytropic law (\ref{polytropic law}), but the results hold true for a more general form of $P(\rho)$, see for instance \cite[Section 2]{Che26G1}.
\end{remark}

\begin{remark}\label{retaining center of mass}
	Notice in Theorem \ref{Properties of LEM}, since the local minimizer $\rho \in {R}_0\left(\mathbb{R}^3\right)$, we know $\bar{x}(\rho)=(0,0,0)^T$ in (\ref{EL}). However, one can generalize to the case where the center of mass is not at the origin, hence we retain the term $\bar{x}(\rho)$ in (\ref{EL}) in Theorem \ref{Properties of LEM} (iv). This is consistent with the form in McCann's paper \cite[Section 6]{McC06}. Furthermore, if $\rho$ is a constrained minimizer instead of a $W^\infty$ local minimizer, (\ref{EL}) still holds on $\Omega_i$.
\end{remark}

\begin{remark}\label{difference between McCann and Chen binary star system}
	McCann initially introduced and proved Theorem \ref{Properties of LEM} in \cite{McC06}. We later complemented this work in \cite{Che26G1} by showing why (\ref{EP}) also holds on the boundary of the support of minimizer. Additionally, in \cite{Che26G1} we verified the local minimizer $\rho$ has finite internal energy $U(\rho)<\infty$, hence $E_J$ is indeed differentiable at $\rho$ due to Lemma \ref{diff. of energy}. Further discussion—including the reason why the local minimum energy is attained under uniform rotation (Theorem \ref{Properties of LEM} (ii)), the interpretation of the stability result (Theorem \ref{Properties of LEM} (viii)), the generalization of the fixed angular momentum from $J_z$ to $\mathbf{J}$, and the equivalence between (\ref{EP'}) and (\ref{EL})—can be found in \cite{McC06, Che26G1, JS22}.

\end{remark}

Thanks to Theorem \ref{Properties of LEM}, together with the arguments of McCann \cite{McC06}, we know a constrained energy minimizer on $W_{m}$ exists and its support lies in the interior of $\Omega_m \cup \Omega_{1-m}$ when $J$ is sufficiently large. Therefore, this constrained minimizer is a $W^\infty$ local minimizer. And we know the following theorem:

\begin{theorem}[Existence of Binary Stars {\cite[Theorem 6.1, Corollary 6.2]{McC06}}] \label{McCann's existence}
	Given $m \in(0,1)$, choose the angular momentum $J$ to be sufficiently large depending on $m$. Then any constrained minimizer $\widetilde{\rho}=\rho^{-}+\rho^{+}$of $E_J(\rho)$ on $W_{m,J}$ will, after a rotation about the z-axis and a translation, have support contained in the interior of $\Omega:=\Omega_{-} \cup \Omega_{+}$, that is, dist$\left(\operatorname{spt} \widetilde{\rho}, \mathbb{R}^3 \backslash \Omega\right)>0$. It will also be symmetric about the plane $z=0$ and a decreasing function of $|z|$. 
	
	What's more, after another translation the center of mass of $\widetilde{\rho}$ can be 0 and is a local minimizer of $E_J(\rho)$. Let $v(x):=\omega \hat{e}_z \times x =\omega\left(-x_2, x_1, 0\right)^T$, where $\omega = \frac{J}{I(\widetilde{\rho})}$, then the pair $(\widetilde{\rho}, v)$ minimizes $E(\rho, v)$ locally on ${R}\left(\mathbb{R}^3\right) \times {V}\left(\mathbb{R}^3\right)$ (thus on ${R}_0\left(\mathbb{R}^3\right) \times {V}\left(\mathbb{R}^3\right))$ subject to the constraint $J_z(\rho, v)=J$ or $\boldsymbol{J}(\rho, v)=J \hat{e}_z$. $\widetilde{\rho}$ satisfies reduced Euler-Poisson equations (\ref{EP'}). Moreover, the uniformly rotating fluid $({\rho}(t, x), v(t, x))$ solves (\ref{EP}) with $V(t, x)=V_{\widetilde{\rho}}\left(R_{-\omega t} x\right)$, here $({\rho}(t, x), v(t, x))=\left(\widetilde{\rho}\left(R_{-\omega t} x\right), \omega\left(-x_2, x_1, 0\right)^T\right)$.
\end{theorem}

\subsection{Results for Non-rotating Bodies}\label{subsection3.2-Results for Non-rotating Bodies}

 In this subsection, we introduce the existence and uniqueness of non-rotating solutions of Euler-Poisson equations, corresponding to the single-star case. We also indicate some useful properties about the solutions.

\begin{theorem}[Non-rotating Stars {\cite{AB71, LY87, McC06}}]\label{non-rotating}
	For $E_0(\rho)$ from (\ref{non-rotating energy}), $e_0(m)$ from (\ref{e_0}) and $m \in [0,  \infty)$:
	\begin{enumerate}[(i)]
		\item $E_0(\rho)$ attains its minimum $e_0(m)$ among $\rho$ such that $\rho \in m{R}\left(\mathbb{R}^3\right)$, and the minimizer $\sigma_m$ is unique up to translation;
		\item $e_0(m)$ decreases continuously from $e_0(0)=0$ and is strictly concave. In particular, for $m>0$, $e_0(m)<0$, $e_0^{\prime}\left(m^{+}\right) \coloneq\lim\limits_{\widetilde{m}\rightarrow m^{+}}{e_{0}^{\prime}\left( \widetilde{m} \right)}$ and $e_0^{\prime}\left(m^{-}\right)\coloneq \lim\limits_{\widetilde{m}\rightarrow m^{-}}{e_{0}^{\prime}\left( \widetilde{m} \right)}$ exists, and we have $e_0^\prime(m^+)\leq e_0^\prime(m^-)<0$;
		\item $\sigma_m$ is spherically symmetric and radially decreasing after translation;
		\item $\left\|\sigma_m\right\|_{L^{\infty}} \leq C(m)$, and $U(\sigma_m)<\infty$ due to Remark \ref{finite energy under L infty};
		\item spt $\sigma_m$ is contained in a ball of radius $R_0(m)$;
		\item $\sigma_m \in C^0 (\mathbb{R}^3) \cap C^1 (\{\sigma_m>0\})$; 
		\item $\sigma_m$ satisfies 
		\begin{equation} \label{EL0}
			A^{\prime}(\sigma_m(x))=\left[V_{\sigma_m}(x)+\lambda_m\right]_{+} \tag{EL0}
		\end{equation}
		on all of $\mathbb{R}^3$ and a single Lagrange multiplier $ \lambda_m=-(5 \gamma-6) m^{\frac{2 \gamma-2}{3 \gamma-4}} U(\sigma) \left\{\begin{array}{l}<0, m>0 \\ =0, m=0\end{array}\right.$, where $\sigma = \sigma_1$ is the minimizer of $E_0(\rho)$ with mass $1$;
		\item the derivatives of $e_0(m)$ exists and satisfies: $e_0^{\prime}(m)=\lambda_m$;
		\item $\sigma_m$  satisfies the reduced Euler-Poisson equations (\ref{EP'}).
	\end{enumerate}
\end{theorem}

\begin{remark}\label{simply connected for non-rotating case}
		Theorem \ref{non-rotating} (iii), which is based on the rearrangement inequality, establishes that the support of $\rho$ is simply connected. This configuration corresponds to a single-star system and aligns with our theoretical expectations. In contrast, although McCann's binary star model \cite{McC06} demonstrates that the support of $\rho$ is contained in two disjoint balls, it does not establish that the support consists of exactly two connected components. Whether the support actually comprises exactly two connected components—thus truly representing a binary star system—remains an open question. In Section \ref{section7-The Maximum Number of Connected Components of Minimizers}, we estimate the distance between possibly multiple connected components and propose Conjecture \ref{conjecture 2 connected components}, which asserts that there can be at most two components.
\end{remark}

\begin{remark}\label{non-rotating for general case}
	In Theorem \ref{non-rotating}, we only mention the case where $P(\rho)$ satisfies polytropic law (\ref{polytropic law}), but some results hold true for a more general form of $P(\rho)$, see for instance \cite{AB71, LY87, McC06, Che26R2}. Auchmuty and Beals initially concluded that (\ref{EP'}) is satisfied in the region where $\rho > 0$ \cite[Theorem A]{AB71}. We later improved the result in \cite[Section 2]{Che26R2} by showing the equation holds over the entire $\mathbb{R}^3$ space. 
	Theorem \ref{non-rotating} (ii) was stated without proof by McCann \cite[Section 3]{McC06}, who mentioned it could be demonstrated using a method analogous to that in \cite{LY87}—which is in the quantum mechanical framework. A proof for the polytropic case was later given in \cite{Che26R2}. 
	
	Note that without polytropic law assumption, the uniqueness of minimizer, the existence of $e'_0(m)$, and the equality relations between $\lambda_m$, $U(\sigma$) and $e^\prime_0(m)$ in (vii) and (viii) are not guaranteed. Instead, we have $e_0^{\prime}\left(m^{+}\right) \leq \lambda_m \leq e_0^{\prime}\left(m^{-}\right)$. Note $e_0^{\prime}\left(m^{-}\right)<0$ if $m>0$ due to Theorem \ref{non-rotating} (ii). Although uniqueness result in the general case is unknown, uniqueness can still be established under assumptions slightly weaker than the polytropic law; see \cite{McC06, Che26R2}.
\end{remark}


It turns out under polytropic law (\ref{polytropic law}) assumption, minimizers with different masses satisfy scaling relation, which helps to strengthen Theorem \ref{non-rotating} (ii), as shown in the following results:

\begin{theorem}[Relations between Minimizers and Minimal Energies with Different Mass {\cite[Section 3]{Che26R2}}] \label{scaling relation for energy}
	Let $\sigma$ be the minimizer with mass 1 and the corresponding minimal energy is $e_0=E_0(\sigma)$, then any other minimizer with mass $m$ can be represented as $\sigma_m(x)=\frac{1}{A} \sigma\left(\frac{1}{B} x\right)$, where $A=m^{-\frac{2}{3 \gamma-4}}, B=m^{\frac{\gamma-2}{3 \gamma-4}}$, and the corresponding minimal energy is $e_0(m)=E_0\left(\sigma_m\right)=m^{\frac{5 \gamma-6}{3 \gamma-4}} e_0$.
\end{theorem}

\begin{proof}[Sketch of Proof]
	Recall (\ref{non-rotating energy}) that $E_0(\sigma)=U(\sigma)-\frac{G(\sigma, \sigma)}{2}=\int_{\mathbb{R}^3} A(\sigma)\,dx-\frac{1}{2} \iint_{\mathbb{R}^3\times \mathbb{R}^3} \frac{\sigma(x) \sigma(y)}{|x-y|} \,dx \,dy$, where $A(\sigma)=\frac{K}{\gamma-1} \sigma^\gamma$. Then let $\sigma_{\widetilde{A}, B}=\widetilde{A} \sigma(B x)$, we can compute
	
	\begin{equation}\label{scaling relation for inertial energy}
		\begin{aligned}U\left(\sigma_{\widetilde{A}, B}\right)&=\int_{\mathbb{R}^3} A\left(\sigma_{\widetilde{A}, B}\right)\,dx\\
			&=\int_{\mathbb{R}^3} \frac{K}{\gamma-1}(\widetilde{A} \sigma(B x))^\gamma\,dx\\
			&=\frac{\widetilde{A}^\gamma}{B^3} \int_{\mathbb{R}^3} \frac{K}{\gamma-1}(\sigma(B x))^\gamma \,d(Bx)\\
			&=\widetilde{A}^\gamma B^{-3} U(\sigma)
		\end{aligned}
	\end{equation}

	\begin{equation}\label{scaling relation for gravitational potential energy}
		\begin{aligned}
			G\left(\sigma_{\widetilde{A}, B}, \sigma_{\widetilde{A}, B}\right)&=\iint_{\mathbb{R}^3\times \mathbb{R}^3} \frac{\sigma_{\widetilde{A}, B}(x) \sigma_{\widetilde{A}, B}(y)}{|x-y|} \,dx\,dy\\
			&=\widetilde{A}^2 B^{-5} \iint_{\mathbb{R}^3\times \mathbb{R}^3}\frac{\sigma(B x) \sigma(B y)}{|B x-B y|} \,d(B x) \,d(B y) \\
			&=\widetilde{A}^2 B^{-5} G(\sigma, \sigma) 
		\end{aligned}
	\end{equation}
	
	\begin{equation}\label{scaling relation for mass}
		\begin{aligned}
			\int_{\mathbb{R}^3} \sigma_{\widetilde{A}, B}\,dx&=\widetilde{A} B^{-3} \int_{\mathbb{R}^3} \sigma(B x) \,d(B x)\\
			&=\widetilde{A} B^{-3} \int_{\mathbb{R}^3} \sigma\,dx
		\end{aligned}
	\end{equation}

	From the scaling relations (\ref{scaling relation for inertial energy}) (\ref{scaling relation for gravitational potential energy}) (\ref{scaling relation for mass}), together with the uniqueness result Theorem \ref{non-rotating} (i), if we take $A=m^{-\frac{2}{3 \gamma-4}}, B=m^{\frac{\gamma-2}{3 \gamma-4}}$, then we conclude the proof. For a complete discussion, one can see \cite[Section 3]{Che26R2}.
\end{proof}

\begin{remark}\label{asymptotic behaviour of radius and norm in single star case}
	Thanks to Theorem \ref{non-rotating}, we know $\|\sigma\|_{L^{\infty}\left(\mathbb{R}^3\right)} \leq C_1$ and spt $\sigma$ is contained in a ball of radius $R_1$, therefore, $\left\|\sigma_m\right\|_{L^{\infty}\left(\mathbb{R}^3\right)} \leq C_m=\frac{C_1}{A}$ and spt $\sigma_m$ is contained in a ball of radius $R_m=B R_1$. Since $A=m^{-\frac{2}{3 \gamma-4}}, B=m^{\frac{\gamma-2}{3 \gamma-4}}$, if we further assume $\gamma >2$, we know $\lim\limits_{m \rightarrow 0} A=+\infty$ and $\lim\limits_{m \rightarrow 0} B=0$, thus $\left\|\sigma_m\right\|_{L^{\infty}\left(\mathbb{R}^3\right)}$ and the size of $\sigma_m$'s support will go to 0 when $m \rightarrow 0$, with rates $m^{\frac{2}{3 \gamma-4}}$ and $m^{\frac{\gamma-2}{3 \gamma-4}}$ respectively. When $\gamma<2$, we still have $\lim\limits_{m \rightarrow 0} A=+\infty$, but $\lim\limits_{m \rightarrow 0} B=\infty$, and the size of $\sigma_m$'s support can go to $\infty$ when $m \rightarrow 0$ (flatten out or dispread). It coincides with a result mentioned in Lieb and Yau's paper \cite[Theorem 5]{LY87}, which says that the radius $R_m \rightarrow \infty$ as $m \rightarrow 0$. In their paper quantum mechanics (fermions case) is discussed.
\end{remark}

\begin{remark} 
	Now we know the minimizer satisfies (\ref{EL}). Actually the converse also holds true: all radial solutions of (\ref{EL}) are in fact minimizers of $E_0(\rho)$ among $mR(\mathbb{R}^3)$ for some mass $m$. Furthermore, the minimizer (or radial solution to (\ref{EL})) can be uniquely parameterized by its central density, with each central density corresponding to a specific mass, see \cite[Section 2]{Che26R2}. Thanks to Theorem \ref{non-rotating}, $\left\|\sigma_m\right\|_{L^{\infty}\left(\mathbb{R}^3\right)}$ is actually the central density of $\sigma_m$, thus Theorem \ref{scaling relation for energy} shows as the central density decreases, the mass also decreases. Moreover, it gives the decay rate. For more general pressure $P(\rho)$, this decreasing relationship is also established in \cite[Section 2]{Che26R2}.
\end{remark}

Let $\sigma_{\widetilde{A}, B}=\widetilde{A} \sigma(B x)$, with $A=m^{-\frac{2}{3 \gamma-4}}$, $B=m^{\frac{\gamma-2}{3 \gamma-4}}$ as above, actually we also have the following scaling relations:
\begin{equation}\label{scaling for A prime}
	A^{\prime}\left(\sigma_{\widetilde{A}, B}\right)(x)=A^{\prime}\left(\widetilde{A} \sigma\right)\left(Bx\right)=\widetilde{A}^{\gamma-1} A^{\prime}(\sigma)\left({B} x\right)=m^{\frac{2-2 \gamma}{3 \gamma-4}}A^{\prime}(\sigma)\left({B} x\right)
\end{equation}

\begin{equation}\label{scaling for V}
	V_{\left(\sigma_{\widetilde{A}, B}\right)}(x)=\frac{\widetilde{A}}{B^2} V_\sigma\left(Bx\right)=\widetilde{A}^{\gamma-1} V_\sigma\left({B}x\right)=m^{\frac{2-2 \gamma}{3 \gamma-4}} V_\sigma\left({B}x\right)
\end{equation}

\begin{equation}\label{scaling for E prime}
	E_0^{\prime}\left(\sigma_{\widetilde{A}, B}\right)(x)=\widetilde{A}^{\gamma-1} E_0^{\prime}(\sigma)\left(Bx\right)=m^{\frac{2-2 \gamma}{3 \gamma-4}} E_0^{\prime}(\sigma)\left({B}x\right)
\end{equation}
Therefore, we can further discuss the relationship between variational derivatives (see Lemma \ref{diff. of energy}) of densities with different mass (not necessarily to be minimizers). 
\begin{proposition}[Relations Between Variational Derivatives of Densities with Different Mass {\cite[Section 3]{Che26R2}}]\label{scaling relation for derivatives}
	Given $\rho_m \in mR(\mathbb{R}^3)$ with $U(\rho_m)<\infty$, then $E_0\left(\rho_m\right)$ is $P_{\infty}\left(\rho_m\right)$ differentiable at $\rho_m$ and the derivative at $\rho_m$ satisfies $$E_0^{\prime}\left(\rho_m\right)(x)= A^{-\gamma-1}E_0^{\prime}(\rho)\left(\frac{x}{B}\right)=m^{\frac{2 \gamma-2}{3 \gamma-4}}E_0^{\prime}(\rho)\left(\frac{1}{B} x\right)$$ 
	where $\rho(x):=\left(\rho_m\right)_{A, B}(x)=A \rho_m(B x) \in R(\mathbb{R}^3)$ with $U(\rho)<\infty$, $A=m^{-\frac{2}{3 \gamma-4}}, B=m^{\frac{\gamma-2}{3 \gamma-4}}$. In particular, $\rho$ has mass 1.
\end{proposition}

\section{Existence of Constrained Minimizers of $E_{J}(\rho)$ on ${W}_{m}$}\label{section4-Existence of Constrained Minimizers}
Similar to the proof of Theorem \ref{non-rotating} (see e.g. \cite{Che26R2}), instead of applying direct method to \textit{admissible class} $W_{m}$ (\ref{admissible class}), we first construct a ``double constrained'' admissible class $W_{m, R}$ in which we can show the lower semi-continuity of the energy. We follow the definition of $W_{m, R}$ introduced in \cite{McC06}.

\begin{definition}[``Double Constrained'' Admissible Class] \label{W_{m,R}}
	Let $W_m$ be as defined in (\ref{admissible class}), we then define $W_{m, R}$ as the following: 	
	$$W_{m, R}:=\left\{\rho(m)=\rho_{m}+\rho_{1-m} \in W_m \mid \rho(m) \leq R \text{ a.e.}\right\}$$
\end{definition}

\begin{remark}\label{non-empty of W_{m, R}}
	$W_{m, R} \subset L^{1} \cap L^{\infty}$ and it is not empty when $R$ is larger than $\frac{384}{\pi \eta^{3}}$, since we can pick $\rho_{m}(x)=\left\{\begin{array}{ll}\frac{384 m}{\pi \eta^{3}},&\left|x-y_{m}\right|<\frac{\eta}{8} \\ 0,&\left|x-y_{m}\right| \geq \frac{\eta}{8}\end{array}\right.$ and $\rho_{1-m}$ similarly, then check $\rho(m)=$ $\rho_{1-m}+\rho_{m}$ is in $W_{m, R}$.
\end{remark}

We want to prove the existence of ``double constrained'' minimizers via direct method \cite[Theorem 1.15]{Dal93}. Here, ``double constrained'' means $\rho(m) \in W_{m}$ and $\rho(m) \leq R$ a.e.. To do that, we first show the lower semicontinuity of $E_{J}$ by reformulating arguments in \cite{AB71} where the authors only consider non-rotating one-body case.

With an abuse of notation, we denote sequences in $W_{m, R}$ by $\left\{\rho_{n, R}\right\}$ and their limit by $\rho_{R}$. Note their total mass ${\int_{\mathbb{R}^3}\rho_{n,R}}\,dx$ is always one.

\begin{lemma}[Lower Semicontinuity of $E_{J}$ in $W_{m, R}$ w.r.t the Weak Topology on $L^{\frac{4}{3}}\left(\mathbb{R}^{3}\right)$] \label{l.s.c on weak L^4/3}
	Fix $J>0$, $m>0$, if $\left\{\rho_{n, R}\right\} \subset W_{m, R}$ and $\rho_{n, R} \rightharpoonup \rho_{R}$ in $L^{\frac{4}{3}}\left(\mathbb{R}^{3}\right)$, then $\rho_{R} \in W_{m, R}$ and $E_{J}\left(\rho_{R}\right) \leq \liminf\limits_{n \rightarrow \infty} E_{J}\left(\rho_{n, R}\right)$.
\end{lemma}

\begin{proof}
	{
		\rm
		If $\left\{\rho_{n, R}\right\} \subset W_{m, R}$ and $\rho_{n, R} \rightarrow \rho_{R}$ in $L^{\frac{4}{3}}$, then for almost every $x \in \mathbb{R}^{3}, 0 \leq \rho_{n, R} \leq$ $R$ and up to a subsequence $\rho_{n, R} \rightarrow \rho_{R}$. Then $0 \leq \rho_{R} \leq R$ almost everywhere. Since $\Omega:=$ $\Omega_{m} \cup \Omega_{1-m}$ is bounded, by Hölder's inequality we know $\left\|\rho_{n, R}-\rho_{R}\right\|_{L^{1}} \leq C \| \rho_{n, R}-$ $\rho_{R} \|_{L^{\frac{4}{3}}} \rightarrow 0$. It implies $\rho_{R} \in W_{m, R}$ thus $W_{m, R}$ is closed. Easy to check it is also a convex subset of $L^{\frac{4}{3}}$. Therefore, $W_{m, R}$ is weakly closed.
		
		To prove the lower semicontinuity, we recall 
		$$
		\begin{aligned}
			E_{J}(\rho)&=U(\rho)-\frac{G(\rho, \rho)}{2}+T_{J}(\rho)\\
			&=\int_{\mathbb{R}^3} A(\rho)\,dx-\frac{\int_{\mathbb{R}^3} V_{\rho} \cdot \rho\,dx}{2}+T_{J}(\rho)\\
			&=\int_{\mathbb{R}^3} \frac{K(\rho(x))^{\gamma}}{\gamma-1} \,dx-\frac{1}{2} \iint_{\mathbb{R}^3\times \mathbb{R}^3} \frac{\rho(x) \cdot \rho(y)}{|x-y|} \,d x \,d y+\frac{J^{2}}{2(I(\rho))}
		\end{aligned}
		$$
		We first notice $A(\rho)$ is a convex function, so $\int_{\mathbb{R}^3} A(\rho)\,dx$ is a convex functional. Then for each constant $k$, $\left\{\rho_{R} \in\right.$ $\left.W_{m, R} \mid \int_{\mathbb{R}^3} A\left(\rho_{R}\right)\,dx \leq k\right\}$ is convex. $\rho_{n, R} \rightarrow \rho_{R}$ in $L^{\frac{4}{3}}$ implies almost everywhere convergence up to a subsequence, and by construction we know $\rho_{n,R} \leq R\cdot1_\Omega$, here $1_\Omega$ is the indicator function of compact set $\Omega$. By (\ref{A}) we know $A(0)=0$, $A(\cdot)$ is non-decreasing, and thus $A(\rho_{n,R}) \leq A(R)\cdot1_\Omega \in L^1$. Therefore, if $\int_{\mathbb{R}^3} A\left(\rho_{n, R}\right)\,dx \leq k$ then $\int_{\mathbb{R}^3} A\left(\rho_{R}\right)\,dx \leq k$ by dominated convergence theorem or Fatou's lemma, which means $\left\{\rho_{R} \in W_{m, R} \mid \int_{\mathbb{R}^3} A\left(\rho_{R}\right)\,dx \leq k\right\}$ is not only convex but also closed. Thus, it is also weakly closed and then $\int_{\mathbb{R}^3} A(\rho)\,dx$ is a weakly lower semicontinuous functional. 
		
		Next, we know $\mathcal{V}: \rho \rightarrow V_{\rho}$ is compact from $L^{\frac{4}{3}}(\Omega)$ to $L^{r}(\Omega)$ for any $1 \leq r<12$ (see, e.g., \cite[Appendix A]{Che26R2}), thus $\mathcal{V}$ is completely continuous. Take $r=11$, then $\rho_{n, R} \rightharpoonup \rho_{R}$ in $L^{\frac{4}{3}}(\Omega)$ implies $V_{\rho_{n, R}} \rightarrow V_{\rho_{R}}$ in $L^{11}(\Omega)$. Since $\Omega$ is bounded, by Holder's inequality we know $V_{\rho_{n, R}} \rightarrow V_{\rho_{R}}$ in $L^{4}(\Omega)$. Since $\rho_{R}$ and $\rho_{n, R}$ vanish outside $\Omega$, we have
		\begingroup
		\small
		$$
		\begin{aligned}
			\left|G\left(\rho_{R}, \rho_{R}\right)-G\left(\rho_{n, R}, \rho_{n, R}\right)\right|  =&\left|\int_{\mathbb{R}^{3}} V_{\rho_{R}} \cdot \rho_{R}\,dx-\int_{\mathbb{R}^{3}} V_{\rho_{n, R}} \cdot \rho_{n, R}\,dx\right| \\
			=&\left|\int_{\Omega} V_{\rho_{R}} \cdot \rho_{R}\,dx-\int_{\Omega} V_{\rho_{R}} \cdot \rho_{n, R}\,dx\right|+\left|\int_{\Omega} V_{\rho_{R}} \cdot \rho_{n, R}\,dx-\int_{\Omega} V_{\rho_{n, R}} \cdot \rho_{n, R}\,dx\right| \\
			\leq&\left|\int_{\Omega} V_{\rho_{R}} \cdot \rho_{R}\,dx-\int_{\Omega} V_{\rho_{R}} \cdot \rho_{n, R}\,dx\right| +\left\|\rho_{n, R}\right\|_{L^{\frac{4}{3}}(\Omega)} \cdot\left\|V_{\rho_{R}}-V_{\rho_{n, R}}\right\|_{L^{4}(\Omega)}
		\end{aligned}
		$$
		\endgroup
		The last inequality comes from Hölder's inequality. Note $V_{\rho_{R}} \in L^{11}(\Omega) \subset L^{4}(\Omega)=\left(L^{\frac{4}{3}}(\Omega)\right)^{*}$, together with $\rho_{n, R} \rightharpoonup \rho_{R}$ in $L^{\frac{4}{3}}\left(\mathbb{R}^{3}\right)$ we know $\left|\int_{\Omega} V_{\rho_{R}} \cdot \rho_{R}\,dx-\int_{\Omega} V_{\rho_{R}} \cdot \rho_{n, R}\,dx\right| \rightarrow 0$. Note $\rho_{n, R} \rightharpoonup \rho_{R}$ in $L^{\frac{4}{3}}\left(\mathbb{R}^{3}\right)$ also implies $\left\|\rho_{n, R}\right\|_{L^{\frac{4}{3}}(\Omega)} \leq C$, thus $\left\|\rho_{n, R}\right\|_{L^{\frac{4}{3}}(\Omega)} \cdot\left\|V_{\rho_{R}}-V_{\rho_{n, R}}\right\|_{L^{4}(\Omega)} \rightarrow 0$ due to $V_{\rho_{n, R}} \rightarrow V_{\rho_{R}}$ in $L^{4}(\Omega)$. Hence we have $G\left(\rho_{n, R}, \rho_{n, R}\right) \rightarrow G\left(\rho_{R}, \rho_{R}\right)$. 
		
		Finally, if $\rho \in W_{m, R}$, then $\bar{x}(\rho)=\frac{\int_{\mathbb{R}^{3}} \rho(x) x\,dx}{\int_{\mathbb{R}^{3}} \rho\,dx}=\int_{\Omega} \rho(x) x\,dx$, and $\rho_{n, R} \rightharpoonup \rho_{R}$ implies $\bar{x}\left(\rho_{n, R}\right) \rightarrow \bar{x}\left(\rho_{R}\right)$ since $f(x)=x \in L^{\infty}\left(\Omega ; \mathbb{R}^{3}\right) \subset\left(L^{\frac{4}{3}}\left(\Omega ; \mathbb{R}^{3}\right)\right)^{*}$. Moreover,

		\begingroup
		\small  
		\[
		\begin{aligned}
			&\quad \left|I\left(\rho_{n, R}\right)-I\left(\rho_{R}\right)\right|  \\
			&= \left|\int_{\Omega} \rho_{n, R}(x) r^{2}\left(x-\bar{x}\left(\rho_{n, R}\right)\right)\,dx-\int_{\Omega} \rho_{R}(x) r^{2}\left(x-\bar{x}\left(\rho_{R}\right)\right)\,dx\right| \\
			&\leq \left|\int_{\Omega} \rho_{n, R}(x) r^{2}\left(x-\bar{x}\left(\rho_{n, R}\right)\right)\,dx-\int_{\Omega} \rho_{n, R}(x) r^{2}\left(x-\bar{x}\left(\rho_{R}\right)\right)\,dx\right| \\
			&+ \left|\int_{\Omega} \rho_{n, R}(x) r^{2}\left(x-\bar{x}\left(\rho_{R}\right)\right)\,dx-\int_{\Omega} \rho_{R}(x) r^{2}\left(x-\bar{x}\left(\rho_{R}\right)\right)\,dx\right| \\
			&\leq \left\|\rho_{n, R}\right\|_{L^{\frac{4}{3}}(\Omega)} \cdot\left\|r^{2}\left(x-\bar{x}\left(\rho_{n, R}\right)\right)-r^{2}\left(x-\bar{x}\left(\rho_{R}\right)\right)\right\|_{L^{4}(\Omega)} + \left|\int_{\Omega}\left(\rho_{n, R}(x)-\rho_{R}(x)\right) r^{2}\left(x-\bar{x}\left(\rho_{R}\right)\right)\,dx\right|
		\end{aligned}
		\]
		\endgroup
		By an argument similar to the above, we know $r^{2}\left(x-\bar{x}\left(\rho_{n, R}\right)\right)$ and $r^{2}\left(x-\bar{x}\left(\rho_{R}\right)\right)$ are in $L^{\infty}(\Omega) \subset$ $L^{4}(\Omega)=\left(L^{\frac{4}{3}}(\Omega)\right)^{*}$.  $\bar{x}\left(\rho_{n, R}\right) \rightarrow \bar{x}\left(\rho_{R}\right)$ means $r^{2}\left(x-\bar{x}\left(\rho_{n, R}\right)\right) \rightarrow r^{2}\left(x-\bar{x}\left(\rho_{R}\right)\right)$ everywhere, thus by dominated convergence theorem we know $\| r^{2}\left(x-\bar{x}\left(\rho_{n, R}\right)\right)-r^{2}\left(x-\bar{x}\left(\rho_{R}\right)\right) \|_{L^{4}(\Omega)} \rightarrow 0$. $\rho_{n, R} \rightarrow \rho_{R}$ in $L^{\frac{4}{3}}\left(\mathbb{R}^{3}\right)$ implies weak convergence in $L^{\frac{4}{3}}(\Omega)$, $\left\|\rho_{n, R}\right\|_{L^{3}(\Omega)} \leq C$, then we have $I\left(\rho_{n, R}\right) \rightarrow I\left(\rho_{R}\right)$, and then $\frac{J^{2}}{2\left(I\left(\rho_{n, R}\right)\right)} \rightarrow \frac{J^{2}}{2\left(I\left(\rho_{R}\right)\right)}$. 
		
		Collecting the results above, we obtain $E_{J}\left(\rho_{R}\right) \leq \liminf\limits_{n \rightarrow \infty} E_{J}\left(\rho_{n, R}\right)$.
	}
\end{proof}

\begin{remark}
	Several results can be established using similar arguments:
	\begin{enumerate}[(i)]
	\item The functional $E_J$ is lower semicontinuous with respect to the weak topology on $L^p(\mathbb{R}^3)$ for $\frac{6}{5} < p \leq \infty$.
	\item $E_J$ is continuous in $W_{m,R}$ with respect to the norm topology on $L^p(\mathbb{R}^3)$, where $\frac{6}{5} < p \leq \infty$.
	\item The non‑rotating energy $E_0$ is lower semicontinuous in $W_m$ (not only in $W_{m,R}$) with respect to the norm topology on $L^{\frac{4}{3}}(\mathbb{R}^3)$. The proof follows the same line of reasoning as in the proof of Theorem~\ref{convergence of scaling densities of planets} below.
	\item We note that a different energy functional $E$ is studied in \cite{AB71} for a single‑star model.
	\end{enumerate}
\end{remark}

We then have the existence of ``double constrained'' minimizers, and results about their variational derivatives $E_J^\prime$ (\ref{variational derivative}).
\begin{theorem}[Existence of ``Double Constrained'' Minimizers]\label{existence of ``double constrained'' minimizers}
	Fix $J>0$, given $m>0, R>\frac{384}{\pi \eta^{3}}$, there is a $\rho_{R} \in W_{m, R}$ which minimizes $E_{J}(\rho)$ on $W_{m, R}$. Any such $\rho_{R}$ is continuous on $\Omega=\Omega_{m} \cup \Omega_{1-m}$. Moreover, there are some constants $\lambda_{m, R}$ and $\lambda_{1-m, R}$ such that
	\begin{equation}\label{ELg in binary case}
		E_J^{\prime}\left(\rho_R\right) \geq\left\{\begin{array}{ll}
			\lambda_{m, R}, & \text { where } x \in \Omega_m \text { and } \rho_R(x)<R \\
			\lambda_{1-m, R}, & \text { where } x \in \Omega_{1-m} \text { and } \rho_R(x)<R
		\end{array}\right. 
	\end{equation}
	\begin{equation}\label{ELl in binary case}
		E_J^{\prime}\left(\rho_R\right) \leq\left\{\begin{array}{ll}
			\lambda_{m, R}, & \text { where } x \in \Omega_m \text { and } \rho_R(x)>0 \\
			\lambda_{1-m, R}, & \text { where } x \in \Omega_{1-m} \text { and } \rho_R(x)>0
		\end{array}\right.
	\end{equation}
	
	In particular,
	\begin{equation}\label{EL in binary case}
		E_J^{\prime}\left(\rho_R\right)=\left\{\begin{array}{ll}
			\lambda_{m, R}, & \text { where } x \in \Omega_m \text { and } 0<\rho_R(x)<R \\
			\lambda_{1-m, R}, & \text { where } x \in \Omega_{1-m} \text { and } 0<\rho_R(x)<R
		\end{array}\right.   
	\end{equation}
\end{theorem}

\begin{proof}
	{
		\rm
		Given a minimizing sequence $\left\{\rho_{n, R}\right\}$ with $\lim\limits_{n \rightarrow \infty} E\left(\rho_{n, R}\right)=\inf\limits_{\rho \in W_{m, R}} E(\rho)$. By Interpolation Inequality \cite[Section 4.2]{Bre11} we know $W_{m, R}$ is bounded in $L^{\frac{4}{3}}\left(\mathbb{R}^{3}\right)$. Since $L^{\frac{4}{3}}\left(\mathbb{R}^{3}\right)$ is reflexive, $\exists \rho_{R} \in L^{\frac{4}{3}}\left(\mathbb{R}^{3}\right)$ such that $\rho_{n, R} \rightharpoonup \rho_{R}$ up to subsequence \cite[Theorem 3.18]{Bre11}. By Lemma \ref{l.s.c on weak L^4/3} $W_{m, R}$ is weakly compact in $L^{\frac{4}{3}}\left(\mathbb{R}^{3}\right), \rho_{R} \in W_{m, R}$, and 
		$$\inf\limits_{\rho \in W_{m, R}} E(\rho) \leq E\left(\rho_{R}\right) \leq \lim\limits_{n \rightarrow \infty} E\left(\rho_{n, R}\right)=\inf\limits_{\rho \in W_{m, R}} E(\rho)$$
		It means $\rho_{R}$ is a minimizer. A slight modification of the proof of \cite[Lemma 2]{AB71}, \cite[Section 5]{Che26G1} and \cite[Section 2]{Che26R2}, which means we need to consider the perturbations have compact supports in $\Omega$, shows that (\ref{ELg in binary case}) and (\ref{ELl in binary case}) hold for almost all points which are mentioned in (\ref{ELg in binary case}) (\ref{ELl in binary case}) with suitable $\lambda_{m, R}$ or $\lambda_{1-m, R}$. The continuity of $\rho_{R}$ on $\Omega$ then follows as in the proof of Theorem \ref{non-rotating} (see, e.g., subsection \ref{subsection5.2-Uniform bound of scaling densities in L^{infty}} of this paper, \cite{AB71}, \cite[Section 2]{Che26G1} and \cite[Section 2]{Che26R2}), where the idea is making use of (\ref{ELg in binary case}) (\ref{ELl in binary case}) (\ref{EL in binary case}) and improving the regularity of potential $V_{\rho_{R}}$ and density $\rho_{R}$ in turns (bootstrap method). Then we know (\ref{ELg in binary case}) and (\ref{ELl in binary case}) hold indeed for all points we considered.
	}
\end{proof}

\begin{remark}\label{bootstrap difference}
	Since we consider potential and density alternatively, we call such method the \textit{bootstrap method}. However, another type of bootstrap method, commonly used in fluid dynamics and general relativity, refers to bootstrapping with respect to time. It differs from the regularity bootstrap discussed in this paper.
\end{remark}
We modify the arguments in \cite[Section 6]{AB71} and apply again bootstrap method to show that there is a $\widetilde{R}$ such that the ``double constrained'' (``DC'' in short) minimizer $\rho_{R}$ remains in $W_{m, \widetilde{R}}$ for all $R \geq \widetilde{R}$. Then we argue that $\rho_{R}$ minimizes $E_J(\rho)$ on $W_{m}$ to prove the existence of the constrained minimizers.

\begin{lemma}[Uniform Bound of DC Minimizers in $L^{\frac{4}{3}}$ Space] \label{uniform bound of minimizers in L^4/3 in binary case}
	There is a constant $k_{0}$ such that $\left\|\rho_{R}\right\|_{L^{\frac{4}{3}}} \leq k_{0}$, all $R \geq R_{0}=\frac{384}{\pi \eta^{3}}$.
\end{lemma}

\begin{proof}
	{
		\rm
		We have $\forall R \in\left[R_{0}, \infty\right), \rho_{R_{0}} \in W_{m, R}$, then $E\left(\rho_{R}\right) \leq E\left(\rho_{R_{0}}\right)<\infty$. Then, one can adapt the arguments from the non-rotating case—see, for instance, \cite[Section 2]{Che26R2} or \cite[Section 6]{AB71}—to show that the $L^{\frac{4}{3}}$ norm of $\rho_R$ is actually bounded by a constant that depends only on $E\left(\rho_{R_{0}}\right)$.
	}
\end{proof}

Deducing the uniform boundedness in $L^\infty$ from that in $L^\frac{4}{3}$ can be accomplished in several ways. For instance, one may refer to the methods of Li in \cite[Proposition 1.4]{Li91}, as McCann mentioned in \cite[Section 6]{McC06}. Li’s method \cite{Li91} can be applied to obtain uniform boundedness in $L^\infty$ for $\gamma > \frac{4}{3}$ by choosing an appropriate test function, although it is not first verified that Lagrange multiplier is negative — Li verified this after the $L^\infty$ bound was obtained in \cite{Li91}. 

Here, we present an alternative approach, which is inspired by the method in \cite{AB71}. Both methods rely on the result that velocity $v^{2}(x)=\frac{J^{2} r^{2}(x-\bar{x}(\rho))}{I^{2}(\rho)}$ is sufficiently small. Under assumption $\gamma> \frac{3}{2}$, we first prove that Lagrange multipliers are negative, using the idea from \cite[Lemma 5]{AB71} and \cite[Proposition 2.5]{Li91}.

\begin{lemma}[Negavie Lagrange-multipliers for DC Minimizers] \label{negative Lagrange-multipliers in binary case}
	$\exists m_{0}>0, \forall 0<m<m_{0}, \exists R_{1}>0$, such that $\forall R>R_{1}, \lambda_{m, R}<0$, $\lambda_{1-m, R}<0$
\end{lemma}

\begin{proof}
	{
		\rm
		For $\lambda_{m, R}$, since the volume (measure) of $\Omega_{m}$ is ${Vol}:=\frac{4}{3} \pi\left(\frac{\eta}{4}\right)^{3}=$ $\frac{\pi J^{6}}{48 m^{6}(1-m)^{6}}$, while the mass in $\Omega_{m}$ is $\int_{\Omega_{m}} \rho_{R}\,dx=m$. 
		
		Claim: $\mu\left(\left\{x \in \Omega_{m} \left\lvert\, \rho_{R}(x) \leq \frac{m}{V o l}\right.\right\}\right)>0$, where $\mu$ denotes the Lebesgue measure. In fact, if not, then $\rho_{R}(x)-\frac{m}{V o l}>0$ almost everywhere in $\Omega_{m}$, while $\int_{\Omega_{m}}\left(\rho_{R}(x)-\frac{m}{V o l}\right)\,dx=0$ implies $\rho_{R}(x)-\frac{m}{V o l}=0$ almost everywhere in $\Omega_{m}$ \cite[Section 4.3]{RF10}. Since the measure of $\Omega_{m}$ is not 0, it makes a contradiction. Now we have $\mu\left(\left\{x \in \Omega_{m} \left\lvert\, \rho_{R}(x) \leq \frac{m}{V o l}\right.\right\}\right)>0$, which means there is at least a point $x_{0}$ in $\Omega_{m}$ such that $\rho_{R}\left(x_{0}\right) \leq \frac{m}{V o l}$ and (\ref{ELg in binary case}) holds true. Notice ${Vol} \sim m^{-6}$ as $m\ll1$, i.e., $\exists C_{1}>C_{2}>0$, such that if $m$ is sufficiently small, then $C_{2} m^{-6} \leq {Vol} \leq C_{1} m^{-6}$. Therefore, $\frac{m}{V o l} \sim$ $m^{7}, A^{\prime}\left(\rho_{R}\left(x_{0}\right)\right)=\frac{K \gamma}{\gamma-1} \rho_{R}^{\gamma-1}\left(x_{0}\right) \stackrel{\gamma>\frac{3}{2}}{=} O\left(m^{\frac{7}{3}}\right)$. By construction of $\Omega_{m}$ and $\Omega_{1-m}$, we have $\forall y \in \Omega_{1-m}$, $$\left|x_{0}-y\right| \leq \operatorname{diam}\left(\Omega_{m}, \Omega_{1-m}\right)=\frac{3 \eta}{2}=\frac{3 J^{2}}{2 m^{2}(1-m)^{2}}$$,  $$V_{\rho_{R}}\left(x_{0}\right) \geq\int_{\Omega_{1-m}} \frac{\rho_{R}(y)}{\left|x_{0}-y\right|} \,dy \geq \frac{2 m^{2}(1-m)^{3}}{3 J^{2}}>A^{\prime}\left(\rho_{R}\left(x_{0}\right)\right)$$ when $0<m<m_{1}$ for some $m_{1}$. Let $R>\frac{m_{1}}{{Vol}\left(m_{1}\right)}$, then $\rho_R(x_0)$ that satisfies the above relation exists, and (\ref{ELg in binary case}) holds true thus
		$$
		\begin{aligned}
			\lambda_{m, R} &\leq E_{J}^{\prime}\left(\rho_{R}\right)\left(x_{0}\right)\\
			&=A^{\prime}\left(\rho_{R}\left(x_{0}\right)\right)-V_{\rho_{R}}\left(x_{0}\right)-\frac{J^{2} r^{2}\left(x_{0}-\bar{x}\left(\rho_{R}\right)\right)}{I^{2}(\rho)} \\
			&\leq A^{\prime}\left(\rho_{R}\left(x_{0}\right)\right)-V_{\rho_{R}}\left(x_{0}\right)<0
		\end{aligned}
		$$
		
		Similarly, $\exists m_{2}>0, \forall 0<m<m_{2}, \exists \widetilde{x_{0}} \in \Omega_{1-m}$, such that $\rho_{R}\left(\widetilde{x_{0}} \right) \leq \frac{1-m}{V o l}$, and since $\gamma>\frac{3}{2}$,
		$$V_{\rho_{R}}\left(\widetilde{x_{0}}\right)\geq \frac{2 m^{3}(1-m)^{2}}{3 J^{2}}>A^{\prime}\left(\rho_{R}\left(\widetilde{x_{0}}\right)\right)$$
		
		Let $R>\frac{m_{2}}{{Vol}\left(m_{2}\right)}$, we have $\lambda_{1-m, R}<0$. Set $m_{0}=\min \left\{m_{1}, m_{2}\right\}, R_{1}=\max$ $\left\{\frac{m_{1}}{{Vol}\left(m_{1}\right)}, \frac{m_{2}}{{Vol}\left(m_{2}\right)}\right\}$, then we get the conclusion.
	}
\end{proof}

\begin{lemma}[Moment of Inertia Increasing Rate] \label{moment of inertia increasing rate}
	Fix $J>0$, there exist an $M>0$ and a $C>0$, such that $\forall 0<m<M$ and $\forall R>0$, $\rho(m)=\rho_{1-m}+\rho_{m}$ is in $W_{m, R}$, we have $I\left(\rho_{1-m}+\rho_{m}\right)>{Cm}^{-3}$.
\end{lemma}

\begin{proof}
	{
		\rm
		$\rho(m)=\rho_{1-m}+\rho_{m} \in {W}_{m, R}$. The centers of mass $\bar{x}\left(\rho_{1-m}\right) \in \Omega_{1-m}$, $\bar{x}\left(\rho_{m}\right) \in \Omega_{m}$, as well as their projections onto $z=0$, are separated by at least $\frac{\eta}{2}$, where $\eta=\frac{J^{2}}{\mu_r^{2}}=\frac{J^{2}}{(m(1-m))^{2}}$. Due to Lemma \ref{expansion of MoI}, we have
		$$
		I(\rho(m))=\mu_r r^{2}\left(\bar{x}\left(\rho_{1-m}\right)-\bar{x}\left(\rho_{m}\right)\right)+I\left(\rho_{1-m}\right)+I\left(\rho_{m}\right) \geq \mu_r r^{2}\left(\bar{x}\left(\rho_{1-m}\right)-\bar{x}\left(\rho_{m}\right)\right) \geq \frac{J^{4}}{4 \mu_r^{3}}
		$$
		
		Since $\mu_r=m(1-m)$, we know $\exists C>0$, such that $I(\rho(m))>C m^{-3}$ when $m \rightarrow 0$.
	}
\end{proof}

\begin{corollary}[Velocity Decreasing Rate]\label{velocity decreasing rate}
	Fix $J>0$, there exist an $M>0$ and a $C>0$, such that $\forall x \in \Omega_{m} \cup \Omega_{1-m}$, $0<m<M$, $R>0$ and $\rho(m)=\rho_{1-m}+\rho_{m}$ is in $W_{m, R}$, we have $v(x, m)=\frac{J r(x-\bar{x}(\rho(m)))}{I(\rho(m))}<C m$.
\end{corollary}

\begin{proof}
	{
		\rm
		By construction we know $\bar{x}(\rho(m)) \in \text{conv} (\Omega_{m} \cup \Omega_{1-m})$, where $\text{conv}$ menas the convex hull. Moreover, $r(x-\bar{x}(\rho(m))) \leq$ $\frac{3}{2} \eta$ for all $x \in \Omega_{m} \cup \Omega_{1-m}$. Thus, when $m$ small enough, by Lemma \ref{moment of inertia increasing rate} we have $v(x, m)=\frac{J r(x-\bar{x}(\rho(m)))}{I(\rho(m))}<C m$ for some $C>0$.
	}
\end{proof}

\begin{theorem}[Uniform Bound for DC Minimizer in $L^{\infty}\left(\mathbb{R}^{3}\right)$] \label{uniform bound of minimizers in L^infty in binary case}
	Fix $J>0$, $\exists M>0$, such that $\forall 0<m<M, \exists k_{1}>0, \forall R>0$ and $\rho_{R}$ is a ``double constrained''
	minimizer in $W_{m, R}$, we have $\left\|\rho_{R}\right\|_{L^{\infty}\left(\mathbb{R}^{3}\right)} < k_{1}$. In particular, $E_J(\rho_R)=E_J(\rho_{k_1})$ for $R \geq k_1$.
\end{theorem}

\begin{remark}
	Here we assume that the inequality $\left\|\rho_{R}\right\|_{L^{\infty}\left(\mathbb{R}^{3}\right)} < k_{1}$ holds automatically if $W_{m, R}$ is empty. But due to Remark \ref{non-empty of W_{m, R}}, $W_{m, R}$ is not empty for large $R$.
\end{remark}

\begin{proof}[Proof of Theorem \ref{uniform bound of minimizers in L^infty in binary case}]
	{
		\rm
		We want to apply bootstrap method. Thanks to Corollary \ref{velocity decreasing rate}, when $m$ is sufficiently small, $\exists C_{1}>0$, such that $\forall x \in \Omega=\Omega_{m} \cup \Omega_{1-m}$, $v^{2}(x) \leq C_{1} m^{2}$. Moreover, similarly as above if $x \in \Omega$, we have $\forall y \in \Omega_{1-m}$, $|x-y| \leq \operatorname{diam}\left(\Omega_{m} \cup \Omega_{1-m}\right)=\frac{3 \eta}{2}=\frac{3 J^{2}}{2 m^{2}(1-m)^{2}}$, and
		$$V_{\rho_{R}}(x) \geq\int_{\Omega_{1-m}} \frac{\rho_{R}(y)}{|x-y|} \,dy \geq \frac{2 m^{2}(1-m)^{3}}{3 J^{2}}>C_{2} m^{2}$$ 
		for some $C_{2}>0$ when $m$ is small enough. Therefore, we know $\exists \widetilde{C}>0$, such that $\forall x \in \Omega$, $v^{2}(x) \leq \widetilde{C} V_{\rho_{R}}(x)$. By (\ref{ELl in binary case}), and Lemma \ref{negative Lagrange-multipliers in binary case} (Lagrange-multipliers are less than 0), when $R>R_{1}$, we have if $\rho_{R}(x)>0$, then $A^{\prime}\left(\rho_{R}(x)\right) \leq V_{\rho}(x)+v^{2}(x) \leq C V_{\rho}(x)$. If $\rho_{R}(x)=0$, then $0=$ $A^{\prime}\left(\rho_{R}(x)\right) \leq C V_{\rho_{R}}(x)$ is obvious.
		
		Thanks to Lemma \ref{uniform bound of minimizers in L^4/3 in binary case}, we know if $R>R_{0}+R_{1},\left\|\rho_{R}\right\|_{L^{\frac{4}{3}}} \leq k_{0}$. By HardyLittlewood-Sobolev Inequality (Proposition \ref{HLSI}), with an abuse of notation $C$, we have $\left\|V_{\rho_{R}}\right\|_{L^{12}\left(\mathbb{R}^{3}\right)} \leq$ $C\left\|\rho_{R}\right\|_{L^{\frac{4}{3}}\left(\mathbb{R}^{3}\right)} \leq C$.   And then $A^{\prime}\left(\rho_{R}\right)=\frac{K \gamma}{\gamma-1} \rho_{R}^{\gamma-1} \leq C V_{\rho}$ implies 
		$$\left\|\rho_{R}\right\|_{L^{12(\gamma-1)}}=\left(\int_{\mathbb{R}^3} \rho_{R}{ }^{12(\gamma-1)}(x) \,dx\right)^{\frac{1}{12(\gamma-1)}} \leq C\left(\int_{\mathbb{R}^3} V_{\rho_{R}}^{12}(x) \,dx\right)^{\frac{1}{12(\gamma-1)}}=C\left\|V_{\rho_{R}}\right\|_{L^{12}}^{\frac{1}{\gamma-1}} \leq C$$
		Since we know $\gamma > \frac{3}{2}>\frac{4}{3}$ and $\Omega$ is bounded, by interpolation inequality \cite[Section 4.2]{Bre11} we have $\left\|\rho_{R}\right\|_{L^{6}} \leq C$. Since we also know $\left\|\rho_{R}\right\|_{L^{1}\left(\mathbb{R}^{3}\right)}=1$ are bounded, thanks to Proposition \ref{diff. of poten.}, we have $V_{\rho_{R}}$ is continuous and $\left\|V_{\rho_{R}}\right\|_{L^{\infty}\left(\mathbb{R}^{3}\right)}$ is bounded uniformly. Again by $\rho_{R}{ }^{\gamma-1} \leq C V_{\rho_{R}}$ we know $\left\|\rho_{R}\right\|_{L^{\infty}\left(\mathbb{R}^{3}\right)} \leq \hat{C}$ for some $\hat{C}>0$ when $R>R_{1}+R_{0}$. When $R \leq R_{1}+R_{0}$, by construction of $W_{m, R}$ we know $\left\|\rho_{R}\right\|_{L^{\infty}\left(\mathbb{R}^{3}\right)} \leq R_{1}+R_{0}$. Pick $k_{1}= \max \left\{\hat{C}, R_{1}+R_{0}\right\} +1$ and then we get the result. For $R\geq k_1$, $E_J(\rho_R)=E_J(\rho_{k_1})$ comes from the fact that they both minimize $E_J$ on $W_{m, k_1}$.
	}
\end{proof}

\begin{remark}\label{rho bounded by V_rho}
	One can observe in $A^{\prime}\left(\rho_{R}(x)\right) \leq C V_{\rho_{R}}(x)$, such $C$ is actually not only independent of $R$ but also a uniform constant when $0<m<m_{0}$ for some $m_{0}$. This inequality also holds for constrained minimizer $\rho$, not only for a double constrained minimizer. The existence of a constrained minimizer follows from Theorem \ref{existence of the constrained minimizers in binary case} below.
\end{remark}

Given any $\sigma \in W_{m} \subset {L^1(\mathbb{R}^3)\cap L^\frac{4}{3}(\mathbb{R}^3)}$, the $L^\infty$ norm of $\sigma$ may be unbounded, hence in order to show the ``double constrained minimizer" $\rho_R$ is indeed a constrained minimizer for $R$ large enough, we need to prove that energy of $\sigma$ still exceeds energy of $\rho_R$, which was not addressed in McCann's work \cite{McC06}. We provide a supplementary proof for it by adapting the arguments in the single-star case \cite[Section 6]{AB71}.

\begin{theorem}[Existence of the Constrained Minimizers]\label{existence of the constrained minimizers in binary case}
	Fix $J>0$, $\exists m_{0}>0$, $\forall 0<m<m_{0}$, there is a $\rho \in W_{m}$ which minimizes $E_{J}(\rho)$ on $W_{m}$. Moreover, there is a constant $R(m)$ such that for any constrained minimizer $\rho$, $\rho$ is continuous on $\Omega=\Omega_{m} \cup \Omega_{1-m}$ and $\rho<R(m)$, that is $\rho \in W_{m, R(m)}$.
\end{theorem}

\begin{proof}
	{
		\rm
		The facts that $\rho$ is continuous on $\Omega$ and $\rho<R(m)$ can be handled similarly to the case for ``double constrained" minimizer, see Theorem \ref{existence of ``double constrained'' minimizers}, Theorem \ref{uniform bound of minimizers in L^infty in binary case}, and Remark \ref{rho bounded by V_rho}. 
		
		To prove the existence of constrained minimizer, let $k_1$ be given in Theorem \ref{uniform bound of minimizers in L^infty in binary case}, we have $ E_{J}\left(\rho_{R}\right)=E_{J}\left(\rho_{k_{1}}\right)$ for $R\geq k_1$ due to Theorem \ref{uniform bound of minimizers in L^infty in binary case}. We will prove $\rho_{k_1}$ is a constrained minimizer. First we note that $\rho_{k_1}\in W_{m,k_1} \subset L^1(\mathbb{R}^3)\cap L^\infty(\mathbb{R}^3)$, then $E_J(\rho_k)<\infty$ due to Remark \ref{finite energy under L infty}. Given any $\sigma \in W_{m} \subset {L^1(\mathbb{R}^3)\cap L^\frac{4}{3}(\mathbb{R}^3)}$, due to Remark \ref{finite gravitational interaction energy} we know $G(\sigma,\sigma)<\infty$. For $\sigma$ with $U(\sigma) =\int_{\mathbb{R}^3} A(\sigma(x)) \,dx=\infty$, we know $E_J(\rho_{k_1})<E_J(\sigma)=\infty$. For $\sigma$ with $U(\sigma) <\infty$, let 
		
		$$
		\widetilde{P_{R}}(\sigma)=\left\{\tau \in L^{\infty}\left(\mathbb{R}^{3}\right) \left\lvert\, \begin{array}{ll}
			\tau(x)=0, & \text { where } x \text { satisfies } \sigma(x)>R \text { or } x \notin \Omega \\
			\tau(x) \geq 0, & \text { where } x \text { satisfies } \sigma(x)<R^{-1}
		\end{array}\right.\right\}
		$$
		
		When $R$ is large enough, we can find a $\tau \in \widetilde{P_{R}}(\sigma)$, with $\tau \geq 0$ and $\int_{\Omega_{m}} \tau\,dx=\int_{\Omega_{1-m}} \tau\,dx=1$. (See e.g. \cite[Section 4]{Che26G1}.) Let $\sigma_{s}^{[1]}(x), \sigma_{s}^{[2]}(x)$ and $\sigma_{s}$ be
		$$\sigma_{s}^{[1]}(x)=\left\{\begin{array}{ll}
			\sigma(x), & \text { if } x \in \Omega \text { and } \sigma(x)<\frac{1}{2} s \\
			0, & \text { otherwise }
		\end{array}\right.$$
		$$\sigma_{s}^{[2]}(x)=\left\{\begin{array}{ll}
			\left(m-\int_{\Omega_{m}} \sigma_{s}^{[1]}\,dx\right) \tau, & \text { if } x \in \Omega_{m} \\
			\left((1-m)-\int_{\Omega_{1-m}} \sigma_{s}^{[1]}\,dx\right) \tau, & \text { if } x \in \Omega_{1-m} \\
			0, & \text { otherwise }
		\end{array}\right.$$
		$$
		\sigma_{s}=\sigma_{s}^{1}+\sigma_{s}^{2}
		$$
		
		By construction we have $\int_{\Omega_{m}} \sigma_s\,dx=m$ and $ \int_{\Omega_{1-m}} \sigma_s\,dx=1-m$, then $\sigma_{s} \in W_{m}$. Similar again to \cite[Section 4]{Che26G1}, we have $\int_{\Omega_{m}} \sigma_{s}^{[1]}\,dx \rightarrow \int_{\Omega_{m}} \sigma\,dx=m$, $\int_{\Omega_{1-m}} \sigma_{s}^{[1]}\,dx \rightarrow \int_{\Omega_{1-m}} \sigma\,dx=1-m$, therefore $\sigma_{s}^{[2]} \rightarrow 0$ almost everywhere in $\mathbb{R}^3$ as $s \rightarrow \infty$ due to $\tau \in L^{\infty}$. $\sigma \in W_{m}$ implies $\sigma$ is finite almost everywhere, thus $\sigma_{s}^{[1]}$ increases to $\sigma$ almost everywhere as $s \rightarrow \infty$. Then we know when $s \rightarrow \infty$, $\sigma_{s} \rightarrow \sigma$ a.e. in $\mathbb{R}^{3}$. Moreover, for almost every point $x \in\left\{x \in \mathbb{R}^{3} \mid \tau(x)=0\right\}$, $ A\left(\sigma_{s}\right)=A\left(\sigma_{s}^{[1]}\right)$ increases to $A(\sigma)$ (except at points where $\tau=\infty$). For $\Omega_{>0} \coloneq \{x\in \mathbb{R}^3\mid \tau(x)>0\}$ (a subset of $\Omega$ thus bounded), $\sigma_{s} \rightarrow \sigma$ implies $A\left(\sigma_{s}\right) \rightarrow A(\sigma)$ a.e. as $s \rightarrow \infty$, and we also know $\left|\sigma_{s}\right| \leq\left|\sigma_{s}^{[1]}\right|+\left|\sigma_{s}^{[2]}\right| \leq \sigma+\|\tau\|_{L^{\infty}}$ a.e., which implies $A(\sigma_s)\leq A(\sigma+\|\tau\|_{L^{\infty}})$ a.e. We claim $A(\sigma+\|\tau\|_{L^{\infty}})\in L^1(\Omega_{>0})$. In fact, since $A(\rho)=\frac{K}{\gamma-1}\rho^\gamma$ is convex, we know 
		$$
		A(\sigma+\|\tau\|_{L^{\infty}})=A(\frac{2\sigma+2\|\tau\|_{L^{\infty}}}{2})\leq \frac{A(2\sigma)+A(2\|\tau\|_{L^{\infty}})}{2}
		$$
		Note $U(\sigma) =\int_{\mathbb{R}^3} A(\sigma(x)) \,dx<\infty$ implies $A(2\sigma)\in L^1(\mathbb{R}^3)$, and $\Omega_{>0}$ is bounded implies $A(2\|\tau\|_{L^{\infty}})\in L^1(\Omega_{>0})$. Therefore, we know $A(\sigma+\|\tau\|_{L^{\infty}})\in L^1(\Omega_{>0})$.
		
		We can then split $\mathbb{R}^3$ into $\mathbb{R}^3 \setminus \Omega_{>0}$ and $\Omega_{>0}$, and show $\int_{\mathbb{R}^3} A\left(\sigma_{s}\right)\,dx \rightarrow \int_{\mathbb{R}^3} A(\sigma)\,dx$ as $s \rightarrow \infty$ by the monotone convergence theorem and dominated convergence theorem. Since $\Omega$ is bounded, by Hölder's inequality we know $\|\tau\|_{L^{\frac{4}{3}}(\Omega)} \leq C\|\tau\|_{L^{\infty}(\Omega)}$, then $$\left\|\sigma_{s}\right\|_{L^{\frac{4}{3}}(\Omega)} \leq\left\|\sigma_{s}^{[1]}\right\|_{L^{\frac{4}{3}}(\Omega)}+\left\|\sigma_{s}^{[2]}\right\|_{L^{\frac{4}{3}}(\Omega)} \leq\|\sigma\|_{L^{\frac{4}{3}}(\Omega)}+\|\tau\|_{L^{\frac{4}{3}}(\Omega)} \leq\|\sigma\|_{L^{\frac{4}{3}}(\Omega)}+C\|\tau\|_{L^{\infty}(\Omega)}$$
		It means $\left\{\sigma_{s}\right\}$ is a bounded set in $L^{\frac{4}{3}}(\Omega)$. Since $L^{\frac{4}{3}}(\Omega)$ is reflexive, thanks to \cite[Theorem 3.18]{Bre11}, there is a sequence $\left\{\sigma_{s_{n}}\right\}$ which converges weakly to a weak limit $\widetilde{\sigma}$. Since we know $\sigma_{s_{n}} \rightarrow \sigma$ a.e., by \cite{To20}, we have $\widetilde{\sigma}=\sigma$. That is, $\sigma_{s_{n}} \rightharpoonup \sigma$ in $L^{\frac{4}{3}}(\Omega)$, therefore, we can apply what we did in the proof of Lemma \ref{l.s.c on weak L^4/3}, and obtain $G\left(\sigma_{s_{n}}, \sigma_{s_{n}}\right) \rightarrow G(\sigma, \sigma)$, $T_{J}\left(\sigma_{s_{n}}\right) \rightarrow T_{J}(\sigma)$. Collecting the results above, we have $E_{J}(\sigma)=\lim\limits_{n \rightarrow \infty} E_{J}\left(\sigma_{s_{n}}\right)$. By the same arguments, we know for any sequence of $\left\{\sigma_{s_{m}}\right\}$, we can always pick a subsequence $\left\{\sigma_{s_{m_{k}}}\right\}$ which also satisfies $E_{J}(\sigma)=\lim\limits_{k \rightarrow \infty} E_{J}\left(\sigma_{s_{m_{k}}}\right)$. Thus, we get, $E_{J}(\sigma)=\lim\limits_{s \rightarrow \infty} E_{J}\left(\sigma_{s}\right)$ (one can show it by contradiction argument).
		
		Since for almost every $x \in \mathbb{R}^{3}$, $\left|\sigma_{s}\right| \leq\left|\sigma_{s}^{[1]}\right|+\left|\sigma_{s}^{[2]}\right|<\frac{1}{2} s+\|\sigma\|_{L^{\infty}}$. For $s> \max \{2\|\sigma\|_{L^{\infty}}, k_1\}$, we have $\left|\sigma_{s}\right|<s$ a.e., thus $\sigma_{s} \in W_{m, s}$, and then $E_{J}\left(\sigma_{s}\right) \geq E_{J}\left(\rho_{s}\right)=E_{J}\left(\rho_{k_{1}}\right)$ due to Theorem \ref{uniform bound of minimizers in L^infty in binary case}. Let $s\rightarrow 0$, then we know $E_{J}(\sigma) \geq E_{J}\left(\rho_{k_{1}}\right)=E\left(\rho_{R}\right)$ for $R \geq k_{1}$ by Theorem \ref{uniform bound of minimizers in L^infty in binary case}. Since we choose $\sigma\in W_m$ arbitrarily, we know $\rho_{R}$ is a constrained minimizer (i.e. $\rho_{R}$ minimizes $E_J$ on $W_{m}$) for $R \geq k_{1}$. In particular, we find that $\rho_{k_{1}}$ is a constrained minimizer. This gives the result. 
	}
\end{proof}

\begin{remark}
	In the proof of Theorem \ref{existence of the constrained minimizers in binary case}, the verification that $A(\sigma+|\tau|{L^{\infty}})\in L^1(\Omega{>0})$ relies only on the convexity of $A$, not on its specific form. Consequently, the same argument is applicable to more general pressure. For the polytropic case studied in this paper, one may also use the following estimate:
	\begingroup
	\small
		\begin{align*}
		A(\sigma+\|\tau\|_{L^{\infty}}) 
		&= \frac{K}{\gamma-1}(\sigma+\|\tau\|_{L^{\infty}})^\gamma \\
		&\leq \frac{K}{\gamma-1}(2\max\{\sigma, \|\tau\|_{L^{\infty}}\})^\gamma \\
		&= \frac{K2^\gamma}{\gamma-1} \max\{\sigma^\gamma, \|\tau\|_{L^{\infty}}^\gamma\} \\
		&\leq \frac{K2^\gamma}{\gamma-1}(\sigma^\gamma + \|\tau\|_{L^{\infty}}^\gamma) \\
		&= 2^\gamma A(\sigma) + 2^\gamma A(\|\tau\|_{L^{\infty}})
		\end{align*}
	\endgroup
\end{remark}

\begin{remark}
	Following the statement in Section \ref{section5-Bound of the Supports of Density Functions}, one can show that the constant $R(m)$ is actually a uniform constant for all $m\in (0, m_{0})$ for some small $m_0$.
\end{remark}

%
%
%

\section{Bound of the Supports of Density Functions}\label{section5-Bound of the Supports of Density Functions}

In this section we estimate a uniform bound for the supports of the density functions, under small mass ratio assumption which is physically desirable in star-planet systems. The analysis focuses on the planet’s density. For the star's density, the estimates are similar or even simpler, and can be adapted directly from the arguments in \cite{McC06} and do not require the scaling method in subsection \ref{subsection5.1-Convergence of scaling densities in L^{4/3} and L^1} and subsection \ref{subsection5.2-Uniform bound of scaling densities in L^{infty}}. With these uniform bounds we will show in Section \ref{section6-Existence for Star-Planet Systems} that the minimizers can be located in the interior of the domains we define.

\subsection{Convergence of scaling densities in $L^{\frac{4}{3}}\left( \mathbb{R}^{3} \right)$ and $L^{1}\left( \mathbb{R}^{3} \right)$}\label{subsection5.1-Convergence of scaling densities in L^{4/3} and L^1}

As the mass ratio $m$ goes to 0, the energy of planet $\rho_m$ goes to 0 (see Remark \ref{reason of scaling}), hence the concentration-compactness lemma in \cite{Lio84} cannot be directly applied. Therefore, we introduce scaling method. Given the constrained minimizers $\rho(m)=\rho_m+\rho_{1-m}$ of $E_J(\rho)$, we show that, after suitable scaling and translation, the scaling planet densities $\widetilde{\rho_m}$, which have mass $1$, will converge to some limit function as $m \rightarrow 0$ by showing the energies will converge. Intuitively, when $m \rightarrow 0$, the distances between star and planet become large. As a result, the gravitational influence they exert on each other becomes relatively weak, even after scaling. Therefore, the limit function turns out to be minimizer of $E_0(\rho)$ with respect to mass $1$, as we will prove later.

\begin{definition}[{Scaling Density for Planet}]\label{scaling density for planet}
	{	\rm
		For any given planet's mass $m>0$, we define \textit{scaling density} $\widetilde{\rho_{m}}$ as following
		
		\begin{equation}\label{scaling density for planet expression}
			\widetilde{\rho_{m}}(x) = A\rho_{m}\left( {Bx} \right)
		\end{equation}
		where $\rho_m$ is the planet's density of the constrained minimizer $\rho(m) = \rho_{1 - m} + \rho_{m}$ in Section \ref{section4-Existence of Constrained Minimizers}, $A$,$B$ are the same coefficients introduced in Theorem \ref{scaling relation for energy}, that is, $A(m)=m^{- \frac{2}{3\gamma - 4}}$, $B(m)=m^{\frac{\gamma - 2}{3\gamma - 4}}$.
	}
\end{definition}

\begin{remark}\label{order of gravitational potential energy}

	Similar to the scaling relations in Theorem \ref{scaling relation for energy} and its proof, we get $\widetilde{\rho_{m}}$ has mass $1$, and $E_{0}\left( \rho_{m} \right) = \frac{B^{3}}{A^{\gamma}}E_{0}\left( \widetilde{\rho_{m}} \right)$.
	
	Moreover, we have
	\begin{equation} \label{expandsion and compare of E_J}
		\begin{aligned}
			&\quad E_{J}\left( {\rho_{1 - m} + \rho_{m}} \right) \\
			&= {\int_{\mathbb{R}^3}{\frac{K{\left( {\rho_{1 - m}(x) + \rho_{m}(x)} \right)}^{\gamma}}{\gamma - 1}\,dx}} - \frac{1}{2}{\iint_{\mathbb{R}^3\times \mathbb{R}^3}{\frac{\rho(m)(x) \cdot \rho(m)(y)}{\left| {x - y} \right|}\,dx\,dy}} + \frac{J^{2}}{2\left( {I\left( {\rho_{1 - m} + \rho_{m}} \right)} \right)} \\
			&= E_{0}\left( \rho_{1 - m} \right) + E_{0}\left( \rho_{m} \right) - G\left( {\rho_{1 - m},\rho_{m}} \right) + T_{J}\left( {\rho_{1 - m} + \rho_{m}} \right) \\
			&= E_{0}\left( \rho_{1 - m} \right) + \frac{B^{3}}{A^{\gamma}}E_{0}\left(\widetilde{\rho_{m}} \right) - G\left( {\rho_{1 - m},\rho_{m}} \right) + T_{J}\left( {\rho_{1 - m} + \rho_{m}} \right) \\
			&> E_{0}\left( \rho_{1 - m} \right) + \frac{B^{3}}{A^{\gamma}}E_{0}\left(\widetilde{\rho_{m}} \right) - G\left( {\rho_{1 - m},\rho_{m}} \right)
		\end{aligned}
	\end{equation}
	
	Notice since $\rho_{1 - m} \in \Omega_{1 - m}$, $\rho_{m} \in \Omega_{m}$, whose distance is at least $\frac{\eta}{2}=\frac{J^2}{2m^2(1-m)^2}$ (\ref{dist}), we obtain $G\left( {\rho_{1 - m},\rho_{m}} \right) \leq \frac{2m\left( {1 - m} \right)}{\eta} < Cm^{3}$ for some $C>0$.
\end{remark}

\vspace*{\baselineskip} 
Our goal is to estimate $E_{0}\left( \widetilde{\rho_m} \right)$. But before doing it, we first show the moment of inertia $I\left( \rho_{1 - m} + \rho_{m} \right)$ will be larger, thus kinetic energy $T_J$ will be small.

\begin{lemma}[{Moment of Inertia Increasing Rate as Mass Decreases}]\label{moment of inertia increasing rate as mass decreases}
	Fix $J>0$, there exist an $M>0$ and a $C>0$, such that $\forall m < M$ and $\rho(m) = \rho_{1 - m} + \rho_{m}$ is the constrained minimizer, we have  $I\left( \rho_{1 - m} + \rho_{m} \right) > Cm^{- 3}$.
\end{lemma}

\begin{proof}
	{\rm
		The proof is similar to the proof of Lemma \ref{moment of inertia increasing rate}. 
	}
\end{proof}

\begin{corollary}[Kinetic Energy Decreasing Rate as Mass Decreases] \label{kinetic energy decreasing rate}
Fix $J>0$, there exist an $M>0$ and a $C>0$, such that $\forall m<M$ and $\rho(m)=\rho_{1 - m}+\rho_m$ is the constrained minimizer, we have $T_J (\rho(m))<Cm^3$.
\end{corollary}
\begin{proof}
{\rm
	The result can be obtained from Lemma \ref{moment of inertia increasing rate as mass decreases} and the definition of $T_J (\rho(m))=\frac{J^2}{2I(\rho(m))}$.}
	\end{proof}
	
	Not only the kinetic energy, we can also estimate velocity $v(x,m)=\frac{Jr(x-\overline{x}({\rho(m)}))}{I(\rho(m))}$, which will be useful later.
	
\begin{corollary}[Velocity Decreasing Rate as Mass Decreases]\label{velocity decreasing rate as mass decreases}
Fix $J>0$, there exist an $M>0$ and a $C>0$, such that $\forall x \in \Omega_m \cup \Omega_{1 - m}$, $m<M$ and $\rho(m)=\rho_{1-m}+\rho_m$ is the constrained minimizer, we have $v(x,m)<Cm$.
\end{corollary}
\begin{proof}
{\rm
	The proof is similar to the proof of Lemma \ref{velocity decreasing rate}.
}
\end{proof}

\begin{remark}
One can show the Lemma \ref{moment of inertia increasing rate as mass decreases}, Corollary \ref{kinetic energy decreasing rate} and Corollary \ref{velocity decreasing rate as mass decreases} actually hold true for any $\rho \in W_m$.
\end{remark}

\vspace*{\baselineskip} 
To find the behavior of $E_{0}\left( \widetilde{\rho_{m}} \right)$ when $m\rightarrow 0$, we compare $E_J (\rho(m))$ with energy $E_J (\varrho(m))$, where
\begin{equation}
\varrho(m)(x) = \varrho_{1 - m}(x) + \varrho_{m}(x) = \sigma_{1 - m}\left( {x - y_{1 - m}} \right) + \sigma_{m}\left( x - y_{m} \right)
\end{equation}
$y_{1-m}$ and $y_m$ are the centers of $\Omega_{1-m}$ and $\Omega_m$. $\sigma_{1-m}$ and $\sigma_m$ are the non-rotating minimizers in Theorem \ref{non-rotating} with mass $(1-m)$ and $m$. In particular, $E_0 (\varrho_m )=E_0 (\sigma_m )\leq E_0 (\rho_m)$, since translations do not change the non-rotating energy. Similarly, $E_0 (\varrho_{1-m} )=E_0 (\sigma_{1-m} )\leq E_0 (\rho_{1-m})$. 

\begin{remark}\label{energy scaling}
By the scaling method and uniqueness of the minimizer, we know $\sigma_{m}(x) = \frac{1}{A}\sigma\left( \frac{1}{B}x \right)$ (see Theorem \ref{scaling relation for energy}), where $A = m^{- \frac{2}{3\gamma - 4}}$, $B = m^{\frac{\gamma - 2}{3\gamma - 4}}$ are the same as before, $\sigma$ is the non-rotating minimizer with mass 1. Similarly, $\sigma_{1 - m} = \frac{1}{\widetilde{A}}\sigma\left( \frac{1}{\widetilde{B}}x \right)$, where $\widetilde{A} = (1 - m)^{- \frac{2}{3\gamma - 4}}$, $\widetilde{B} = (1 - m)^{\frac{\gamma - 2}{3\gamma - 4}}$. Moreover, $\frac{B^{3}}{A^{\gamma}}E_{0}(\sigma) = E_{0}\left( \sigma_{m} \right) = E_{0}\left( \varrho_{m} \right)$, and $\frac{{\widetilde{B}}^{3}}{{\widetilde{A}}^{\gamma}}E_{0}(\sigma) = E_{0}\left( \sigma_{1 - m} \right) = E_{0}\left( \varrho_{1 - m} \right)$.
\end{remark}

\begin{proposition}[Energy Converges to Non-Rotating Minimum]\label{energy converges to non-rotating minimum}
Fix $J>0$, there exist a $M>0$ and a $C>0$, such that $\forall m<M$ and $\rho(m)=\rho_{1-m}+\rho_m$ is the constrained minimizer, we have
\begin{equation}\label{scaling density approaches non-rotating density}
	\begin{split}
		e_{0} \leq E_{0}\left( \widetilde{\rho_{m}} \right) < e_{0} + Cm^{\frac{4\gamma - 6}{3\gamma - 4}}
	\end{split}
\end{equation}
where $e_0$ is the infimum of $E_0 (\rho)$ with mass 1, that is, $e_0=E_0 (\sigma)$. Since $\gamma>\frac{3}{2}$, in particular we have ${\lim\limits_{m\rightarrow 0}{E_{0}\left( \widetilde{\rho_{m}} \right)}} = e_{0}$.
\end{proposition}

\begin{proof}
{	\rm
	Thanks to Remark \ref{order of gravitational potential energy}, we know $\widetilde{\rho_m}$ has mass $1$, hence we know $e_{0} \leq E_{0}\left( \widetilde{\rho_{m}} \right) $. By Theorem \ref{non-rotating} (v), Theorem \ref{scaling relation for energy} and Remark \ref{asymptotic behaviour of radius and norm in single star case}, we know $\sigma_{1-m}$ and $\sigma_m$ are supported in balls with same constant radius, while the radii of $\Omega_{1-m}$ and $\Omega_m$ go to $\infty$ when $m\rightarrow 0$. Take $m$ large enough such that $\varrho_{1-m} (x)$ and $\varrho_m (x)$ are supported in $\Omega_{1-m}$ and $\Omega_m$, thus $\varrho(m)\in W_m$, and $E_J (\varrho(m))\geq E_J (\rho(m))$ since $\rho(m)$ is the constrained minimizer. Similar to (\ref{expandsion and compare of E_J}), we consider its energy decomposition: 
	\begin{equation}\label{energy decomposition}
		E_{J}\left( {\varrho_{m} + \varrho_{1 - m}} \right) - E_{0}\left( \varrho_{1 - m} \right) - E_{0}\left( \varrho_{m} \right) = - G\left( {\varrho_{1 - m},\varrho_{m}} \right) + T_{J}\left( {\varrho_{1 - m} + \varrho_{m}} \right)
	\end{equation}
	Again, by Theorem \ref{non-rotating}, $\varrho_{1-m}$ and $\varrho_m$ have some nice properties, in particular they are spherically symmetric and radially decreasing after translations. Therefore, by comparison with the point masses, we get $G\left( {\varrho_{1 - m},\varrho_{m}} \right) = \mu_r\eta^{- 1} = \frac{\mu_r^{3}}{J^{2}}$ (one can verify this by using Legendre polynomials or refer to \cite{McC06} \cite[Section 2]{Che26R2}), and $T_{J}\left( {\varrho_{1 - m} + \varrho_{m}} \right) \leq \frac{\mu_r^{3}}{2J^{2}}$ similar to that in Corollary \ref{kinetic energy decreasing rate}. Thus, the right side of (\ref{energy decomposition}) is less than 0. This observation, together with Remark \ref{energy scaling} gives us that: 
	\begin{equation*}
		\begin{aligned}
			E_{0}\left( \varrho_{1 - m} \right) + \frac{B^{3}}{A^{\gamma}}E_{0}(\sigma) 
			&= E_{0}\left( \varrho_{1 - m} \right) + E_{0}\left( \varrho_{m} \right) \\
			&> E_{J}\left( {\varrho_{m} + \varrho_{1 - m}} \right) \\
			&\geq E_{J}\left( {\rho_{m} + \rho_{1 - m}} \right) \\
			&> E_{0}\left( \rho_{1 - m} \right) + \frac{B^{3}}{A^{\gamma}}E_{0}\left( \widetilde{\rho_{m}} \right) - G\left( {\rho_{1 - m},\rho_{m}}\right) \\
			&\geq E_{0}\left( \varrho_{1 - m} \right) + \frac{B^{3}}{A^{\gamma}}E_{0}\left( \widetilde{\rho_{m}} \right) - G\left( {\rho_{1 - m},\rho_{m}} \right)
		\end{aligned}
	\end{equation*}
	Therefore, we get $\frac{B^{3}}{A^{\gamma}}E_{0}\left( \widetilde{\rho_{m}} \right) < G\left( {\rho_{1 - m},\rho_{m}} \right) + \frac{B^{3}}{A^{\gamma}}E_{0}(\sigma)$ (i.e. $E_{0}\left( \widetilde{\rho_{m}} \right) < \frac{A^{\gamma}}{B^{3}}G\left( {\rho_{1 - m},\rho_{m}} \right) + E_{0}(\sigma)$). Since $\frac{A^{\gamma}}{B^{3}} = m^{\frac{6 - 5\gamma}{3\gamma - 4}}$, $G\left( {\rho_{1 - m},\rho_{m}} \right) \leq Cm^{3}$ due to Remark \ref{order of gravitational potential energy}, we have
	\begin{equation*}
		E_{0}\left( \widetilde{\rho_{m}} \right) < Cm^{\frac{4\gamma - 6}{3\gamma - 4}} + E_{0}(\sigma)
	\end{equation*}
	as we expect.
}
\end{proof}

Finally, we can show the convergence of $\left\{ \widetilde{\rho_{m}} \right\}$ in $L^{\frac{4}{3}}\left( \mathbb{R}^{3} \right)$ and $L^{1}\left( \mathbb{R}^{3} \right)$.
\begin{theorem}[Convergence of Scaling Densities of Planets in $L^{\frac{4}{3}}\left( \mathbb{R}^{3} \right)$ and $L^{1}\left( \mathbb{R}^{3} \right)$] \label{convergence of scaling densities of planets}
Fix $J>0$, $m_0>0$ which is given in Theorem \ref{existence of the constrained minimizers in binary case}. Let $0<m<m_0$, and $\rho(m)=\rho_{1-m}+\rho_m$ is the constrained minimizer on $W_m$, $\widetilde{\rho_{m}} = A\rho_{m}\left( {Bx} \right)$ as in Definition \ref{scaling density for planet}, we have, up to translation, $\{\widetilde{\rho_{m}}\}$ converges strongly to $\sigma$ in $L^{\frac{4}{3}}\left( \mathbb{R}^{3} \right)$ and $L^{1}\left( \mathbb{R}^{3} \right)$ as $m \rightarrow 0$. Here $\sigma$ is the non-rotating minimizer with mass 1.
\end{theorem}

\begin{proof}
{\rm
	Due to Proposition \ref{energy converges to non-rotating minimum}, we know ${\lim\limits_{m\rightarrow 0}{E_{0}\left( \widetilde{\rho_{m}} \right)}} = e_{0} < 0$. Moreover, we can apply the concentration-compactness lemma of Lions \cite[Theorem II.2 and Corollary II.1]{Lio84} to get that $\{\widetilde{\rho_{m}}\}$ is relatively compact up to translations in $L^{\frac{4}{3}}\left( \mathbb{R}^{3} \right)$ and $L^{1}\left( \mathbb{R}^{3} \right)$. Thus, there is a subsequence $\{\widetilde{\rho_{m_{k}}}\}$ that converges, up to translations, to some function $f$ in $L^{\frac{4}{3}}\left( \mathbb{R}^{3} \right)$, $L^{1}\left( \mathbb{R}^{3} \right)$ and almost everywhere in $\mathbb{R}^{3}$. We denote again the subsequence after translations by $\left\{\widetilde{\rho_{m_{k}}}\right\}$. In particular, $\left\| f \right\|_{L^{1}} = 1$. 
	
	And we claim that $E_{0}(f) \leq {\lim\limits_{k\rightarrow\infty}{E_{0}\left( \widetilde{\rho_{m_{k}}} \right)}}$. In fact, we recall 
	$$E_{0}(\rho) = {\int_{\mathbb{R}^3}{\frac{K\left( {\rho(x)} \right)^{\gamma}}{\gamma - 1}\,dx}} - \frac{1}{2}{\iint_{\mathbb{R}^3\times \mathbb{R}^3}{\frac{\rho(x) \cdot \rho(y)}{\left| {x - y} \right|}\,dx\,dy}}$$ 
	
    For the inertial energy $\int_{\mathbb{R}^3}{\frac{K\left( {\rho(x)} \right)^{\gamma}}{\gamma - 1}\,dx}$, by $\left. \widetilde{\rho_{m_{k}}}\rightarrow f \right.$ a.e. and Fatou's Lemma we have ${\int_{\mathbb{R}^3}{\frac{K\left( {f(x)} \right)^{\gamma}}{\gamma - 1}\,dx}} \leq \lim\limits_{k\rightarrow\infty}\inf{\int_{\mathbb{R}^3}{\frac{K\left( {\widetilde{\rho_{m_{k}}}(x)} \right)^{\gamma}}{\gamma - 1}\,dx}}$. 
    
    For the potential energy $\frac{1}{2}{\iint_{\mathbb{R}^3\times \mathbb{R}^3}{\frac{\rho(x) \cdot \rho(y)}{\left| {x - y} \right|}\,dx\,dy}} = \frac{1}{2}{\int_{\mathbb{R}^3}{\rho(x)V_{\rho}(x)\,dx}}$, we have
	\begin{equation*}
		\begin{split}			
			&\quad\, \left| {{\int_{\mathbb{R}^3}{\widetilde{\rho_{m_{k}}}(x)V_{\widetilde{\rho_{m_{k}}}}(x)\,dx}} - {\int_{\mathbb{R}^3}{f(x)V_{f}(x)\,dx}}} \right| \\
			&= \left| {{\int_{\mathbb{R}^3}{\widetilde{\rho_{m_{k}}}(x)V_{\widetilde{\rho_{m_{k}}}}(x)\,dx}} - {\int_{\mathbb{R}^3}{f(x)V_{\widetilde{\rho_{m_{k}}}}(x)\,dx}} + {\int_{\mathbb{R}^3}{f(x)V_{\widetilde{\rho_{m_{k}}}}(x)\,dx}} - {\int_{\mathbb{R}^3}{f(x)V_{f}(x)\,dx}}} \right| \\
			&\leq \left| {{\int_{\mathbb{R}^3}{\widetilde{\rho_{m_{k}}}(x)V_{\widetilde{\rho_{m_{k}}}}(x)\,dx}} - {\int_{\mathbb{R}^3}{f(x)V_{\widetilde{\rho_{m_{k}}}}(x)\,dx}}} \right| + \left| {{\int_{\mathbb{R}^3}{f(x)V_{\widetilde{\rho_{m_{k}}}}(x)\,dx}} - {\int_{\mathbb{R}^3}{f(x)V_{f}(x)\,dx}}} \right| \\
			&\leq \left\| {\widetilde{\rho_{m_{k}}} - f} \right\|_{L^{\frac{4}{3}}} \cdot \left\| V_{\widetilde{\rho_{m_{k}}}} \right\|_{L^{4}} + \left\| f \right\|_{L^{\frac{12}{11}}} \cdot \left\| {V_{f} - V_{\widetilde{\rho_{m_{k}}}}} \right\|_{L^{12}}
		\end{split}
	\end{equation*}
	The last inequality above comes from Hölder's inequality. Thanks to Hardy-Littlewood-Sobolev Inequality (Proposition \ref{HLSI}), we have $\left\| V_{\widetilde{\rho_{m_{k}}}} \right\|_{L^{4}} \leq C\left\| \widetilde{\rho_{m_{k}}} \right\|_{L^{\frac{12}{11}}}$, and $\left. \widetilde{\rho_{m_{k}}}\rightarrow f \right.$ in $L^{\frac{4}{3}}\left( \mathbb{R}^{3} \right)$ implies $\left. V_{\widetilde{\rho_{m_{k}}}}\rightarrow V_{f}(x) \right.$ in $L^{12}\left( \mathbb{R}^{3} \right)$. Due to the Interpolation Inequality \cite[Section 4.2]{Bre11}, we have $\left\| \widetilde{\rho_{m_{k}}} \right\|_{L^{\frac{12}{11}}} \leq \left\| \widetilde{\rho_{m_{k}}} \right\|_{L^{1}} \cdot \left\| \widetilde{\rho_{m_{k}}} \right\|_{L^{\frac{4}{3}}}$, and $\left\| f \right\|_{L^{\frac{12}{11}}} \leq \left\| f \right\|_{L^{1}} \cdot \left\| f \right\|_{L^{\frac{4}{3}}}$. Thus $\left\| \widetilde{\rho_{m_{k}}} \right\|_{L^{\frac{12}{11}}}$ and $\left\| f \right\|_{L^{\frac{12}{11}}}$ are bounded, which—together with the inequality $\left\| V_{\widetilde{\rho_{m_{k}}}} \right\|_{L^{4}} \leq C\left\| \widetilde{\rho_{m_{k}}} \right\|_{L^{\frac{12}{11}}}$—implies the boundedness of $\left\| V_{\widetilde{\rho_{m_{k}}}} \right\|_{L^{4}}$. Then
	$$\left. \left| {{\int_{\mathbb{R}^3}{\widetilde{\rho_{m_{k}}}(x)V_{\widetilde{\rho_{m_{k}}}}(x)\,dx}} - {\int_{\mathbb{R}^3}{f(x)V_{f}(x)\,dx}}} \right|\rightarrow 0 \right.$$ 
	Thus, we know the claim above holds that $E_{0}(f) \leq {\lim\limits_{k\rightarrow\infty}{E_{0}\left( \widetilde{\rho_{m_{k}}} \right)}} = e_{0} = E_{0}(\sigma)$. 
	
	In fact, the convergence up to a subsequence can be strengthened to the convergence of the whole sequence. Since the non-rotating minimizer is unique up to translation due to Theorem \ref{non-rotating}, we get $\left\{\widetilde{\rho_{m_{k}}} \right\}$ converges strongly to $\sigma$ in $L^{\frac{4}{3}}\left( \mathbb{R}^{3} \right)$ and $L^{1}\left( \mathbb{R}^{3} \right)$ up to translations. By the same arguments, we know for any subsequence $\left\{ \widetilde{\rho_{m_{k}}} \right\}$, we can always pick a subsequence $\left\{ \widetilde{\rho_{m_{k_{l}}}} \right\}$ of $\left\{ \widetilde{\rho_{m_{k}}} \right\}$ which converges strongly to $\sigma$ up to translations. Thus, by an argument of contradiction, we get, up to translations, $\left\{\widetilde{\rho_{m}} \right\}$ converges strongly to $\sigma$ in $L^{\frac{4}{3}}\left( \mathbb{R}^{3} \right)$ and $L^{1}\left( \mathbb{R}^{3} \right)$.
}
\end{proof}

Actually, we can also show the convergence of $\{\widetilde{\rho_{1-m}}\}$ and $\{{\rho_{1-m}}\}$ in $L^{\frac{4}{3}}\left( \mathbb{R}^{3} \right)$ and $L^{1}\left( \mathbb{R}^{3} \right)$. 

\begin{corollary}[Convergence of Scaling Densities of Stars in $L^{\frac{4}{3}}\left( \mathbb{R}^{3} \right)$ and $L^{1}\left( \mathbb{R}^{3} \right)$] \label{convergence of scaling densities of stars}
Fix $J>0$, $m_0>0$ which is given in Theorem \ref{existence of the constrained minimizers in binary case}. Let $0<m<m_0$, and $\rho(m)=\rho_{1-m}+\rho_m$ is the constrained minimizer on $W_m$, $\widetilde{\rho_{1-m}} = A(1-m)\rho_{1-m}\left( {B(1-m)x} \right)$ as in Definition \ref{scaling density for planet}, we have, up to translations, $\{\widetilde{\rho_{1-m}}\}$ converges strongly to $\sigma$ in $L^{\frac{4}{3}}\left( \mathbb{R}^{3} \right)$ and $L^{1}\left( \mathbb{R}^{3} \right)$ as $m\rightarrow 0$.
\end{corollary}

\begin{proof}
{
	\rm
	We can apply the same arguments as above and obtain:
	\begin{equation*}
		E_{0}\left( \widetilde{\rho_{1-m}} \right) < C(1-m)^{\frac{6-5\gamma}{3\gamma - 4}}\cdot m^3 + E_{0}(\sigma)
	\end{equation*}
	Thus ${\lim\limits_{m\rightarrow 0}{E_{0}\left( \widetilde{\rho_{1-m}} \right)}} = e_{0} < 0$, and we can follow a proof similar to that in Theorem \ref{convergence of scaling densities of planets} above to get the result.
}
\end{proof}

\begin{corollary}[Convergence of Densities of Stars in $L^{\frac{4}{3}}\left( \mathbb{R}^{3} \right)$ and $L^{1}\left( \mathbb{R}^{3} \right)$] \label{convergence of densities of stars}
Fix $J>0$, $m_0>0$ which is given in Theorem \ref{existence of the constrained minimizers in binary case}. Let $0<m<m_0$, and $\rho(m)=\rho_{1-m}+\rho_m$ is the constrained minimizer on $W_m$, we have, up to translations, $\{\rho_{1-m}\}$ converges strongly to $\sigma$ in $L^{\frac{4}{3}}\left( \mathbb{R}^{3} \right)$ and $L^{1}\left( \mathbb{R}^{3} \right)$ as $m \rightarrow 0$.
\end{corollary}

\begin{proof}
{
	\rm
	Similarly, we can still compare $E_0 (\sigma_{1-m})$ and $E_0 (\rho_{1-m})$ without scaling the densities and obtain: 
	\begin{equation*}
		E_{0}\left( \rho_{1-m} \right) < Cm^3 + (1-m)^{\frac{5\gamma-6}{3\gamma - 4}}E_{0}(\sigma)
	\end{equation*}
	
	Thus ${\lim\limits_{m\rightarrow 0}{E_{0}\left( {\rho_{1-m}} \right)}} = e_{0} < 0$, and we can follow a proof similar to that in Theorem \ref{convergence of scaling densities of planets} above to get the result.
}
\end{proof}

\begin{remark}\label{reason of scaling}
For $\rho_m$, we can similarly obtain:
$$
E_{0}\left( \rho_{m} \right) < Cm^{3} + m^{\frac{5\gamma - 6}{3\gamma - 4}} \cdot E_{0}(\sigma)
$$
Then we have ${\lim\limits_{m\rightarrow 0}{E_{0}\left( \rho_{m} \right)}} = 0$, which violates the condition of concentration-compactness lemma \cite[Theorem II.2 and Corollary II.1]{Lio84}. In this case, we can actually show that $\rho_m$ converges to $0$: from the scaling relation (\ref{scaling density for planet expression}) and the fact ${\lim\limits_{m\rightarrow 0}{A(m)}} = \infty$ and ${\lim\limits_{m\rightarrow 0}{B(m)}} = 0$, we can show $\left. \left\| \rho_{m} \right\|_{L^{\frac{4}{3}}} = \left( {A(m)} \right)^{- \frac{4}{3}}\left( {B(m)} \right)^{3}\left\| \widetilde{\rho_{m}} \right\|_{L^{\frac{4}{3}}}\rightarrow 0 \right.$ as $m \rightarrow 0$. However, this vanishing limit result is not enough. We seek a stronger estimate that will allow us to estimate the Lagrange multiplier and then the support size of $\rho_m$, as we will see later. This is the reason we introduce the scaling density.
\end{remark}

\begin{remark}
Note the convergence holds true for the whole sequence instead of up to a subsequence. 
\end{remark}

\subsection{Uniform bound of scaling densities in $L^{\infty}\left(\mathbb{R}^{3}\right)$}\label{subsection5.2-Uniform bound of scaling densities in L^{infty}}

Thanks to the convergence results in subsection \ref{subsection5.1-Convergence of scaling densities in L^{4/3} and L^1}, now we can estimate the uniform $L^\infty$ bound of scaling densities $\left\{ \widetilde{\rho_{m}} \right\}$ when $m$ is sufficiently small. Since $\widetilde{\rho_{m}} \rightarrow \sigma $ in $L^{\frac{4}{3}}\left( \mathbb{R}^{3} \right)$ (Theorem \ref{convergence of scaling densities of planets}), we know $\left\{ \widetilde{\rho_{m}} \right\}$ are bounded uniformly in $L^{\frac{4}{3}}\left( \mathbb{R}^{3} \right)$ when $m$ is sufficiently small. The next steps are to apply bootstrap method as follows: (1) $\left\{ V_{\widetilde{\rho_{m}}} \right\}$ are bounded uniformly in $L^{12}\left( \mathbb{R}^{3} \right)$; (2) $\left\{ \widetilde{\rho_{m}} \right\}$ are bounded uniformly in $L^{6}\left( \mathbb{R}^{3} \right)$; (3) $\left\{ V_{\widetilde{\rho_{m}}} \right\}$ are bounded uniformly in $L^{\infty}\left( \mathbb{R}^{3} \right)$; (4) $\left\{ \widetilde{\rho_{m}} \right\}$ are bounded uniformly in $L^{\infty}\left( \mathbb{R}^{3} \right)$.

\begin{theorem}[Uniform bound of Scaling Densities of Planets and Potentials in $L^{\infty}\left( \mathbb{R}^{3} \right)$] \label{uniform bound of scaling densities of planets and potential in L^infty}
Fix $J>0$, $m_0>0$ which is given in Theorem \ref{existence of the constrained minimizers in binary case}. Let $0<m<m_0$ and $\rho(m)=\rho_{1-m}+\rho_m$ is the constrained minimizer on $W_m$, $\widetilde{\rho_{m}} = A(m)\rho_{m}\left( {B(m)x} \right)$ as in Definition \ref{scaling density for planet}, then $\exists C>0$, $\delta>0$, $\forall 0<m<\delta$, we have $\widetilde{\rho_{m}}\in L^{\infty}\left( \mathbb{R}^{3} \right)$, and  
$\left\| \widetilde{\rho_{m}} \right\|_{L^{\infty}(\mathbb{R}^{3})} \leq C$, $V_{\widetilde{\rho_{m}}} \in L^{\infty}\left( \mathbb{R}^{3} \right)$, and $\left\| V_{\widetilde{\rho_{m}}} \right\|_{L^{\infty}(\mathbb{R}^{3})} \leq C$, where $V_{\widetilde{\rho_{m}}}(x) = {\int_{\mathbb{R}^3}{\frac{\widetilde{\rho_{m}}(y)}{\left| {x - y} \right|}\,dy}}$ is the potential from $\widetilde{\rho_{m}}$. Moreover, $V_{\widetilde{\rho_{m}}}(x)$ is continuous.
\end{theorem}

\begin{proof}
{
	\rm
	We modify a little bit the arguments in Theorem \ref{uniform bound of minimizers in L^infty in binary case}. By Hardy-Littlewood-Sobolev Inequality (Proposition \ref{HLSI}) and the uniform bound of $\left\{ \widetilde{\rho_{m}} \right\}$ in $L^{\frac{4}{3}}\left( \mathbb{R}^{3} \right)$ due to Theorem \ref{convergence of scaling densities of planets} (here uniform bound means a bound independent of $m$ in a small interval $(0,\tilde \delta)$), we have $\left\| V_{\widetilde{\rho_{m}}} \right\|_{L^{12}(\mathbb{R}^{3})} \leq C_{1}\left\| \widetilde{\rho_{m}} \right\|_{L^{\frac{4}{3}}{(\mathbb{R}^{3})}} \leq C$. Due to Remark \ref{rho bounded by V_rho}, together with (\ref{scaling for A prime}) (\ref{scaling for V}), we know $\exists C>0$, such that ${\widetilde{\rho_{m}}}^{\gamma - 1} \leq CV_{\widetilde{\rho_{m}}}$, then there is a $C_2>0$, such that
	$$
	\left\| \widetilde{\rho_{m}} \right\|_{L^{12{({\gamma - 1})}}} = \left( {\int_{\mathbb{R}^3}{{\widetilde{\rho_{m}}}^{12{({\gamma - 1})}}(x)\,dx}} \right)^{\frac{1}{12{({\gamma - 1})}}} \leq {\widetilde{C}\left( {\int_{\mathbb{R}^3}{V_{\widetilde{\rho_{m}}}^{12}(x)\,dx}} \right)}^{\frac{1}{12{({\gamma - 1})}}} = \widetilde{C}\left\| V_{\widetilde{\rho_{m}}} \right\|_{L^{12}}^{\frac{1}{\gamma - 1}} \leq C_{2}
	$$ 	
	By assumption $\gamma > \frac{3}{2}>\frac{4}{3}$, together with Interpolation Inequality \cite[Section 4.2]{Bre11} we have $\left\| \widetilde{\rho_{m}} \right\|_{L^{12}} \leq \left\| \widetilde{\rho_{m}} \right\|_{L^{12{({\gamma - 1})}}}^{\alpha} \cdot \left\| \widetilde{\rho_{m}} \right\|_{L^{1}}^{1 - \alpha} \leq C$ for some $C>0$. Since we already know $\left\| \widetilde{\rho_{m}} \right\|_{L^{1}{(\mathbb{R}^{3})}}=1$ and $\left\| \widetilde{\rho_{m}} \right\|_{L^{12}{(\mathbb{R}^{3})}}$ are bounded uniformly, due to Proposition \ref{diff. of poten.}, we have $V_{\widetilde{\rho_{m}}}$  is continuous and $\left\| V_{\widetilde{\rho_{m}}} \right\|_{L^{\infty}{(\mathbb{R}^{3})}}$ is bounded uniformly. Again by ${\widetilde{\rho_{m}}}^{\gamma - 1} \leq CV_{\widetilde{\rho_{m}}}$ we know $\left\| \widetilde{\rho_{m}} \right\|_{L^{\infty}(\mathbb{R}^{3})}$ is also bounded uniformly.
}
\end{proof}

\begin{remark}
The uniform bound in $L^{\infty}(\mathbb{R}^{3})$ can be comparable with Li's result \cite[Proposition 1.4]{Li91}, where the rotating single stars are considered and the bound $C(m)$ however depends on $m$. Hence, if it can be shown that ${\lim\limits_{m\rightarrow 0}{\sup{C(m)}}} < \infty$, then an analogous result to Theorem \ref{uniform bound of scaling densities of planets and potential in L^infty} should also hold.

%
\end{remark}

\begin{remark}\label{uniform bound for the norm of stars or scaling stars}
By the similar arguments, we can obtain that when $m$ is sufficiently small, $\left\{ \widetilde{\rho_{1 - m}} \right\}$ and $\left\{ \rho_{1 - m} \right\}$ are also bounded uniformly in $L^{\infty}(\mathbb{R}^{3})$.
\end{remark}

\begin{remark}\label{bound convergence of planets}
Similar to the discussion in Remark \ref{asymptotic behaviour of radius and norm in single star case}, once we know $\left\| \widetilde{\rho_{m}} \right\|_{L^{\infty}(\mathbb{R}^{3})}$ is bounded uniformly, by definition we have $\rho_{m} = \frac{1}{A}\widetilde{\rho_{m}}\left( \frac{x}{B} \right)$, $A = m^{- \frac{2}{3\gamma - 4}}$, thus ${\lim\limits_{m\rightarrow 0}\left\| \rho_{m} \right\|_{L^{\infty}(\mathbb{R}^{3})}} = 0$. 
\end{remark}

\subsection{Estimate for the supports of densities}\label{subsection5.3-Estimate for the supports of densities}

In this subsection, we finish the estimate for the supports of scaling densities as well as the original densities. We first seek an upper bound for the Lagrange multiplier. To avoid confusion with the Lagrange multiplier $\lambda_m$ corresponding to star-planet system, we will use $\kappa_m$ to represent the corresponding non-rotating Lagrange multiplier with mass $m$ in this subsection.

Since both the function $A(\rho)=\frac{K}{\gamma-1} \rho^\gamma$ in (\ref{A}) and the scaling coefficient $A$ in (\ref{scaling density for planet expression}) will appear frequently in this subsection, to avoid confusion, we will use $\mathcal{A}(\rho)$ to denote the function $A(\rho)$, while keeping $A$ to represent the coefficient $A$ in this subsection.

Similar to the arguments in the proof of Theorem \ref{existence of ``double constrained'' minimizers} (see also \cite[Lemma 2]{AB71} \cite[Section 5]{Che26G1} \cite[Section 2]{Che26R2}), it can be shown that given a constrained minimizer $\rho(m)$, $\rho(m)$ satisfies Euler-Lagrange equation in the following sense:
\begin{equation} \label{ELm}
\mathcal{A}^{\prime}(\rho (m)(x))=\left[\frac{J^2}{2 I^2(\rho (m))} r^2(x-\bar{x}(\rho (m)))+V_{\rho (m)}(x)+\lambda_\alpha\right]_{+} \text{on } \Omega_\alpha \tag{ELm}
\end{equation}
where $\alpha \in \{m,1-m\}$.

Similar to Proposition \ref{negative Lagrange-multipliers in binary case} we can show $\lambda_\alpha <0$ for $m$ small enough. Now we give a stronger result in the following lemma.

\begin{lemma}[Bound for the Lagrange Multiplier]\label{bound for the Lagrange Multiplier for oiginal densities}
Fix $J>0$, $\forall 0<\epsilon<\kappa_{1}$, $\exists \delta>0$, $\forall 0<m<\delta$, if $\rho(m)=\rho_{m}+\rho_{1-m}$ minimizes $E_{J}(\rho)$ on $W_{m}$, then the Lagrange multiplier $\lambda_{m}$ in the Euler-Lagrange equations (\ref{ELm}) satisfies $A^{\gamma-1}(m) \cdot \lambda_{m} \leq \kappa_{1}+\epsilon<0$. Here $\kappa_{1}$ denotes the non-rotating Lagrange multiplier from Theorem \ref{non-rotating} with respect to mass $1$.
\end{lemma}

\begin{proof}[Proof of Lemma \ref{bound for the Lagrange Multiplier for oiginal densities}]
{
	\rm
	Let $m$ be smaller than $m_{0}$ given in Theorem \ref{existence of the constrained minimizers in binary case}, then minimizers exist. From Euler-Lagrange equations (\ref{ELm}) we have
	$$
	\lambda_{m} \leq \mathcal{A}^{\prime}\left(\rho_{m}\right)-V_{\rho_{m}} \text { a.e. on } \Omega_{m}
	$$
	where $\mathcal{A}(\rho)(x)=\frac{K\left(\rho^{\gamma}(x)\right)}{\gamma-1}$, $\mathcal{A}^{\prime}(\rho)(x)=\frac{K \gamma\left(\rho^{\gamma-1}(x)\right)}{\gamma-1}$. Due to (\ref{scaling for A prime}) (\ref{scaling for V}), we have $$\mathcal{A}^{\prime}\left(\rho_{m}\right)(x)=\mathcal{A}^{\prime}\left(\frac{1}{A} \widetilde{\rho_{m}}\right)\left(\frac{1}{B} x\right)=\frac{1}{A^{\gamma-1}} \mathcal{A}^{\prime}\left(\widetilde{\rho_{m}}\right)\left(\frac{1}{B} x\right)$$ $$V_{\rho_{m}}(x)=\frac{B^{2}}{A} V_{\widetilde{\rho_{m}}}\left(\frac{1}{B} x\right)=\frac{1}{A^{\gamma-1}} V_{\widetilde{\rho_{m}}}\left(\frac{1}{B} x\right)$$
	Therefore, we have
	$$
	A^{\gamma-1} \lambda_{m} \leq \mathcal{A}^{\prime}\left(\widetilde{\rho_{m}}\right)\left(\frac{1}{B} x\right)-V_{\widetilde{\rho_{m}}}\left(\frac{1}{B} x\right) \text { a.e. on } \Omega_{m}
	$$
	
	Then $\mu\left(\left\{x \in \Omega_{m} \mid A^{\gamma-1} \lambda_{m}>\mathcal{A}^{\prime}\left(\widetilde{\rho_{m}}\right)\left(\frac{1}{B} x\right)-V_{\widetilde{\rho_{m}}}\left(\frac{1}{B} x\right)\right\}\right)=0$
	
	Claim: Up to translations, ${\widetilde{\rho_{m}}}$ converges to ${\sigma}$ in (Lebesgue) measure, $V_{\widetilde{\rho_{m}}}$ converges to $V_{\sigma}$ in measure, $\mathcal{A}^{\prime}\left(\widetilde{\rho_{m}}\right)$ converges to $A^{\prime}(\sigma)$ in measure. 
	
	In fact, by Theorem \ref{convergence of scaling densities of planets}, we know up to translations, $\left\{\widetilde{\rho_{m}}\right\}$ converges strongly to $\sigma$ in $L^{\frac{4}{3}}\left(\mathbb{R}^{3}\right)$ and $L^{1}\left(\mathbb{R}^{3}\right)$. We denote the sequence after translations by $\left\{\widetilde{\rho_{m}}\right\}$ as well. Then ${\widetilde{\rho_{m}}}$ converges to ${\sigma}$ in measure. By Hardy-Littlewood-Sobolev Inequality (Proposition \ref{HLSI}), we know $V_{\widetilde{\rho_{m}}} \rightarrow V_{\sigma}$ strongly in $L^{12}\left(\mathbb{R}^{3}\right)$. Thus $V_{\widetilde{\rho_{m}}}$ converges to $V_{\sigma}$ in measure. For $\mathcal{A}^{\prime}\left(\widetilde{\rho_{m}}\right)$, by Theorem \ref{uniform bound of scaling densities of planets and potential in L^infty} and Theorem \ref{non-rotating} (iv), we know $\exists C>0$, such that $0 \leq \widetilde{\rho_{m}}<C$, $0 \leq \sigma<C$ a.e. on $\mathbb{R}^{3}$. Consequently, $\left|\widetilde{\rho_{m}}(x)-\sigma(x)\right|$ is well-defined almost everywhere. At such points, if $\left|\widetilde{\rho_{m}}(x)-\sigma(x)\right| \neq 0$, by mean value theorem, there is $\xi \in(0, C)$, such that
	$$
	\frac{\left | \mathcal{A}^{\prime} (\widetilde{\rho_{m}}) - \mathcal{A}^{\prime} (\sigma)\right |}{ \left |\widetilde{\rho_{m}} -\sigma \right |} =\frac{K_{\gamma}}{\gamma -1} \cdot \frac{\left|\widetilde{\rho_{m}}^{\gamma -1}-\sigma^{\gamma-1} \right|}{\left|\widetilde{\rho_{m}}-\sigma \right|} \leq K \gamma \xi ^{\gamma-2} \leq K \gamma C^{\gamma-2}
	$$
	Thus $\left|\mathcal{A}^{\prime}\left(\widetilde{\rho_{m}}\right)-\mathcal{A}^{\prime}(\sigma)\right| \leq C_{1}\left|\widetilde{\rho_{m}}-\sigma\right|$ with $C_{1}=K \gamma C^{\gamma-2}>0$. If $\left|\widetilde{\rho_{m}}(x)-\sigma(x)\right|=0$, then $\widetilde{\rho_{m}}(x)=\sigma(x)$, thus $\left|\mathcal{A}^{\prime}\left(\widetilde{\rho_{m}}\right)-\mathcal{A}^{\prime}(\sigma)\right| \leq C_{1}\left|\widetilde{\rho_{m}}-\sigma\right|$ also holds true. Since $\widetilde{\rho_{m}}$ converges to $\sigma$ in measure, we know when $m \rightarrow 0$, $$\mu\left(\left\{x \in \mathbb{R}^{3}\mid \left| \mathcal{A}^{\prime}\left(\widetilde{\rho_{m}}\right)-\mathcal{A}^{\prime}(\sigma) \right | (x) \geq \epsilon\right\}\right) \leq \mu\left(\left\{x \in \mathbb{R}^{3}\mid \left| \widetilde{\rho_{m}}(x)-\sigma(x) \right| \geq \frac{\epsilon}{c_{1}}\right\}\right) \rightarrow 0.$$ Thus $\mathcal{A}^{\prime}\left(\widetilde{\rho_{m}}\right)$ converges to $A^{\prime}(\sigma)$ in measure. Hence we prove the claim.
	
	By Theorem \ref{non-rotating} (v) and the fact $\sigma$ has mass 1, we know $0<\mu\left(\left\{x \in \mathbb{R}^{3} \mid \sigma(x)>0\right\}\right)<\infty$. Since $0 \leq \sigma< C$ a.e. on $\mathbb{R}^{3}$, let $E_{n}=\left\{x \in \mathbb{R}^{3} \mid \frac{C}{n} \leq \sigma(x)<C\right\}$, then $\mathbb{R}^{3}=\cup_{n=1}^{\infty} E_{n}$. And then $1=\int_{\mathbb{R}^{3}} \sigma\,dx= \sum_{n=1}^{\infty} \int_{E_{n}} \sigma\,dx$. Then $\exists N>0$, such that $$\frac{1}{2}<\sum_{n=1}^{N} \int_{E_{n}} \sigma\,dx<\sum_{n=1}^{N} \int_{E_{n}} C\,dx=C \mu(\{\{x \in \mathbb{R}^{3}\mid \frac{C}{N} \leq\sigma(x)<C\})$$
	Therefore, let $M:=\mu\left(\left\{x \in \mathbb{R}^{3} \mid \frac{C}{N}<\sigma(x)<C\right\}\right)$, then $0<\frac{1}{2 C}<M<\infty$. Thus $\exists \delta>0$, $\forall 0<m<\delta$, due to the claim we have
	$$
	\begin{aligned}
		&\mu\left(\left\{x \in \mathbb{R}^{3}\mid  \left |\mathcal{A}^{\prime}\left(\widetilde{\rho_{m}}\right)\left(\frac{x}{B}\right)-\mathcal{A}^{\prime}(\sigma)\left(\frac{x}{B}\right) \right | \geq \frac{\epsilon}{3}\right\}\right) \\
		&=B^{3} \mu\left(\left\{y \in \mathbb{R}^{3}\mid \left| \mathcal{A}^{\prime}\left(\widetilde{\rho_{m}}\right)(y)-\mathcal{A}^{\prime}(\sigma)(y)\right| \geq \frac{\epsilon}{3}\right\}\right) \\
		& <B^{3} \frac{M}{4}
	\end{aligned}
	$$
	and similarly,
	$$
	\begin{aligned}
		& \mu\left(\left\{x \in \mathbb{R}^{3}\mid \left | V_{\widetilde{\rho_{m}}}\left(\frac{x}{B}\right)-V_{\sigma}\left(\frac{x}{B}\right) \right | \geq \frac{\epsilon}{3}\right\}\right)<B^{3} \frac{M}{4} \\
		& \mu\left(\left\{x \in \mathbb{R}^{3}\mid \left| \widetilde{\rho_{m}}\left(\frac{x}{B}\right)-\sigma\left(\frac{x}{B}\right) \right | \geq \frac{C}{3 N}\right\}\right)<B^{3} \frac{M}{4}
	\end{aligned}
	$$
	
	Let $E=\left\{x \in \mathbb{R}^{3} \mid \frac{C}{N}<\sigma\left(\frac{x}{B}\right)<C\right\}$, $\mu(E)=B^{3} M$. Let
	\begin{align*}
		E_{0} &= \left\{ x \in \mathbb{R}^{3} \mid \left| \mathcal{A}^{\prime}(\widetilde{\rho_{m}})\left(\frac{x}{B}\right) - \mathcal{A}^{\prime}(\sigma)\left(\frac{x}{B}\right) \right| \geq \frac{\epsilon}{3} \right\} \\
		&\quad \cup \left\{ x \in \mathbb{R}^{3} \mid \left| V_{\widetilde{\rho_{m}}}\left(\frac{x}{B}\right) - V_{\sigma}\left(\frac{x}{B}\right) \right| \geq \frac{\epsilon}{3} \right\} \\
		&\quad \cup \left\{ x \in \mathbb{R}^{3} \mid \left| \widetilde{\rho_{m}}\left(\frac{x}{B}\right) - \sigma\left(\frac{x}{B}\right) \right| \geq \frac{C}{3N} \right\} \\
		&\quad \cup \left\{ x \in \Omega_{m} \mid A^{\gamma-1} \lambda_{m} > \mathcal{A}^{\prime}(\widetilde{\rho_{m}})\left(\frac{1}{B} x\right) - V_{\widetilde{\rho_{m}}}\left(\frac{1}{B} x\right) \right\}
	\end{align*}
	We further define $\tilde{E}=E \backslash E_{0}$, then from the results above we know $\mu(\tilde{E})> B^{3} \frac{M}{4}>0$. In particular, $\tilde{E}$ is not empty. Pick $x \in \tilde{E}$, then $\left|\widetilde{\rho_{m}}\left(\frac{x}{B}\right)-\sigma\left(\frac{x}{B}\right)\right|<\frac{C}{3 N}$, and 
	$$\widetilde{\rho_{m}}\left(\frac{x}{B}\right) \geq \sigma\left(\frac{x}{B}\right)-\left|\widetilde{\rho_{m}}\left(\frac{x}{B}\right)-\sigma\left(\frac{x}{B}\right)\right|>\frac{C}{N}-\frac{C}{3 N}>0$$
	Hence $\rho_{m}(x)=\frac{1}{A} \widetilde{\rho_{m}}\left(\frac{x}{B}\right)>0$. Thus $x \in \Omega_{m}$ (up to translations),
	\begin{equation}\label{estimate of lambda_m}
			\begin{aligned}
			A^{\gamma-1} \lambda_{m} & \leq \mathcal{A}^{\prime}\left(\widetilde{\rho_{m}}\right)\left(\frac{1}{B} x\right)-V_{\widetilde{\rho_{m}}}\left(\frac{1}{B} x\right) \\
			& =\mathcal{A}^{\prime}\left(\widetilde{\rho_{m}}\right)\left(\frac{1}{B} x\right)-\mathcal{A}^{\prime}(\sigma)\left(\frac{x}{B}\right)+\mathcal{A}^{\prime}(\sigma)\left(\frac{x}{B}\right)-V_{\sigma}\left(\frac{x}{B}\right)+V_{\sigma}\left(\frac{x}{B}\right)-V_{\widetilde{\rho_{m}}}\left(\frac{1}{B} x\right) \\
			& \leq \mathcal{A}^{\prime}(\sigma)\left(\frac{x}{B}\right)-V_{\sigma}\left(\frac{x}{B}\right)+\frac{2 \epsilon}{3}\\
			& <\mathcal{A}^{\prime}(\sigma)\left(\frac{x}{B}\right)-V_{\sigma}\left(\frac{x}{B}\right)+\epsilon \\
			& =\kappa_{1}+\epsilon
		\end{aligned}
	\end{equation}
	
	The last inequality comes from the facts that $x \in E$ implies $\frac{x}{B}\in \{\sigma>0\}$ and then $\mathcal{A}^{\prime}(\sigma)(\frac{x}{B})-V_{\sigma}(\frac{x}{B})=\kappa_{1}$ by Theorem \ref{non-rotating} (vii).
}

\end{proof}

\begin{remark}
Note that the above proof relies on convergence in measure. In McCann's paper \cite[Lemma 6.5]{McC06}, only almost everywhere convergence up to a subsequence is used. Nonetheless, since in this paper we assume $m$ is very small, relying solely on almost everywhere convergence up to a subsequence will lead to the following two issues. One is that it is up to a subsequence, i.e., $\widetilde{\rho_{m_{k}}}(x) \rightarrow \sigma(x)$ a.e. on $\mathbb{R}^{3}$ instead of $\left\{\widetilde{\rho_{m}}\right\}$ itself, while we hope the bound for the Lagrange multiplier should be valid for all sufficiently small mass. Another one is when $m \rightarrow 0$, not only $\widetilde{\rho_{m}}(x) \rightarrow \sigma(x)$ a.e. on $\mathbb{R}^{3}$, but the measures of relevant sets above go to $0$. For example, when $m \rightarrow 0$, $\mu\left(\left\{x \in \mathbb{R}^{3} \mid \sigma\left(\frac{x}{B(m)}\right)>0\right\}\right)=B^{3}(m) \mu\left(\left\{x \in \mathbb{R}^{3} \mid \sigma(x)>0\right\}\right) \rightarrow 0$. Therefore, we need to rule out the case that $\widetilde{E}$ is empty; otherwise, it is not clear whether there exists an $x$ that satisfies the last inequality in (\ref{estimate of lambda_m}). Therefore, we make more use of $L^{p}$ convergence—namely, that it implies convergence in measure, to solve the issues.
\end{remark}

Before we finally estimate the sizes of supports of scaling densities, we supplement the proof of a proposition about the $L^{\infty}$ bound of potential mentioned in \cite[Proposition 6.6]{McC06}, which can be viewed as a complement to \cite[Proposition 5]{AB71}, in which the indices $\frac{2}{3}$ and $\frac{1}{3}$ can be approximated but cannot be reached.

\begin{proposition}[Bound of Potential in $L^{\infty}\left(\mathbb{R}^{3}\right)$ {\cite[Section 6]{McC06}}]\label{bound of potential in McCann's paper}
$\exists k>0, \forall \rho \in L^{1}\left(\mathbb{R}^{3}\right) \cap L^{\infty}\left(\mathbb{R}^{3}\right)$, one has $\left\|V_{\rho}\right\|_{L^{\infty}\left(\mathbb{R}^{3}\right)} \leq k\|\rho\|_{L^{1}\left(\mathbb{R}^{3}\right)}^{\frac{2}{3}}\|\rho\|_{L^{\infty}\left(\mathbb{R}^{3}\right)}^{\frac{1}{3}}$, where $V_{\rho}(x)=\int_{\mathbb{R}^3} \frac{\rho(y)}{|x-y|} \,dy$ is the potential from $\rho$.
\end{proposition}

\begin{proof}
{
	Let $B_R(\cdot)$ be the open ball defined by $B_R(\cdot) \coloneq \left\{x \in \mathbb{R}^3\mid |x-\cdot| <R\right\}$.
	$$
	\begin{aligned}
		V_{\rho}(x)&=\int_{\mathbb{R}^{3}} \frac{\rho(y)}{|x-y|} \,dy\\
		& =\int_{B_{R}(x)} \frac{\rho(y)}{|x-y|} \,dy+\int_{\mathbb{R}^{3} \backslash B_{R}(x)} \frac{\rho(y)}{|x-y|} \,dy \\
		& \leq\|\rho\|_{L^{\infty}} \int_{B_{R}(x)} \frac{1}{|x-y|} \,dy+\frac{1}{R} \int_{\mathbb{R}^{3} \backslash B_{R}(x)} \rho(y) \,dy \\
		& \leq C R^{2}\|\rho\|_{L^{\infty}}+\frac{1}{R}\|\rho\|_{L^{1}}
	\end{aligned}
	$$
	
	Consider function $f(R)=C R^{2}\|\rho\|_{L^{\infty}}+\frac{1}{R}\|\rho\|_{L^{1}}, R \geq 0$, it has the maximal value at $R=\left(\frac{\|\rho\|_{L^{1}}}{2 C\|\rho\|_{L^{\infty}}}\right)^{\frac{1}{3}}$ and $f\left(\left(\frac{\|\rho\|_{L^{1}}}{2 C\|\rho\|_{L^{\infty}}}\right)^{\frac{1}{3}}\right)=k\|\rho\|_{L^{1}\left(\mathbb{R}^{3}\right)}^{\frac{2}{3}}\|\rho\|_{L^{\infty}\left(\mathbb{R}^{3}\right)}^{\frac{1}{3}}$ for some $k>0$. Thus, we have $$\left\|V_{\rho}\right\|_{L^{\infty}\left(\mathbb{R}^{3}\right)} \leq k\|\rho\|_{L^{1}\left(\mathbb{R}^{3}\right)}^{\frac{2}{3}}\|\rho\|_{L^{\infty}\left(\mathbb{R}^{3}\right)}^{\frac{1}{3}}$$.
}
\end{proof}

\begin{theorem}[Bound on the sizes of supports of scaling densities $\widetilde{\rho_{m}}$]\label{bound on the sizes of scaling densities}
There exists a radius $R(J)$ and $\delta>0$ independent of $m$, $\forall 0<m<\delta$, if $\rho(m)= \rho_{1-m}+\rho_{m}$ minimizes $E_{J}(\rho)$ on ${W}_{m}$, then spt $\widetilde{\rho_{m}}$ is contained in a ball of radius $R(J)$, where $\widetilde{\rho_{m}}$ is defined in Definition \ref{scaling density for planet}.
\end{theorem}

\begin{proof}
{
	\rm
	Take $m$ and $\epsilon$ small enough so that constrained minimizers $\rho(m)$ exist by Theorem \ref{existence of the constrained minimizers in binary case} and $A(m)^{\gamma-1} \cdot \lambda_{m} \leq \kappa_{1}+\epsilon<$ 0 by Lemma \ref{bound for the Lagrange Multiplier for oiginal densities}, and the velocity $v(x, m)$ satisfies $v(x, m)^{2} \leq C m^{2}$ by Corollary \ref{velocity decreasing rate as mass decreases}. By the construction of $\Omega_{m}$ and $\Omega_{1-m}$, we have $\operatorname{dist}\left(\Omega_{m}, \Omega_{1-m}\right)=\frac{\eta}{2}$, where $\eta=\frac{J^{2}}{\mu_r^{2}}=\frac{J^{2}}{(m(1-m))^{2}}$, thus $\exists C>0$, $\forall x \in \Omega_{m}$, $V_{\rho_{1-m}}(x)<C m^{2}$. In the Euler-Lagrange equation (\ref{ELm}), we know for almost every $x \in \Omega_{m}$,
	$$
	\mathcal{A}^{\prime}\left(\rho_{m}\right)(x)=\frac{K \gamma}{\gamma-1} \rho_{m}^{\gamma-1}(x)=\left[v(x, m)^{2}+V_{\rho_{m}}(x)+V_{\rho_{1-m}}(x)+\lambda_{m}\right]_{+}
	$$
	
	Since $\rho_{m}(x)=\frac{1}{A} \widetilde{\rho_{m}}\left(\frac{1}{B} x\right)$, without loss of generality we assume $\frac{K \gamma}{\gamma-1}=1$, then we have
	$$
	\begin{aligned}
		{\widetilde{\rho_{m}}}^{\gamma-1}\left(\frac{x}{B}\right) & =A^{\gamma-1}\left[v(x, m)^{2}+\frac{1}{A^{\gamma-1}} V_{\widetilde{\rho_{m}}}\left(\frac{x}{B}\right)+V_{\rho_{1-m}}(x)+\lambda_{m}\right]_{+} \\
		& \leq\left[C A^{\gamma-1} m^{2}+V_{\widetilde{\rho_{m}}}\left(\frac{x}{B}\right)+\kappa_{1}+\epsilon\right]_{+}
	\end{aligned}
	$$
	
	Recall $A=m^{-\frac{2}{3 \gamma-4}}$, $A^{\gamma-1} m^{2}=m^{\frac{4 \gamma-6}{3 \gamma-4}}$. Thus when $m$ is small enough, for almost every $y \in \Omega_{m}^{B}:=\left\{y \in \mathbb{R}^{3} \mid {B}y \in \Omega_{m}\right\}$, we have ${\widetilde{\rho_{m}}}^{\gamma-1}(y) \leq\left[V_{\widetilde{\rho_{m}}}(y)+\frac{\kappa_{1}+\epsilon}{2}\right]_{+}$. Now, let $R_{0}$ from Theorem \ref{non-rotating} (v) bound the support radius of non-rotating minimizer $\sigma$, and choose $\tilde{R}(J) \geq R_{0}$ large enough so that $\frac{1}{\tilde{R}(J)-R_{0}} \leq-\frac{\kappa_{1}+\epsilon}{6}$. Using Theorem \ref{convergence of scaling densities of planets}, we know up to translation, $\widetilde{\rho_{m}} \rightarrow \sigma$ in $L^{1}\left(\mathbb{R}^{3}\right)$. In particular, $\forall \delta>0$, when $m$ is small enough, all but mass $\delta$ of $\widetilde{\rho_{m}}$ is forced into a ball of radius $R_{0}$, which is, without loss of generality, centered on 0. Then $\widetilde{\rho_{m}}=\widetilde{\rho_{m}} \cdot \mathbf{1}_{\left\{|x|<R_{0}\right\}}+\widetilde{\rho_{m}} \cdot \mathbf{1}_{\left\{|x| \geq R_0\right\}}$, $V_{\widetilde{\rho_m}}=V_{\widetilde{\rho_m}} \cdot \mathbf{1}_{\left\{|x|<R_0\right\}}+V_{\widetilde{\rho_m}} \cdot \mathbf{1}_{\left\{|x| \geq R_0\right\}}$. Given $y$ outside the larger ball of radius $\tilde{R}(J)$ centered on $0$, due to $\frac{1}{\tilde{R}(J)-R_{0}} \leq-\frac{\kappa_{1}+\epsilon}{6}$, we have 
	$$V_{\widetilde{\rho_{m}} \cdot \mathbf{1}_{\left\{|x|<R_{0}\right\}}}(y) \leq-\frac{\kappa_{1}+\epsilon}{6}$$ 
	For $V_{\widetilde{\rho_{m}}} \cdot \mathbf{1}_{\left\{|x| \geq R_{0}\right\}}(y)$, by Theorem \ref{uniform bound of scaling densities of planets and potential in L^infty} we can show $\widetilde{\rho_{m}} \cdot \mathbf{1}_{\left\{|x| \geq R_{0}\right\}} \in L^{1}\left(\mathbb{R}^{3}\right) \cap L^{\infty}\left(\mathbb{R}^{3}\right)$ with mass $\delta$. And we pick a $\delta$ satisfying $k \delta^{\frac{2}{3}} \tilde{C}^{\frac{1}{3}}<-\frac{\kappa_{1}+\epsilon}{6}$, where $\tilde{C}$ is the uniform bound of $\left\|\widetilde{\rho_{m}}\right\|_{L^{\infty}}$ in Theorem \ref{uniform bound of scaling densities of planets and potential in L^infty}. Thus, when $m$ is small enough, by Proposition \ref{bound of potential in McCann's paper}, we have (note $V_{\widetilde{\rho_m}}$ is continuous due to Theorem \ref{uniform bound of scaling densities of planets and potential in L^infty})
\begin{align*}
	V_{\widetilde{\rho_m}} \cdot \mathbf{1}_{\{|x| \geq R_0\}}(y)
	&\leq \bigl\| V_{\widetilde{\rho_m}} \cdot \mathbf{1}_{\{|x| \geq R_0\}} \bigr\|_{L^\infty} \\
	&\leq k \bigl\| \widetilde{\rho_m} \cdot \mathbf{1}_{\{|x| \geq R_0\}} \bigr\|_{L^1}^{2/3} 
	\bigl\| \widetilde{\rho_m} \cdot \mathbf{1}_{\{|x| \geq R_0\}} \bigr\|_{L^\infty}^{1/3} \\
	&\leq k \delta^{2/3} \tilde{C}^{1/3} \\
	&< -\frac{\kappa_1+\epsilon}{6}.
\end{align*}
	By construction of $\Omega_{m}$ we know $\widetilde{\rho_{m}}=0$ outside $\Omega_{m}^{B}$. On $\Omega_{m}^{B}$ but outside the ball $B_{\tilde{R}(J)}(0)$, we know $0 \leq \widetilde{\rho_{m}}{ }^{\gamma-1}(y) \leq\left[V_{\widetilde{\rho_{m}}}(y)+\frac{\kappa_{1}+\epsilon}{2}\right]_{+} \leq\left[\frac{\kappa_{1}+\epsilon}{6}\right]_{+}=0$ a.e. In Theorem \ref{existence of the constrained minimizers in binary case} we know $\rho_m$ is continuous on $\Omega_{m}$, hence $\widetilde{\rho_{m}}(x)=A\rho_m(Bx)$ is continuous on $\Omega_{m}^B$, then we can replace ``a.e.'' by ``everywhere''. Therefore, up to translation, spt $\widetilde{\rho_{m}}$ is contained in $B_{2 \tilde{R}(J)}(0)$.
}
\end{proof}

\begin{remark}
There are various ways to bound $\left\|V_{\widetilde{\rho_{m}}} \cdot \mathbf{1}_{\left\{|x| \geq  R_{0}\right\}}\right\|_{L^{\infty}}$. For example, instead of applying Proposition \ref{bound of potential in McCann's paper}, one can apply \cite[Proposition 5]{AB71} directly.
\end{remark}


\begin{corollary}[Bound on the sizes of supports of original densities $\rho_{m}$] \label{bound on the sizes of original densities}
There exists a radius $R(J)$ and $\delta>0$ independent of $m, \forall 0<m<\delta$, if $\rho(m)= \rho_{1-m}+\rho_{m}$ minimizes $E_{J}(\rho)$ on ${W}_{m}$, then spt $\rho_{m}$ is contained in a ball of radius $B \cdot R(J)$, where $B$ is defined in Definition \ref{scaling density for planet}. In particular, if $\gamma>2$, spt $\rho_{m}$ is contained in a ball whose radius goes to 0 with rate $B$ since $\lim\limits_{m \rightarrow 0} B=0$.
\end{corollary}

\begin{proof}
{
	\rm
	The result comes from the fact that spt $\rho_{m}=B \text{ spt } \widetilde{\rho_{m}}=\{x \in \mathbb{R}^{3} \mid  \frac{x}{B} \in \text{ spt } \widetilde{\rho_{m}}\}$.
}
\end{proof}

We have shown the convergence rates of the bound of the supports. In particular, when planets' mass goes to zero, their density supports will be contained in some balls with same radius $R$. Same results can be shown for stars' scaling densities and original densities.

\begin{corollary}[Bound on the sizes of supports of scaling densities $\widetilde{\rho_{1-m}}$]\label{bound on the sizes of scaling stars}
There exists a radius $R(J)$ and $\delta>0$ independent of $m, \forall 0<m<\delta$, if $\rho(m)= \rho_{1-m}+\rho_{m}$ minimizes $E_{J}(\rho)$ on ${W}_{m}$, then spt $\widetilde{\rho_{1-m}}$ is contained in a ball of radius $R(J)$, where $\widetilde{\rho_{1-m}}$ is defined in Definition \ref{scaling density for planet}.
\end{corollary}

\begin{proof}
{
	\rm
	The proof is essentially the same as the proof above.
}
\end{proof}

\begin{corollary}[Bound on the sizes of supports of original densities $\rho_{1-m}$]\label{bound on the sizes of original stars}
There exists a radius $R(J)$ and $\delta>0$ independent of $m, \forall 0<m<\delta$, if $\rho(m)= \rho_{1-m}+\rho_{m}$ minimizes $E_{J}(\rho)$ on ${W}_{m}$, then spt $\rho_{1-m}$ is contained in a ball of radius $R(J)$.
\end{corollary}

\begin{proof}
{
	\rm
	The proof is essentially the same as the proof above, even without scaling stars' densities.
}
\end{proof}
\section{Existence for Star-Planet Systems}\label{section6-Existence for Star-Planet Systems}
So far, we have dealt with the uniform bound of the supports of density functions with different mass. Our next step is to show the centers of mass can be chosen to be not too far away from the centers of $\Omega_{1-m}$ and $\Omega_{m}$. Thanks to those results, we can then show the constrained minimizers are actually $W^{\infty}$ local energy minimizers, which represent solutions to the reduced Euler-Poisson system (\ref{EP'}) up to translation due to Theorem \ref{Properties of LEM}.

\subsection{Estimate for the Center of Mass Separation}\label{subsection6.1-Estimate for the Center of Mass Separation}
In this subsection, we estimate the centers of mass separation by modifying the arguments in \cite[Lemma 6.7 and Proposition 6.8]{McC06} such that they can be applied in our case.

\begin{lemma}\label{complex convergence}
	For $\epsilon>0$, define $g_{\epsilon}(z)=-\frac{1}{z-2 \epsilon}+\frac{1}{2\left(z^{2}+\epsilon^{\frac{3}{2}} R^{\frac{1}{2}} J^{-1}\right)}+\frac{1}{1+2 \epsilon}-\frac{1}{2}$. For $\epsilon$ is small enough, $g_{\epsilon}$ converges uniformly to $g_{0}(z)=-\frac{1}{z}+\frac{1}{2 z^{2}}+\frac{1}{2}$ as $\epsilon \rightarrow 0$ on $\left\{z \in \mathbb{C}  \mid |z| \geq \frac{1}{2}\right\}$.
\end{lemma}

\begin{proof}
	{
		\rm
		One can observe when $\epsilon$ is small enough, the functions $g_{\epsilon}$ and $g_{\epsilon}^{\prime}$ are analytic and uniformly bounded on $\left\{z \in \mathbb{C}\mid | z |>\frac{1}{4}\right\}$. The result can be checked by direct computation and estimation.
	}
\end{proof}

\begin{proposition}[Estimate for the Center of Mass Separation]\label{estimate of separation}
	Fix $J>0$, let $0<\zeta<\frac{1}{2}, \exists \delta>0$, such that $\forall 0<m<\delta$, if $\rho(m)=\rho_{1-m}+\rho_{m}$ minimizes $E_{J}(\rho)$ on ${W}_{m}$, then $1-\zeta<\frac{\left|\bar{x}\left(\rho_{m}\right)-\bar{x}\left(\rho_{1-m}\right)\right|}{\left|y_{m}-y_{1-m}\right|}<1+\zeta$, where $y_{m}$ and $y_{1-m}$ are the centers of $\Omega_{m}$ and $\Omega_{1-m}$, $\left|y_{m}-y_{1-m}\right|=\eta=\frac{J^{2}}{\mu_r^{2}}=\frac{J^{2}}{(m(1-m))^{2}}$.
\end{proposition}

\begin{proof}
	{
		\rm
		Take $m$ small enough such that Corollary \ref{bound on the sizes of original densities} and Corollary \ref{bound on the sizes of original stars} provide bound $R(J)$ for the supports of $\rho_{m}$ and $\rho_{1-m}$, and $R:=2 R(J)<\frac{\eta}{4}$. Due to Corollary \ref{bound on the sizes of original densities}, we know spt $\rho_{m} \subset B_{R(J)}\left(z_{m}\right)$ for some $z_m \in \mathbb{R}^3$, then $\bar{x}\left(\rho_{m}\right) \in B_{R(J)}\left(z_{m}\right)$. If $x$ is outside $B_{R}\left(\bar{x}\left(\rho_{m}\right)\right)$, then $\left|x-z_{m}\right| \geq\left|x-\bar{x}\left(\rho_{m}\right)\right|-\mid z_{m}-$ $\bar{x}\left(\rho_{m}\right) \mid>R(J)$, $\rho_{m}(x)=0$. Therefore, spt $\rho_{m} \subset B_{R}\left(\bar{x}\left(\rho_{m}\right)\right)$. Similarly, spt $\rho_{1-m} \subset$ $B_{R}\left(\bar{x}\left(\rho_{1-m}\right)\right)$. Now we translate $\rho_m$ and $\rho_{1 - m}$ separately so that their centers of mass are positioned at $y_{m}$ and $y_{1-m}$ respectively, and denote the shifted functions as $\kappa_{m}$ and $\kappa_{1-m}$. Therefore, $\kappa(m)\coloneq \kappa_{m}+\kappa_{1-m} \in W_{m}$. Due to the proof of Lemma \ref{interior} below we know $\bar{x}\left(\rho_{m}\right)$ and $\bar{x}\left(\rho_{1-m}\right)$ lie in the plane $z=c$. Hence the distance between them equals the distance between their projections onto $x_1x_2$ plane. That is, $d:=\left|\bar{x}\left(\rho_{m}\right)-\bar{x}\left(\rho_{1-m}\right)\right|=r\left(\bar{x}\left(\rho_{1-m}\right)-\bar{x}\left(\rho_{m}\right)\right)$, where $r$ is given in Definition \ref{notations}. Since spt $\rho_{m} \subset B_{R}\left(\bar{x}\left(\rho_{m}\right)\right)$ and spt $\rho_{1-m} \subset B_{R}\left(\bar{x}\left(\rho_{1-m}\right)\right)$, we have (recall $\mu_r=m(1-m)$)
		
		$$
		\frac{\mu_r}{d+2 R} \leq G\left(\rho_{m}, \rho_{1-m}\right)=\iint_{\mathbb{R}^3\times \mathbb{R}^3} \frac{\rho_{m}(x) \cdot \rho_{1-m}(y)}{|x-y|} \,d x \,d y \leq \frac{\mu_r}{d-2 R}
		$$
		
		By Lemma \ref{expansion of MoI}, we know		
		$$
			I(\rho(m))= \mu_r r^{2}\left(\bar{x}\left(\rho_{1-m}\right)-\bar{x}\left(\rho_{m}\right)\right)+\int_{\mathbb{R}^3} r^{2}\left(x-\bar{x}\left(\rho_{m}\right)\right) \rho_{1-m}(x) \,dx +\int_{\mathbb{R}^3} r^{2}\left(x-\bar{x}\left(\rho_{m}\right)\right) \rho_{m}(x) \,dx
		$$
		Hence
		$$\mu_r r^{2}\left(\bar{x}\left(\rho_{1-m}\right)-\bar{x}\left(\rho_{m}\right)\right)=\mu_r d^{2}\leq I(\rho(m))\leq \mu_r r^{2}\left(\bar{x}\left(\rho_{1-m}\right)-\bar{x}\left(\rho_{m}\right)\right)+(1-m) R^{2}+m R^{2}
		= \mu_r d^{2}+R^{2} $$
		
		By definition of $T_J$ (\ref{T_J}), we have
		$$\frac{J^{2}}{2\left(\mu_r d^{2}+R^{2}\right)} \leq T_{J}(\rho(m))=\frac{J^{2}}{2 I(\rho(m))} \leq \frac{J^{2}}{2 \mu_r d^{2}}$$
		
		Similar inequalities hold true for $G\left(\kappa_{m}, \kappa_{1-m}\right)$ and $T_{J}(\kappa(m))$, but we replace $d$ by $\eta$, since $\left|\bar{x}\left(\kappa_{m}\right)-\bar{x}\left(\kappa_{1-m}\right)\right|=r\left(\bar{x}\left(\kappa_{1-m}\right)-\bar{x}\left(\kappa_{m}\right)\right)=\left|y_m-y_{1-m}\right|=\eta$. Note $\rho(m)$ minimizes $E_{J}(\rho)$ on ${W}_{m}$, based on the energy decomposition similar to \eqref{energy decomposition}, we have $$E_{J}(\kappa(m))-E_{J}(\rho(m))=-G\left(\kappa_{m}, \kappa_{1-m}\right)+T_{J}(\kappa(m))-G\left(\rho_{m}, \rho_{1-m}\right)+T_{J}(\rho(m)) \geq 0$$ 
		Using the previous estimates and $J^{2}=\mu_r^{2} \eta$, we have
		
		$$
		-\frac{1}{d-2 R}+\frac{\mu_r \eta}{2\left(\mu_r d^{2}+R^{2}\right)} \leq-\frac{1}{\eta+2 R}+\frac{1}{2 \eta}
		$$
		
		Let $x=\frac{d}{\eta}$, $\epsilon=\frac{R}{\eta}$, then

		\begin{equation}\label{g_epsilon}
			g_{\epsilon}(x)=-\frac{1}{x-2 \epsilon}+\frac{1}{2\left(x^{2}+\epsilon^{\frac{3}{2}} R^{\frac{1}{2}} J^{-1}\right)}+\frac{1}{1+2 \epsilon}-\frac{1}{2} \leq 0 \tag{3.3.1}
		\end{equation}

		$\rho(m) \in W_{m}$ implies $\bar{x}\left(\rho_{m}\right) \in \Omega_{m}, \bar{x}\left(\rho_{1-m}\right) \in \Omega_{1-m}$, thus $\frac{1}{2} \leq x \leq \frac{3}{2}$ by the construction of $\Omega_{m}$ and $\Omega_{1-m}$. Let $g_{0}(z)=-\frac{1}{z}+\frac{1}{2 z^{2}}+\frac{1}{2}$, by Lemma \ref{complex convergence} we know $g_{\epsilon}(z)$ converges uniformly to $g_{0}(z)$ as $\epsilon \rightarrow 0$ on $|z| \geq \frac{1}{2}$. One can see $g_0(1)=0$, and $\forall \zeta>0$, $\exists C>0$, such that on $\left[\frac{1}{2}, \frac{3}{2}\right] \setminus [1-\zeta, 1+\zeta]$, we have $g_{0}(x)>C$. Due to Lemma \ref{complex convergence}, we have $g_{\epsilon}(x)>\frac{C}{2}>0$ on $\left[\frac{1}{2}, \frac{3}{2}\right] \setminus [1-\zeta, 1+\zeta]$ for $\epsilon$ small enough (i.e. for $m$ small enough since $\epsilon=\frac{R}{\eta}=\frac{Rm^2(1-m)^2}{J^2}$). Thus (\ref{g_epsilon}) implies $x=\frac{d}{\eta}$ has to satisfy $1-\zeta<\frac{d}{\eta}<1+\zeta$, which proves the proposition.
	}
\end{proof}

\subsection{Existence theorem}\label{subsection6.2-Existence theorem}
After making some modifications to McCann's proof in \cite[Section 6]{McC06}, we proceed by translating $\rho$ such that it enjoys a plane of symmetry $z=0$, and then use the results above to show $\rho$ is supported away from the boundary of $\Omega_{1-m} \cup \Omega_{m}$, then we finalize the proof of existence.

\begin{lemma}[Support Separation from the Boundary]\label{interior}
	Fix $J>0$, $\exists \delta>0, \forall 0<m<\delta$, any constrained minimizer $\rho(m)=\rho_{m}+\rho_{1-m}$ of $E_{J}(\rho)$ on $W_{m}$ will, after a rotation with respect to the $z$-axis and a translation, have support contained in the interior of $\Omega_{1-m} \cup \Omega_{m}$, which means dist(spt $\left.\rho(m), \mathbb{R}^{3} \backslash\left(\Omega_{m} \cup \Omega_{1-m}\right)\right)>0$. Moreover, $\bar{x}\left(\rho_{m}\right)$ and $\bar{x}\left(\rho_{1-m}\right)$ lie in the plane $z=0$, and $\rho(m)$ will also be symmetric about the plane $z=0$ and a decreasing function of $|z|$.
\end{lemma}

\begin{proof}
	{
		\rm
		Take $m$ small enough so that minimizers exist. First, we claim that any minimizer $\rho(m)=\rho_{m}+\rho_{1-m}$ for $E_{J}(\rho)$ on $W_{m}$ may be translated so that both $\bar{x}\left(\rho_{m}\right)$ and $\bar{x}\left(\rho_{1-m}\right)$ lie in the plane $z=0$. In fact, since the $\Omega_{m}$ and $\Omega_{1-m}$ are convex and symmetric about $z=0$, it is enough to know that $\rho(m)$ enjoys a plane of symmetry $z=c$. This follows from a strong rearrangement inequality in Lieb \cite[Lemma 3]{Lie77} and Fubini's Theorem: by the definition of $U(\rho)$ and $I(\rho)$, they will not be changed after the symmetric decreasing rearrangement of $\rho$ along lines parallel to the z-axis; however, since $\left(r^{2}+z^{2}\right)^{-\frac{1}{2}}$ is strictly decreasing as a function of $|z|$, the rearrangement increases $G(\rho, \rho)$ unless $\rho$ is already symmetric decreasing about a plane $z=c$. Note $\Omega_{1 - m}$ and $\Omega_{m}$ are convex, hence $\rho(m)$'s rearrangement is in $W_{m}$. Since $\rho(m)$ minimizes $E_{J}(\rho)$, $G(\rho, \rho)$ cannot be increased after rearrangement, which means $\rho(m)$ enjoys a plane of symmetry $z=c$.
		
		Now, similar to the arguments in the proof of Proposition \ref{estimate of separation}, we take $m$ small enough so that Corollary \ref{bound on the sizes of original densities} and Corollary \ref{bound on the sizes of original stars} provide a bound $R$ such that spt $\rho_{m} \subset B_{R}\left(\bar{x}\left(\rho_{m}\right)\right)$ and spt $\rho_{1-m} \subset B_{R}\left(\bar{x}\left(\rho_{1-m}\right)\right)$ if $\rho(m)=\rho_{m}+\rho_{1-m}$ minimizes $E_{J}(\rho)$ on $W_{m}$. Translate $\rho(m)$ so that its symmetry plane is $z=0$ and let $d:=$ $\left|\bar{x}\left(\rho_{m}\right)-\bar{x}\left(\rho_{1-m}\right)\right|$. Then by construction of $\Omega_{m}$ and $\Omega_{1-m}$ we know that if $d-2 R>\frac{\eta}{2}$ and $d+2 R<\frac{3 \eta}{2}$, a translation and rotation of $\rho(m)$ with respect to $z$-axis yields a minimizer in $W_{m}$ supported away from the boundary of $\Omega_{m} \cup \Omega_{1-m}$. By Proposition \ref{estimate of separation}, this is certainly true when $m$ is sufficiently small.
	}
\end{proof}

Now we finally come to the proofs of our main results Theorem \ref{themA} and Theorem \ref{themA'}

\begin{proof}[Theorem \ref{themA}]
	{
		\rm
		Let $\delta$ as in Lemma \ref{interior} and $\rho(m)$ be the constrained energy minimizer in Theorem \ref{existence of the constrained minimizers in binary case}, then Lemma \ref{interior} shows that spt $\rho({m})$ is compact in the interior of $\Omega_{m} \cup \Omega_{1-m}$, therefore separated from the boundary by a positive distance $\zeta$. Lemma \ref{properties of Wasser.} (ii) shows that if $\kappa \in {R}\left(\mathbb{R}^{3}\right)$ with $W^{\infty}(\rho(m), \kappa)<\eta$, then $\kappa$ in $W_{m}$. Thus $E_{J}(\rho(m)) \leq E_{J}(\kappa)$, and $\rho (m)$ is a Wasserstein $L^\infty$ ($W^\infty$) local energy minimizer on ${R}\left(\mathbb{R}^{3}\right)$ (\textbf{Existence of $W^\infty$ local energy minimizer}).
		By the definition of $W_{m}$ \eqref{admissible class}, we know all the functions in $W_{m}$ have compact support. Thus after translation we can set $\rho (m) \in {R}_{0}\left(\mathbb{R}^{3}\right)$. 
		Since $\rho (m)$ is a $W^\infty$ local energy minimizer on ${R}\left(\mathbb{R}^{3}\right)$, in particular on ${R}_{0}\left(\mathbb{R}^{3}\right)$, Theorem \ref{Properties of LEM} provides a local minimizer $(\rho(m), v)$ of $E(\rho, v)$ subject to the constraint on $J_{z}=J$, with $v(x)=$ $\frac{J}{I(\rho(m))} \hat{e}_{z} \times x$. \cite[Section 2]{McC06} and \cite[Section 2]{Che26G1} show that $(\rho(m), v)$ minimizes $E(\rho, v)$ subject to the constraint on the vector angular momentum as well (\textbf{Part (i)}). 
		Part (ii) comes from Lemma \ref{interior} (\textbf{Part (ii)}).
		By Theorem \ref{Properties of LEM} we know $\rho(m)$ is continuous and satisfies (\ref{EP'}), and $(\rho(m), v)$ satisfies (\ref{EP}) (\textbf{Part (iii)}). Parts (iv) and (v) come from Remark \ref{uniform bound for the norm of stars or scaling stars}, Remark \ref{bound convergence of planets}, Corollary \ref{bound on the sizes of original densities} and Corollary \ref{bound on the sizes of original stars} (\textbf{Part (iv) and Part (v)}).
	}
\end{proof}

\begin{proof}[Theorem \ref{themA'}]
	{
		\rm
		The proof of Theorem \ref{themA'} is essentially the same as the proof of Theorem \ref{themA}. Two things we need to note are (1) if we do not know $\gamma>2$, we may not have the fact that there is a ``uniform bound on the radius of $\rho(m)$'s support for all small $m$" from Corollary \ref{bound on the sizes of original densities}. 
		However, since the radii of $\Omega_{1 - m}$ and $\Omega_{m}$ is $\frac{\eta}{4} \sim m^{-2}$, recall $B(m) = m^{\frac{\gamma - 2}{3\gamma - 4}}$, due to Corollary \ref{bound on the sizes of original densities}, we can still have the fact that $\rho(m)$ is located in the interior of $\Omega_m \cup \Omega_{1-m}$ if $ {\frac{\gamma - 2}{3\gamma - 4}}>-2$, that is $\gamma>\frac{10}{7}$; (2) in order to obtain the energy convergence in Proposition \ref{energy converges to non-rotating minimum}, we hope the term $Cm^{\frac{4\gamma - 6}{3\gamma - 4}}$ in (\ref{scaling density approaches non-rotating density}) satisfies $\frac{4\gamma - 6}{3\gamma - 4}>0$, which means $\gamma>\frac{3}{2}$. Therefore, when $\gamma > \max \{\frac{3}{2}, \frac{10}{7}\}=\frac{3}{2}$, Theorem \ref{themA'} holds.
	}
\end{proof}

\section{Upper Bounds for Distances Between Connected Components of Minimizers}\label{section7-The Maximum Number of Connected Components of Minimizers}

It should be noted that, although we prove the existence of Wasserstein $L^\infty$ local minimizer on $R(\mathbb{R}^3$), which is also a constrained minimizer $\rho(m)=\rho_{1 - m}+\rho_{m}$ on $W_m$ (Theorem \ref{themA} and Theorem \ref{themA'}), in particular, spt $\rho(m)$ has at least two connected components, it is still unclear whether it has exactly two connected components. Neither we in this paper nor McCann \cite[Theorem 6.1 and Corollary 6.2]{McC06} elaborated on or provided proof for the result. If spt $\rho_{1 - m}$ or spt $\rho_{m}$ have multiple connected components, the model may correspond to a system with multiple stars and planets, such as the relationship between the Solar System and the asteroid belt. Moreover, one could interpret it as describing a large and a small galaxy; therefore the existence results retain physical relevance and may inspire further study of the Vlasov–Poisson system via Rein’s reduction method (see e.g. \cite{JS22, Rei03, Rei07}).

Although we have not proven simple connectedness in each subdomain, we can still make the following observations and conclusions that the distance between each two convex connected components of spt $\rho_{1 - m}$ or spt $\rho_{m}$ should not be too large.

We assume spt $\rho_m$ has 2 convex connected components, said $U$ and $V$, and set dist $(U,V)=d$. By Lemma \ref{interior} we can assume $\rho(m)$ is symmetric about the plane $z=0$. We recall $\rho(m)$ is continuous in the whole space (Theorem \ref{themA} or Theorem \ref{themA'}). Also we notice that $E_J$ is invariant with respect to rotation around $z$-axis.Without loss of generality, we can assume there exists $b$, $b'$ such that
\begin{equation}\label{UV in order}
	\sup\limits_{x\in U}x_1 <b<\inf\limits_{y\in V} y_1 \leq \sup\limits_{y \in V} y_1< b' <\inf\limits_{z\in \operatorname{spt} \rho_{1 - m}} z_1
\end{equation}
Then we know $d= \inf\limits_{y\in V} y_1-\sup\limits_{x\in U}x_1$. We denote the mass in $U$ by $m_1\coloneq \int_{U}\rho_{m}\,dx$, then the mass in $V$ is $m_2\coloneq\int_{V}\rho_{m}\,dx=m-m_1$. Notice $\rho_m(x)=\rho_U(x)+\rho_V(x)$, where $\rho_U(x)=\rho_m(x)\cdot\mathbf{1}_U(x)$, $\rho_V(x)=\rho_m(x)\cdot\mathbf{1}_V(x)$, $\mathbf{1}_E$ is the indicator function of set $E$. 
We want to move $\rho_U(x)$ and $\rho_V(x)$ in the following way: given $h>0$, let $h_1>0$ and $h_2>0$ satisfy $h_1+h_2=h$ and $h_1\cdot m_1=h_2 \cdot m_2$. With an abuse of notation we denote $h_1=h_1\cdot \hat{e_1}= (h_1,0,0)^T$, similarly $h_2=(h_2,0,0)^T$, $h=(h,0,0)^T$. Then we define $\rho_m^{[h]}=\rho_U(x-h_1)+\rho_V(x+h_2)$. 

Easy to see when $h$ is small enough, the connected components of spt $\rho_m^{[h]}$ are $U_h=U+h_1$ and $V_h=V-h_2$, and we still have $\sup\limits_{x\in U_h} x_1 <b<\inf \limits_{y\in V_h} y_1$. We can write $\rho_m^{[h]}=\rho_{U}^{\lbrack h\rbrack} + \rho_{V}^{\lbrack h\rbrack}$ where  $\rho_{U}^{\lbrack h\rbrack}=\rho_m^{[h]}\cdot\mathbf{1}_{U_h}$ and $\rho_{V}^{\lbrack h\rbrack}=\rho_m^{[h]}\cdot\mathbf{1}_{V_h}$. We notice the centers of mass of $\rho_m$ and $\rho_m^{[h]}$ are the same. 

Let $\rho^{[h]}(m)=\rho_m^{[h]}+\rho_{1-m}$, then we will show if $d$ is too large, then $E_J(\rho^{[h]}(m))<E_J(\rho(m))$.

In fact, 
\begin{equation}\label{delta E_J}
	E_{J}\left( {\rho(m)} \right) - E_{J}\left( {\rho^{\lbrack h\rbrack}(m)} \right) = T_{J}\left( {\rho(m)} \right) - T_{J}\left( {\rho^{\lbrack h\rbrack}(m)} \right) - \frac{1}{2}\left( {G\left( {\rho(m),\rho(m)} \right) - G\left( {\rho^{\lbrack h\rbrack}(m),\rho^{\lbrack h\rbrack}(m)} \right)} \right)
\end{equation}

We apply Lemma \ref{expansion of MoI} (Expansion of Moment of Inertia) twice and obtain
\begingroup
\small
\begin{equation*}
	\begin{split}
		I\left( {\rho^{\lbrack h\rbrack}(m)} \right) - I\left( {\rho(m)} \right) = I\left( \rho_{m}^{\lbrack h\rbrack} \right) - I\left( \rho_{m} \right) = \frac{m_{1}m_{2}}{m_{1} + m_{2}}r^{2}\left( {\bar{x}\left( \rho_{U}^{\lbrack h\rbrack} \right) - \bar{x}\left( \rho_{V}^{\lbrack h\rbrack} \right)} \right) - \frac{m_{1}m_{2}}{m_{1} + m_{2}}r^{2}\left( {\bar{x}\left( \rho_{U} \right) - \bar{x}\left( \rho_{V} \right)} \right)
	\end{split}
\end{equation*}
\endgroup
Notice ${\bar{x}}\left( \rho_{V}^{\lbrack h\rbrack} \right) - {\bar{x}}\left( \rho_{U}^{\lbrack h\rbrack} \right) = {\bar{x}}\left( \rho_{V} \right) - {\bar{x}}\left( \rho_{U} \right) - h$, thus we have 
\begingroup
\small
\begin{equation*}
	\begin{split}
		I\left( {\rho^{\lbrack h\rbrack}(m)} \right) - I\left( {\rho(m)} \right)&= \frac{m_{1}m_{2}}{m}\left( {- 2h\left( {{\bar{x}}_{1}\left( \rho_{V} \right) - {\bar{x}}_{1}\left( \rho_{U} \right)} \right) + h^{2}} \right)\\
		&= \frac{1}{m}{(- 2h)\iint{\left( {y_{1} - x_{1}} \right)\rho_{U}(x)\rho_{V}(y)\,dx\,dy}} + \frac{m_{1}m_{2}}{m}h^{2} 
	\end{split}
\end{equation*}
\endgroup
Therefore, we have
\begin{equation}\label{delta T_J}
	\begin{split}		
		&\quad T_{J}\left( {\rho(m)} \right) - T_{J}\left( {\rho^{\lbrack h\rbrack}(m)} \right) \\
		&= \frac{J^{2}}{2I\left( {\rho(m)} \right)} - \frac{J^{2}}{2I\left( {\rho^{\lbrack h\rbrack}(m)} \right)} \\
		&= \frac{J^{2}\left( {I\left( {\rho^{\lbrack h\rbrack}(m)} \right) - I\left( {\rho(m)} \right)} \right)}{2I\left( {\rho(m)} \right)I\left( {\rho^{\lbrack h\rbrack}(m)} \right)} \\
		&= \frac{J^{2}}{2I\left( {\rho(m)} \right)I\left( {\rho^{\lbrack h\rbrack}(m)} \right)}\frac{1}{m}\cdot {\left( {{(- 2h)\iint_{\{{\rho_{U}(x) > 0,\rho_{V}(y) > 0}\}}{\left( {y_{1} - x_{1}} \right)\rho_{U}(x)\rho_{V}(y)\,dx\,dy}} + m_{1}m_{2}h^{2}} \right)}
	\end{split}
\end{equation}
On the other hand, 
\begin{equation}\label{delta G}
	G\left( {\rho(m),\rho(m)} \right) - G\left( {\rho^{\lbrack h\rbrack}(m),\rho^{\lbrack h\rbrack}(m)} \right) = G\left( {\rho_{m},\rho_{m}} \right) - G\left( {\rho_{m}^{\lbrack h\rbrack},\rho_{m}^{\lbrack h\rbrack}} \right) + 2G\left( \rho_{1 - m},\rho_{m} - \rho_{m}^{\lbrack h\rbrack} \right)
\end{equation}
Let $h>0$ small enough, such that if $\rho_U (x)>0$ and $\rho_V (y)>0$ then $y_1-x_1-h>0$, by Taylor expansion we have $\frac{1}{\left| {x - y} \right|} - \frac{1}{\left| {x - y + h} \right|} = \frac{- \left( {y_{1} - x_{1}} \right)h}{\left| {x - y} \right|^{3}} + O\left( h^{2} \right)$, therefore,

\begin{equation}\label{delta G without star}
	\begin{split}
		G \left( {\rho_{m},\rho_{m}} \right) - G\left( {\rho_{m}^{\lbrack h\rbrack},\rho_{m}^{\lbrack h\rbrack}} \right) &= G\left( {\rho_{U} + \rho_{V},\rho_{U} + \rho_{V}} \right) - G\left( {\rho_{U}^{\lbrack h\rbrack} + \rho_{V}^{\lbrack h\rbrack},\rho_{U}^{\lbrack h\rbrack} + \rho_{V}^{\lbrack h\rbrack}} \right) \\
		&= 2G\left( {\rho_{U},\rho_{V}} \right) - 2G\left( {\rho_{U}^{\lbrack h\rbrack},\rho_{V}^{\lbrack h\rbrack}} \right)\\
		&= 2{\iint_{\mathbb{R}^3\times \mathbb{R}^3}{\frac{\rho_{U}(x)\rho_{V}(y) - \rho_{U}\left( {x - h_{1}} \right)\rho_{V}\left( {y + h_{2}} \right)}{\left| {x - y} \right|}\,dx\,dy}}\\
		&= 2{\iint_{\mathbb{R}^3\times \mathbb{R}^3}{\frac{\rho_{U}(x)\rho_{V}(y)}{\left| {x - y} \right|} - \frac{\rho_{U}(x)\rho_{V}(y)}{\left| {x - y + h} \right|}\,dx\,dy}} \\
		&= 2{\iint_{\{{\rho_{U}(x) > 0,\rho_{V}(y) > 0}\}}{\rho_{U}(x)\rho_{V}(y)\left( {\frac{- \left( {y_{1} - x_{1}} \right)h}{\left| {x - y} \right|^{3}} + O\left( h^{2} \right)} \right)\,dx\,dy}} \\
		&= 2h{\iint_{\{{\rho_{U}(x) > 0,\rho_{V}(y) > 0}\}}{\rho_{U}(x)\rho_{V}(y)\frac{\left(- \left( {y_{1} - x_{1}} \right)\right)}{\left| {x - y} \right|^{3}}\,dx\,dy}} + 2m_{1}m_{2}O\left( h^{2} \right)\\
		&\leq - \frac{2h}{R^{3}(m)}{\iint_{\{{\rho_{U}{(x)} > 0,\rho_{V}{(y)} > 0}\}}{\rho_{U}(x)\rho_{V}(y)\left( {y_{1} - x_{1}} \right)\,dx\,dy}} + 2m_{1}m_{2}O\left( h^{2} \right)
	\end{split}
\end{equation}
where $R(m)=B(m)\cdot R(J)$ given in Corollary \ref{bound on the sizes of original densities}. Similarly, we can show when $h$ is small enough,
\begin{small}
	\begin{equation*}
		\begin{split}
			&G\left( {\rho_{1 - m},\rho_{m} - \rho_{m}^{\lbrack h\rbrack}} \right) \\
			&= {\iint_{\mathbb{R}^3\times \mathbb{R}^3}{\frac{\rho_{1 - m}(x)\left( {\rho_{U}(y) + \rho_{V}(y)} \right)}{\left| {x - y} \right|}\,dx\,dy}} - {\iint_{\mathbb{R}^3\times \mathbb{R}^3}{\frac{\rho_{1 - m}(x)\left( {\rho_{U}\left( {y - h_{1}} \right) + \rho_{V}\left( {y + h_{2}} \right)} \right)}{\left| {x - y} \right|}\,dx\,dy}} \\
			&= {\iint_{\mathbb{R}^3\times \mathbb{R}^3}{\frac{\rho_{1 - m}(x)\left( {\rho_{U}(y) - \rho_{U}\left( {y - h_{1}} \right)} \right)}{\left| {x - y} \right|}\,dx\,dy}} + {\iint_{\mathbb{R}^3\times \mathbb{R}^3}{\frac{\rho_{1 - m}(x)\left( {\rho_{V}(y) - \rho_{V}\left( {y + h_{2}} \right)} \right)}{\left| {x - y} \right|}\,dx\,dy}} \\
			&= {\iint_{\mathbb{R}^3\times \mathbb{R}^3}{\rho_{1 - m}(x)\rho_{U}(y)\left( {\frac{1}{\left| {x - y} \right|} - \frac{1}{\left| {x - y - h_{1}} \right|}} \right)\,dx\,dy}} + {\iint_{\mathbb{R}^3\times \mathbb{R}^3}{\rho_{1 - m}(x)\rho_{V}(y)\left( {\frac{1}{\left| {x - y} \right|} - \frac{1}{\left| {x - y + h_{2}} \right|}} \right)\,dx\,dy}} \\
			&\leq {\iint_{\mathbb{R}^3\times \mathbb{R}^3}{\rho_{1 - m}(x)\rho_{U}(y)\left( \frac{h_{1}}{\left| {x - y} \right|\left| {x - y - h_{1}} \right|} \right)\,dx\,dy}} + {\iint_{\mathbb{R}^3\times \mathbb{R}^3}{\rho_{1 - m}(x)\rho_{V}(y)\left( \frac{h_{2}}{\left| {x - y} \right|\left| {x - y + h_{2}} \right|} \right)\,dx\,dy}}
		\end{split}
	\end{equation*}
\end{small}
By the construction of $\Omega_{1-m}$ and $\Omega_{m}$ (\ref{domains}) we know $\rho_{1-m} (x)>0$, $\rho_m (y)>0$ implies $\frac{J^{2}}{2m^{2}\left( {1 - m} \right)^{2}} \leq \left| {x - y} \right| \leq \frac{3J^{2}}{2m^{2}\left( {1 - m} \right)^{2}}$, and when $h$ is small enough, $\frac{J^{2}}{4m^{2}\left( {1 - m} \right)^{2}} < \left| {x - y - h_{1}} \right| < \frac{2J^{2}}{m^{2}\left( {1 - m} \right)^{2}}$ and $\frac{J^{2}}{4m^{2}\left( {1 - m} \right)^{2}} < \left| {x - y - h_{2}} \right| < \frac{2J^{2}}{m^{2}\left( {1 - m} \right)^{2}}$. We also notice that by construction $${\int_{\mathbb{R}^3}{\rho_{U}h_{1}}\,dx} = {\int_{\mathbb{R}^3}{\rho_{V}h_{2}}\,dx} = \frac{m_{1}m_{2}h}{m_{1} + m_{2}} = \frac{m_{1}m_{2}h}{m}$$
Then
$$
G\left( {\rho_{1 - m},\rho_{m} - \rho_{m}^{\lbrack h\rbrack}} \right) \leq \frac{16m_{1}m_{2}m^{3}\left( {1 - m} \right)^{5}}{J^{4}}h
$$

Similarly, we have
$$
- G\left( {\rho_{1 - m},\rho_{m} - \rho_{m}^{\lbrack h\rbrack}} \right) = G\left( {\rho_{1 - m},\rho_{m}^{\lbrack h\rbrack} - \rho_{m}} \right) \leq \frac{16m_{1}m_{2}m^{3}\left( {1 - m} \right)^{5}}{J^{4}}h
$$

Thus,
\begin{equation}\label{delta G with star}
	\left| {G\left( {\rho_{1 - m},\rho_{m} - \rho_{m}^{\lbrack h\rbrack}} \right)} \right| \leq \frac{16m_{1}m_{2}m^{3}\left( {1 - m} \right)^{5}}{J^{4}}h
\end{equation}

Similar to the arguments in Lemma \ref{moment of inertia increasing rate as mass decreases}, we know $\frac{J^{2}}{I\left( {\rho(m)} \right)I\left( {\rho^{\lbrack h\rbrack}(m)} \right)}\frac{1}{m} \leq Cm^5$ when $m$ is small enough. Therefore, similar to Corollary \ref{bound on the sizes of original densities} and due to Definition \ref{scaling density for planet} which gives coefficient $B(m)=m^{\frac{\gamma - 2}{3\gamma - 4}}$, we collect (\ref{delta E_J}) (\ref{delta T_J}) (\ref{delta G}) (\ref{delta G without star}) (\ref{delta G with star}) and have (recall we assume $\gamma>\frac{3}{2}$)

\begin{equation}\label{delta E_J 2}
	\begin{split}
		&E_{J}\left( {\rho(m)} \right) - E_{J}\left( {\rho^{\lbrack h\rbrack}(m)} \right) \\
		&\geq \left( {\frac{1}{R^{3}(m)} - \frac{J^{2}}{I\left( {\rho(m)} \right)I\left( {\rho^{\lbrack h\rbrack}(m)} \right)}\frac{1}{m}} \right)\cdot h \cdot{\iint_{\{{\rho_{U}{(x)} > 0,\rho_{V}{(y)} > 0}\}}{\rho_{U}(x)\rho_{V}(y)\left( {y_{1} - x_{1}} \right)\,dx\,dy}} \\
		&\quad + O\left( h^{2} \right) - \frac{16m_{1}m_{2}m^{3}\left( {1 - m} \right)^{5}}{J^{4}}h \\
		&\geq \left( {\frac{dm_{1}m_{2}}{2R^{3}(m)} - \frac{16m_{1}m_{2}m^{3}\left( {1 - m} \right)^{5}}{J^{4}}} \right)h + O\left( h^{2} \right)
	\end{split}
\end{equation}

Therefore, when $d > \frac{32m^{3}\left( {1 - m} \right)^{5}R^{3}(m)}{J^{4}} = \frac{32\left( {1 - m} \right)^{5}R^{3}(J)}{J^{4}}m^{\frac{12\gamma - 18}{3\gamma - 4}}$, we know for $h$ small enough,
$$E_{J}\left( {\rho(m)} \right) - E_{J}\left( {\rho^{\lbrack h\rbrack}(m)} \right) > 0$$ 
That is, if $\rho(m)$ is a constrained minimizer, then $d\leq Cm^{\frac{12\gamma - 18}{3\gamma - 4}}$ for some $C$ and all small $m$. More precisely, we have the following proposition:

\begin{proposition}\label{components distance for planet}
	Given polytropic law $P(\rho)=K\rho^\gamma$ indexed by $\gamma>\frac{3}{2}$ and a constrained minimizer $\rho (m)=\rho_{1 - m}+\rho_{m}$ mentioned in Theorem \ref{themA'}. There is a $C>0$ and a $\delta>0$, such that for all $m\in (0,\delta)$, if spt $\rho_{m}$ has 2 convex connected components, then the distance between these 2 components is less than $\eta$, where $\eta=Cm^{\frac{12\gamma - 18}{3\gamma - 4}}$.  
\end{proposition}

\begin{proof}
	{
		\rm
      See the arguments above.
	}
\end{proof}

\begin{remark}
	We put the condition $\gamma > \frac{3}{2}$ so that we can apply Theorem \ref{themA'}, but in fact the estimate (\ref{delta E_J 2}) actually holds even for smaller values of $\gamma$. 
\end{remark}

Actually, a similar result can be shown for stars' densities.

\begin{corollary}\label{components distance for star}
	Given polytropic law $P(\rho)=K\rho^\gamma$ indexed by $\gamma>\frac{3}{2}$ and a constrained minimizer $\rho (m)=\rho_{1 - m}+\rho_{m}$ mentioned in Theorem \ref{themA'}. There is a $C>0$ and a $\delta>0$, such that for all $m\in (0,\delta)$, if spt $\rho_{1-m}$ has 2 convex connected components,  then the distance between these 2 components is less than $\eta$, where $\eta=Cm^{5}$.
\end{corollary}

\begin{proof}
	{
		\rm
		The proof is essentially the same as the proof above. Suppose spt $\rho_{1 - m}$ has two convex connected components $U$ and $V$, after moving $U$ and $V$ similarly, we have
		\begin{equation*}
		\begin{split}
			&\quad E_{J}\left( {\rho(m)} \right) - E_{J}\left( {\rho^{\lbrack h\rbrack}(m)} \right) \\
			&\geq \left( {\frac{1}{R^{3}\left( {1 - m} \right)} - \frac{J^{2}}{I\left( {\rho(m)} \right)I\left( {\rho^{\lbrack h\rbrack}(m)} \right)}\frac{1}{1 - m}} \right){\iint_{\{{\rho_{U}{(x)} > 0,\rho_{V}{(y)} > 0}\}}{\rho_{U}(x)\rho_{V}(y)\left( {y_{1} - x_{1}} \right)h\,dx\,dy}} \\
			&\quad + O\left( h^{2} \right) - \frac{16m_{1}m_{2}m^{5}\left( {1 - m} \right)^{3}}{J^{4}}h \\
			&\geq \left( {\frac{dm_{1}m_{2}}{2R^{3}\left( {1 - m} \right)} - \frac{16m_{1}m_{2}m^{5}\left( {1 - m} \right)^{3}}{J^{4}}} \right)h + O\left( h^{2} \right)
		\end{split}
	\end{equation*}
		
		We just note that for star density, $R(1-m)$ will be bounded by $R(J)$ as $m$ goes to 0 ($R(J)$ is given in Corollary \ref{bound on the sizes of original stars}). Thus we know $\frac{1}{R^{3}\left( {1 - m} \right)} \geq \frac{2J^{2}}{I\left( {\rho(m)} \right)I\left( {\rho^{\lbrack h\rbrack}(m)} \right)}\frac{1}{1 - m}$ for $m$ small enough. Moreover, when $d > \frac{32m^{5}\left( {1 - m} \right)^{3}R^{3}(1-m)}{J^{4}}$, we know $
		E_{J}\left( {\rho(m)} \right) - E_{J}\left( {\rho^{\lbrack h\rbrack}(m)} \right) > 0$ for $h$ small enough. Thus we can apply the method of contradiction to prove this corollary.
	}
\end{proof}

\begin{remark}
	We have estimated distance between (possible) two connected components. For cases with more than two connected components, say $U_1, U_1, U_3, ...$ (could be uncountable many), if they satisfy certain shape configurations — for example $U_i$ and $V_i$ satisfy (\ref{UV in order}) up to rotation, where $V_i = {\bigcup\limits_{\alpha \neq i}U_{\alpha}}$ — we can then move $U_i$ and $V_i$ similarly as above and analyze the distance relationships between them. Furthermore, we can also use Theorem \ref{themA'} (ii) to rule out some configurations of connected components. For example, the case in which an outer shell encloses an inner solid sphere is impossible for $\rho_m$ or $\rho_{1 - m}$.
\end{remark}

\begin{remark}
	It should be noted that the troublesome term in (\ref{delta E_J 2}) preventing us from further reducing the distance between connected components $U$ and $V$ is the term $\frac{16m_{1}m_{2}m^{3}\left( {1 - m} \right)^{5}}{J^{4}}$, which comes from (\ref{delta G with star}). In other words, it may become possible to prove simple connectedness in each subdomain by moving and gluing the relevant connected components, if one can show that $G\left( {\rho_{m}^{\lbrack h\rbrack},\rho_{m}^{\lbrack h\rbrack}} \right) - G\left( {\rho_{m},\rho_{m}} \right)$ can absorb $\left| {G\left( {\rho_{1 - m},\rho_{m} - \rho_{m}^{\lbrack h\rbrack}} \right)} \right|$. For instance, this could be achieved by proving $G\left( {\rho_{m}^{\lbrack h\rbrack},\rho_{m}^{\lbrack h\rbrack}} \right) - G\left( {\rho_{m},\rho_{m}} \right) > 10\left| {G\left( {\rho_{1 - m},\rho_{m} - \rho_{m}^{\lbrack h\rbrack}} \right)} \right|$ for all small $m$ (that is $m\in (0,\delta)$ for some $\delta>0$) and for all positive distance $d$ between two connected components $U$ and $V$.
\end{remark}

Although we have not proved here that the simple connectedness in each subdomain when $J>0$, we recall when $J=0$ (non-rotating case), the support of any minimizer $\sigma_m$ in Theorem \ref{non-rotating} is truly simply connected (see Remark \ref{simply connected for non-rotating case}). This leads us to the following conjecture:

\begin{conjecture}\label{conjecture 2 connected components}
	The support of $\rho(m)$ in Theorem \ref{themA'} should have exactly 2 connected components when $m$ is small enough.
\end{conjecture}

\begin{remark}
	Results on simple connectedness obtained in other settings may also be consulted; see, for example, \cite{JM19, SW17, SW19}.
\end{remark}


	\section*{Appendix}\label{Appendix}
	\appendix

	\section{Properties of Sobolev Spaces}\label{sectionA-properties of sobolev spaces}
	Here we recall two propositions about the potential $V_\rho(x) = \int_{\mathbb{R}^3} \frac{\rho(y)}{|x-y|} \,dy = \left( \frac{1}{|\cdot|} * \rho \right)(x)$.
	\begin{proposition}[Hardy-Littlewood-Sobolev Inequality {\cite[Theorem 1.7]{BCD11}}] \label{HLSI}
		Let $1<p, r<\infty$ and $0<\alpha<n$ be such that $\frac{1}{p}+\frac{\alpha}{n}=\frac{1}{r}+1$. $\exists C_{p, \alpha, n}>0$, such that
		\begin{equation} \label{BoV}
			\left\||\cdot|^{-\alpha} * f\right\|_{L^{r}\left(\mathbb{R}^{n}\right)} \leq C_{p, \alpha, n}\|f\|_{L^{p}\left(\mathbb{R}^{n}\right)}
		\end{equation}
	\end{proposition}
	
	\begin{proposition}[Differentiability of Potential {\cite[Section 3]{AB71}} {\cite[Appendix B]{Che26G1}}] \label{diff. of poten.}
		If $\rho \in L^{1}\left(\mathbb{R}^{3}\right) \cap L^{p}\left(\mathbb{R}^{3}\right)$ for some $p>3$, then $V_{\rho} \in$ $W^{1, \infty}\left(\mathbb{R}^{3}\right)$ is continuously differentiable and the weak derivative coincides with the classical one for all $x \in \mathbb{R}^{3}$.
	\end{proposition}

%
%
%
%

%

\addcontentsline{toc}{section}{Acknowledgments}
\section*{Acknowledgments}\label{section-acknowledgments}
	The author is partially supported by the National Science Foundation grant DMS-2308208. This work was primarily carried out during the author’s Master’s studies at the University of Bonn. The author thanks Juan Velázquez and Dimitri Cobb for their continued advice and support since the author’s time in Bonn, as well as Christof Sparber and Mimi Dai for their comments and support during the author's Ph.D. studies at the University of Illinois Chicago. Thanks also to Lorenzo Pompili, Shao Liu, Xiaopeng Cheng, Bernhard Kepka, Daniel Sánchez Simón del Pino, Han Cao for discussions, and to Théophile Dolmaire and other instructors. The author is grateful to his parents.

\bibliographystyle{abbrv}
\addcontentsline{toc}{section}{References}
\bibliography{references}

@misc{Che26G1,
	title={Gradient Existence and Energy Finiteness of Local Minimizers in the {W}asserstein $L^\infty$ Topology for Binary-Star Systems}, 
	author={Hangsheng Chen},
	year={2026},
	eprint={2602.01678},
	archivePrefix={arXiv},
	primaryClass={math.AP},
	url={https://arxiv.org/abs/2602.01678}, 
}

@misc{Che26R2,
	title={Revisiting Non-Rotating Star Models: Classical Existence and Uniqueness Theory and {S}caling Relations}, 
	author={Hangsheng Chen},
	year={2026},
	eprint={2602.02631},
	archivePrefix={arXiv},
	primaryClass={math.AP},
	url={https://arxiv.org/abs/2602.02631}, 
}

@misc{CVD24,
	language = {eng},
	publisher = {Rheinische Friedrich-Wilhelms-Universität Bonn},
	title = {Existence for stable rotating star-planet systems},
	address = {Bonn},
	author = {Chen, Hangsheng and Velázquez, J. J. L. and Cobb, Dimitri and Rheinische Friedrich-Wilhelms-Universität Bonn Begründer eines Werks},
	year = {2024},
}

@incollection{Rei07,
	title = {Chapter 5 - Collisionless Kinetic Equations from Astrophysics – The Vlasov–Poisson System},
	editor = {C.M. Dafermos and E. Feireisl},
	series = {Handbook of Differential Equations: Evolutionary Equations},
	publisher = {North-Holland},
	volume = {3},
	pages = {383-476},
	year = {2007},
	issn = {1874-5717},
	doi = {https://doi.org/10.1016/S1874-5717(07)80008-9},
	url = {https://www.sciencedirect.com/science/article/pii/S1874571707800089},
	author = {Gerhard Rein}
}

@misc{AKV23,
	title={Rotating solutions to the incompressible Euler-Poisson equation with external particle}, 
	author={Diego Alonso-Orán and Bernhard Kepka and Juan J. L. Velázquez},
	year={2023},
	eprint={2302.01146},
	archivePrefix={arXiv},
	primaryClass={math.AP}
}

@Article{Rei03,
	author={Rein, Gerhard},
	title={Non-Linear Stability of Gaseous Stars},
	journal={Archive for Rational Mechanics and Analysis},
	year={2003},
	month={Jun},
	day={01},
	volume={168},
	number={2},
	pages={115-130},
	issn={1432-0673},
	doi={10.1007/s00205-003-0260-y},
	url={https://doi.org/10.1007/s00205-003-0260-y}
}

@Article{SW19,
	author={Strauss, Walter A.
	and Wu, Yilun},
	title={Rapidly Rotating Stars},
	journal={Communications in Mathematical Physics},
	year={2019},
	month={Jun},
	day={01},
	volume={368},
	number={2},
	pages={701-721},
	issn={1432-0916},
	doi={10.1007/s00220-019-03414-7},
	url={https://doi.org/10.1007/s00220-019-03414-7}
}

@article{SW17,
	author = {Strauss, Walter A. and Wu, Yilun},
	title = {Steady States of Rotating Stars and Galaxies},
	journal = {SIAM Journal on Mathematical Analysis},
	volume = {49},
	number = {6},
	pages = {4865-4914},
	year = {2017},
	doi = {10.1137/17M1119391},
	URL = { 
	https://doi.org/10.1137/17M1119391
	},
	eprint = { 
	https://doi.org/10.1137/17M1119391
	}
}

@article{Hei94,
	author = {Heilig, U.},
	title = {On {Lichtenstein's} analysis of rotating newtonian stars},
	journal = {Annales de l'I.H.P. Physique th\'eorique},
	pages = {457--487},
	publisher = {Gauthier-Villars},
	volume = {60},
	number = {4},
	year = {1994},
	mrnumber = {1288588},
	zbl = {0808.35107},
	language = {en},
	url = {http://www.numdam.org/item/AIHPA_1994__60_4_457_0/}
}

@Article{JM17,
	author={Jang, Juhi
	and Makino, Tetu},
	title={On Slowly Rotating Axisymmetric Solutions of the Euler--Poisson Equations},
	journal={Archive for Rational Mechanics and Analysis},
	year={2017},
	month={Aug},
	day={01},
	volume={225},
	number={2},
	pages={873-900},
	issn={1432-0673},
	doi={10.1007/s00205-017-1115-2},
	url={https://doi.org/10.1007/s00205-017-1115-2}
}

@Article{CL94,
	author={Chanillo, Sagun
	and Li, Yan Yan},
	title={On diameters of uniformly rotating stars},
	journal={Communications in Mathematical Physics},
	year={1994},
	month={Dec},
	day={01},
	volume={166},
	number={2},
	pages={417-430},
	issn={1432-0916},
	doi={10.1007/BF02112323},
	url={https://doi.org/10.1007/BF02112323}
}

@article{CF80,
	title = {The shape of axisymmetric rotating fluid},
	journal = {Journal of Functional Analysis},
	volume = {35},
	number = {1},
	pages = {109-142},
	year = {1980},
	issn = {0022-1236},
	doi = {https://doi.org/10.1016/0022-1236(80)90082-8},
	url = {https://www.sciencedirect.com/science/article/pii/0022123680900828},
	author = {Luis A Caffarelli and Avner Friedman}
}

@Article{Auc91,
	author={Auchmuty, Giles},
	title={The global branching of rotating stars},
	journal={Archive for Rational Mechanics and Analysis},
	year={1991},
	month={Jun},
	day={01},
	volume={114},
	number={2},
	pages={179-193},
	issn={1432-0673},
	doi={10.1007/BF00375402},
	url={https://doi.org/10.1007/BF00375402}
}

@article{AB71M,
	author = {J. F. G. Auchmuty and Richard Beals},
	year = {1971},
	month = {04},
	pages = {L79},
	title = {Models of Rotating Stars},
	volume = {165},
	journal = {The Astrophysical Journal},
	doi = {10.1086/180721}
}

@Article{Lic33,
	author={Lichtenstein, Leon},
	title={Untersuchungen {\"u}ber die Gleichgewichtsfiguren rotierender Fl{\"u}ssigkeiten, deren Teilchen einander nach dem Newtonschen Gesetze anziehen},
	journal={Mathematische Zeitschrift},
	year={1933},
	month={Dec},
	day={01},
	volume={36},
	number={1},
	pages={481-562},
	issn={1432-1823},
	doi={10.1007/BF01188634},
	url={https://doi.org/10.1007/BF01188634}
}

@article{McC94,
	title={A convexity theory for interacting gases and equilibrium crystals},
	author={Robert J. McCann},
	year={1994},
	url={https://api.semanticscholar.org/CorpusID:118075048},
	journal = {Ph.D. Thesis, Princeton University},
}

@article{Mor02,
	author = {Frank Morgan},
	title = {The Perfect Shape for a Rotating Rigid Body},
	journal = {Mathematics Magazine},
	volume = {75},
	number = {1},
	pages = {30-32},
	year = {2002},
	publisher = {Taylor & Francis},
	doi = {10.1080/0025570X.2002.11953096},
	URL = { 		https://doi.org/10.1080/0025570X.2002.11953096	},
	eprint = { 		https://doi.org/10.1080/0025570X.2002.11953096 }
}

@article{GS84,
	author = {Clark R. Givens and Rae Michael Shortt},
	title = {{A class of Wasserstein metrics for probability distributions.}},
	volume = {31},
	journal = {Michigan Mathematical Journal},
	number = {2},
	publisher = {University of Michigan, Department of Mathematics},
	pages = {231 -- 240},
	year = {1984},
	doi = {10.1307/mmj/1029003026},
	URL = {https://doi.org/10.1307/mmj/1029003026}
}

@MISC {To20,
	TITLE = {Weak and almost everywhere convergence},
	AUTHOR = {Tomás},
	HOWPUBLISHED = {Mathematics Stack Exchange},
	NOTE = {URL:https://math.stackexchange.com/q/611997 (version: 2020-06-12)},
	EPRINT = {https://math.stackexchange.com/q/611997},
	URL = {https://math.stackexchange.com/q/611997}
}

@book{RF10,
	title={Real Analysis},
	author={Royden, H.L. and Fitzpatrick, P.},
	isbn={9780131437470},
	lccn={2009048692},
	url={https://books.google.de/books?id=0Y5fAAAACAAJ},
	year={2010},
	publisher={Prentice Hall}
}

@book{Dal93,
	title={An introduction to $\Gamma$-convergence},
	author={Dal Maso, Gianni},
	volume={8},
	year={1993},
	series={Progress in Nonlinear Differential Equations and Their Applications},
	publisher={Birkhäuser Boston, MA}
}

@article{Lie77,
	author = {Lieb, Elliott H.},
	title = {Existence and Uniqueness of the Minimizing Solution of {C}hoquard's Nonlinear Equation},
	journal = {Studies in Applied Mathematics},
	volume = {57},
	number = {2},
	pages = {93-105},
	doi = {https://doi.org/10.1002/sapm197757293},
	url = {https://onlinelibrary.wiley.com/doi/abs/10.1002/sapm197757293},
	eprint = {https://onlinelibrary.wiley.com/doi/pdf/10.1002/sapm197757293},
	year = {1977}
}

@book{Bre11,
	title={Functional analysis, Sobolev spaces and partial differential equations},
	author={Brezis, Haim},
	volume={2},
	number={3},
	year={2011},
	publisher={Springer}
}

@book{BCD11,
	title={Fourier Analysis and Nonlinear Partial Differential Equations},
	author={Bahouri, H. and Chemin, J.Y. and Danchin, R.},
	isbn={9783642168307},
	series={Grundlehren der mathematischen Wissenschaften},
	url={https://books.google.de/books?id=CcTnaveQkn0C},
	year={2011},
	publisher={Springer Berlin Heidelberg}
}

@article{AB71,
	title={Variational solutions of some nonlinear free boundary problems},
	author={J. F. G. Auchmuty and Richard Beals},
	journal={Archive for Rational Mechanics and Analysis},
	year={1971},
	volume={43},
	pages={255-271},
	url={https://api.semanticscholar.org/CorpusID:122838332}
}

@article{JM19,
	title = {On rotating axisymmetric solutions of the {E}uler–{P}oisson equations},
	journal = {Journal of Differential Equations},
	volume = {266},
	number = {7},
	pages = {3942-3972},
	year = {2019},
	issn = {0022-0396},
	doi = {https://doi.org/10.1016/j.jde.2018.09.023},
	url = {https://www.sciencedirect.com/science/article/pii/S0022039618305667},
	author = {Juhi Jang and Tetu Makino},
}

@article{JS22,
	place = {Country unknown/Code not available}, 
	title = {On Uniformly Rotating Binary Stars and Galaxies}, 
	url = {https://par.nsf.gov/biblio/10334799}, 
	DOI = {10.1007/s00205-022-01766-4}, 
	abstractNote = {}, 
	journal = {Archive for Rational Mechanics and Analysis}, 
	volume = {244}, 
	number = {2}, 
	author = {Juhi Jang and Jinmyoung Seok}, 
	year = {2022},
}

@article{Li91,
	author={YanYan Li},
	title={On uniformly rotating stars},
	journal={Archive for Rational Mechanics and Analysis},
	year={1991},
	month={Dec},
	day={01},
	volume={115},
	number={4},
	pages={367-393},
	issn={1432-0673},
	doi={10.1007/BF00375280},
	url={https://doi.org/10.1007/BF00375280}
}

@article{LY87,
	author={Lieb, Elliott H.
	and Yau, Horng-Tzer},
	title={The {C}handrasekhar theory of stellar collapse as the limit of quantum mechanics},
	journal={Communications in Mathematical Physics},
	year={1987},
	month={Mar},
	day={01},
	volume={112},
	number={1},
	pages={147-174},
	issn={1432-0916},
	doi={10.1007/BF01217684},
	url={https://doi.org/10.1007/BF01217684}
}

@article{Lio84,
	author = {Lions, P. L.},
	title = {The concentration-compactness principle in the calculus of variations. {The} locally compact case, part 1},
	journal = {Annales de l'I.H.P. Analyse non lin\'eaire},
	pages = {109--145},
	publisher = {Gauthier-Villars},
	volume = {1},
	number = {2},
	year = {1984},
	mrnumber = {778970},
	zbl = {0541.49009},
	language = {en},
	url = {http://www.numdam.org/item/AIHPC_1984__1_2_109_0/}
}

@article{McC06,
	title={STABLE ROTATING BINARY STARS AND FLUID IN A TUBE},
	author={Robert J. McCann},
	year={2006},
	journal={Houston Journal of Mathematics},
	volume={32(2)},
	pages={603–631},
	url={https://api.semanticscholar.org/CorpusID:7123278}
}

\end{document}